\RequirePackage{etex} 
\documentclass[english,monochrome]{amsart}
\usepackage{etex}
\usepackage[T1]{fontenc}
\usepackage{amsmath}
\usepackage{amsthm}
\usepackage{amssymb}
\usepackage{float}
\usepackage{courier}
\usepackage{comment}

\usepackage[margin=8pt,skip=5pt,font=small]{caption}

\usepackage[normalem]{ulem}

\usepackage{hyperref}

\usepackage[ruled,vlined,linesnumbered,algosection,algo2e]{algorithm2e}

\usepackage{inconsolata} 
\usepackage[usenames,dvipsnames]{xcolor} 
\usepackage{listings}[mathescape]

\lstdefinelanguage{julia}
{
  keywordsprefix=\@,
  morekeywords={
    exit,whos,edit,load,is,isa,isequal,typeof,tuple,ntuple,uid,hash,finalizer,convert,promote,
    subtype,typemin,typemax,realmin,realmax,sizeof,eps,promote_type,method_exists,applicable,
    invoke,dlopen,dlsym,system,error,throw,assert,new,Inf,Nan,pi,im,begin,while,for,in,return,
    break,continue,macro,quote,let,if,elseif,else,try,catch,end,bitstype,ccall,do,using,module,
    import,export,importall,baremodule,immutable,local,global,const,Bool,Int,Int8,Int16,Int32,
    Int64,Uint,Uint8,Uint16,Uint32,Uint64,Float32,Float64,Complex64,Complex128,Any,Nothing,None,
    function,type,typealias,abstract
  },
  sensitive=true,
  morecomment=[l]{\#},
  morestring=[b]',
  morestring=[b]" 
}

\usepackage{todonotes}

\usepackage{amsaddr}

\makeatletter
\theoremstyle{plain}
\newtheorem{thm}{\protect\theoremname}
\theoremstyle{remark}

\theoremstyle{plain}
\newtheorem{lem}{\protect\lemmaname}

\newtheorem{defn}{\protect\definitionname}
\newtheorem{prop}{\protect\Propname}

\newtheorem{remark}{Remark}
\usepackage{fullpage}

\makeatother

\usepackage{babel}
\providecommand{\claimname}{Claim}
\providecommand{\theoremname}{Theorem}
\providecommand{\lemmaname}{Lemma}
\providecommand{\Propname}{Proposition}
\providecommand{\corollaryname}{Corollary}
\providecommand{\definitionname}{Definition}

\usepackage{prettyref}

\newrefformat{def}{Definition \ref{#1}}
\newrefformat{thm}{Theorem \ref{#1}}
\newrefformat{prop}{Proposition \ref{#1}}
\newrefformat{lem}{Lemma \ref{#1}}
\newrefformat{cor}{Corollary \ref{#1}}
\newrefformat{sec}{Section \ref{#1}}
\newrefformat{sub}{Subsection \ref{#1}}
\newrefformat{rem}{Remark \ref{#1}}
\newrefformat{alg}{Algorithm \ref{#1}}
\newrefformat{ex}{Example \ref{#1}}
\newrefformat{app}{Appendix \ref{#1}}
\newrefformat{subfig}{Figure \ref{#1}}
\newrefformat{fig}{Figure \ref{#1}}

\newcommand{\srf}{\textrm{SRF}} 
\newcommand{\W}{\Omega}

\renewcommand{\vec}[1]{\boldsymbol{\mathrm{#1}}}
\newcommand{\xv}{\vec{x}}

\newcommand{\xvec}{\vec{x}}
\newcommand{\av}{\vec{\alpha}}

\newcommand{\mm}{\Lambda} 
\newcommand{\discr}{\mathcal{V}}
\newcommand{\qs}{\mathfrak{q}}

\DeclareMathOperator{\diag}{diag}

\usepackage[backend=biber,style=numeric,sorting=nyt]{biblatex}
\begin{document}
\title{On the accuracy of Prony's method for recovery of exponential sums with closely spaced exponents}
\author{Rami Katz}
\address{School of Electrical Engineering\\Tel Aviv University\\Tel Aviv, Israel\\
\;\\
 Department of Industrial Engineering\\University of Trento, Italy}
\author{Nuha Diab \and Dmitry Batenkov}
\address{Department of Applied Mathematics\\School of Mathematical Sciences\\Tel Aviv University\\Tel Aviv, Israel}
\date{\today}
\thanks{DB and ND are supported by Israel Science Foundation Grant 1793/20, and by a collaborative grant from the Volkswagen Foundation.}

\email{ramkatsee@gmail.com, nuhadiab@tauex.tau.ac.il, dbatenkov@tauex.tau.ac.il}

\begin{abstract}
    In this paper we establish accuracy bounds of Prony's method (PM) for recovery of sparse measures from incomplete and noisy frequency measurements, or the so-called problem of super-resolution, when the minimal separation between the points in the support of the measure may be much smaller than the Rayleigh limit. In particular, we show that PM is optimal with respect to the previously established min-max bound for the problem, in the setting when the measurement bandwidth is constant, with the minimal separation going to zero. Our main technical contribution is an accurate analysis of the inter-relations between the different errors in each step of PM, resulting in previously unnoticed cancellations. We also prove that PM is numerically stable in finite-precision arithmetic. We believe our analysis will pave the way to providing accurate analysis of known algorithms for the super-resolution problem in full generality.
\end{abstract}

\keywords{Prony's method; super-resolution; sparse spike deconvolution; exponential analysis}

\maketitle

\section{Introduction}

Consider the problem of finding the parameters $\{\alpha_j,x_j\}_{j=1}^n$ of the exponential sum 
\begin{equation}\label{eq:exp-sum}
    f(s) = \sum_{j=1}^n \alpha_j x_j^s
\end{equation}
from the noisy samples $\{f(s_k)+\epsilon_k\}_{k=0}^N$. This exponential fitting problem appears in a wide range of different settings, such as direction of arrival estimation, parametric spectrum estimation, finite rate of innovation sampling, phase retrieval, as well as Pad\'{e} approximation, Gaussian quadrature, and moment problems, to name a few (see e.g. \cite{auton1981, batenkov2013b,lyubich2004, pereyra2010,stoica1995} and the references therein).

An instance of \eqref{eq:exp-sum} of particular interest occurs when $|x_j|=1$ for each $j=1,\dots,n$. This setting is motivated by the problem of super-resolution (SR) of sparse measures of the form $\mu(t) = \sum_{j=1}^n \alpha_j \delta(t-t_j)$ from the samples of its Fourier transform
\begin{equation}\label{eq:ft-of-measure}
f(s)=f_{\mu}(s)=\int e^{2\pi\imath t s}d\mu(t) = \sum_{j=1}^n \alpha_j e^{2\pi\imath s t_j},
\end{equation}
known approximately in some bandwidth $s\in[-\W,\W]$ \cite{donoho1992a}. An important question in applied harmonic analysis is to develop robust reconstruction procedures to solve the SR problem with best possible accuracy, a question which we consider to  still be open even in one spatial dimension. In this paper we consider only the model \eqref{eq:ft-of-measure}; however, some of our results may be extended to the more general model \eqref{eq:exp-sum} (i.e. for arbitrary $x_j\in\mathbb{C}\setminus\{0\}$). The min-max stability of solving \eqref{eq:ft-of-measure} has recently been established in \cite{batenkov2021b} when two or more nodes $x_j$ nearly collide, such that the minimal separation $\delta$ is much smaller than $1/\Omega$, see \prettyref{thm:minmax} below. However, we are not aware of a tractable algorithm provably attaining these bounds.

The Prony's method (PM) \cite{prony1795} is an explicit algebraic procedure (see \prettyref{alg:classical-prony} below) which provides an exact answer to the exponential fitting problem (for arbitrary $x_j\in\mathbb{C}$) under the assumption of exact data (i.e., in the noiseless regime $\epsilon_k\equiv0$), requiring access to only $2n$ consecutive samples $f(0),f(1),\dots,f(2n-1)$. The main insight by de Prony was that the linear parameters $\{\alpha_j\}_{j=1}^n$ can be eliminated from the equations, reducing the problem of recovering $\{x_j\}_{j=1}^n$ to finding roots of a certain algebraic polynomial (the Prony polynomial). The coefficients $\alpha_j$ are recovered in the second step by solving a Vandermonde-type linear system. In the presence of noise, PM is considered to be suboptimal -- however, to the best of our knowledge, no rigorous analysis of its stability has been available, in particular in the context of the super-resolution problem.

\subsection{Main contributions}
Our main result in this paper is a rigorous proof that \emph{Prony's method is optimal} for the SR problem when the measurement bandwidth $\Omega$ is constant while the minimal node separation satisfies $\delta\to 0$. By analyzing each step of PM and taking care of error propagation, we show that the error amplification factors for both the nodes and amplitudes are asymptotically equivalent to the min-max bounds (i.e., the best achievable reconstruction errors under the worst-case perturbation scenario) of  \prettyref{thm:minmax} under the optimal noise scaling (a.k.a. the threshold SNR). These results are given by \prettyref{thm:node-accuracy-Dima} and \prettyref{thm:coeffs-accuracy-Dima} respectively. In effect, our results provide a generalization of \prettyref{thm:minmax} to the true multi-cluster setting (but still restricted to $\Omega=\operatorname{const}$). As a direct corollary, we also show that PM is numerically stable in finite-precision arithmetic (\prettyref{sec:finite-precision}).

Since PM is a multi-step procedure, in principle the errors from each step might have an adverse effect on the next step (such a phenomenon may occur, for instance, when solving linear systems by the $LU$ decomposition method, where the condition number of the original matrix may be unnecessarily amplified, \cite{higham1996}). Our main technical contribution is an accurate analysis of the inter-relations between the different errors in each step of PM, eventually resulting in previously unnoticed cancellations. For comparison, a ``naive'' estimation by textbook numerical analysis methods provides too pessimistic bounds, as we demonstrate in \prettyref{prop:prony-naive-analysis} and \prettyref{prop:ampl-naive}.

\begin{remark}\label{rem:amplitudes-remark}
    The error inter-relations become especially prominent in the multi-cluster setting, as for example it turns out that if the approximate nodes $\{\tilde{x}_j\}_{j=1}^n$ recovered from the first step of PM are further perturbed in an arbitrary direction, for instance by projecting them back to the unit circle, then the resulting errors in the amplitudes $\alpha_j$ may no longer be optimal (contrary to the single cluster example in \prettyref{prop:ampl-naive}). Cf. \prettyref{sec:numerics}.
\end{remark}

\subsection{Towards optimal SR}\label{sub:dpm-intro}

The full SR problem (in particular when $\Omega\delta$ is small but fixed) is apparently still algorithmically open. We believe our results may have implications for analyzing the high-resolution SR algorithms such as ESPRIT/Matrix Pencil/MUSIC, towards establishing their (non-)optimality. Furthermore, the error analysis for different steps may have implications for implementing these methods, cf. \prettyref{sec:numerics} and also \prettyref{rem:amplitudes-remark}.

Recently we have developed the Decimated Prony's Method (DPM) \cite{decimatedProny} (cf. \prettyref{sec:numerics}), which reduces the full SR problem  to a sequence of small problems indexed by a ``decimation parameter'' $\lambda$, followed by further filtering of the results. In more detail, for each admissible $\lambda$, the spectrum $f$ is sampled at $2n$ equispaced frequencies $\{\lambda k\}_{k=0}^{2n-1}$ and the resulting system is subsequently solved by applying PM. Min-max bounds are attained if one can choose $\lambda=O(\Omega)$. This ``decimation'' approach was first proposed in \cite{batenkov2013a}, and further developed in a number of publications on SR \cite{batenkov2017c, batenkov2018,batenkov2020,batenkov2022}, as well as in the resolution of the Gibbs phenomenon \cite{batenkov2015b}. Decimation was the key idea in \cite{batenkov2021b} for establishing the upper bound on the min-max error $\Lambda$ in \prettyref{thm:minmax} as well. \emph{As a consequence of the results of the present paper {\color{blue} we have rigorously shown in \cite{decimatedProny} that DPM attains the upper bounds on $\Lambda$ in the special case of $\ell=n$, under a mild assumption of genericity (see also  \prettyref{sec:numerics}). We conjecture that DPM in fact attains the upper bounds on $\Lambda$ in the general case.} We leave the rigorous proof of this conjecture to a future work.}


\subsection{Organization of the paper}
 In \prettyref{sec:sr-prony-details} we describe the min-max bounds for SR, and present PM with an initial sub-optimal stability analysis in this context. The main results are formulated in \prettyref{sec:main-results}, and subsequently proved in Sections~\ref{Sec:Prelims},\ref{Sec:PfMainthm} and \ref{Sec:PfAmplitude}, with the more technical proofs delegated to the appendices. \prettyref{sec:finite-precision} is devoted to analyzing the performance of PM in finite-precision arithmetic, while \prettyref{sec:numerics} demonstrates the different theoretical results numerically.  

\subsection{Notation} We utilize the following common notations. For $\zeta\in \mathbb{N}$, $[\zeta]$ denotes the set $\{1,2,\dots,\zeta\}$. Asymptotic inequality $A\lessapprox B,A\gtrapprox B$  ($A\asymp B$) means inequality (resp. equality) up to constants. If not specified otherwise, the constants are assumed to be independent of the minimal separation $\delta$ and the perturbation size $\epsilon$. We will use the notation {\color{blue} $\operatorname{col}\left\{ y_i\right\}_{i=0}^N := \begin{bmatrix}
y_0&\dots & y_N
\end{bmatrix}^T$}, where $\left\{y_i \right\}_{i=0}^N\subseteq \mathbb{C}$  are arbitrary scalars. {\color{blue} $B(z,r)$ denotes the standard ball $\{z'\in\mathbb{C}: |z-z'|\leq r\}$.}

{\color{blue}
\subsection{Acknowledgements} We thank the two anonymous referees for their extremely valuable suggestions, which helped us improve the manuscript.
}

\section{Super-resolution and Prony's method}\label{sec:sr-prony-details}

\subsection{Optimal super-resolution}\label{sub:superres-intro}
 The fundamental limits of SR in the sparse model were investigated in several works in the recent years, starting with the seminal paper \cite{donoho1992a} and further developed in \cite{batenkov2020, batenkov2021b, demanet2015,li2021, liu2022b}. For the purposes of this paper, we shall consider the following min-max accuracy bounds derived in \cite{batenkov2021b}. In what follows, we re-formulate the original bounds in terms of the geometry of the complex nodes $x_j:=e^{2\pi\imath t_j}$ directly, thereby making the notations consistent with \eqref{eq:exp-sum}.



\begin{defn}[Minimax rate] Let $F$ denote a set of signals of interest of the form $\mu=(\av,\xv)$, where $\av=(\alpha_1,\dots,\alpha_n)$ and $\xv=(x_1,\dots,x_n)$, {\color{blue} with $x_1<\dots<x_n$}. Given a signal $\mu\in F$ and a perturbation function $e(s)\in L_{\infty}([-\W,\W])$ with $\|e\|_{\infty}\leq\epsilon$, let $\widetilde{\mu}=\widetilde{\mu}(f_{\mu}+e)$ denote any deterministic algorithm which produces an approximation $(\widetilde{\av},\widetilde{\xv})\in F$. Then the min-max error rates for recovering each node $x_j$ and amplitude $\alpha_j$ are given by
\begin{align*}
  \mm^{\xv,j}(\epsilon, F, \Omega) &= \inf_{\widetilde{\mu}=(\widetilde{\av},\widetilde{\xv})} \ \sup_{\mu=(\av,\xv) \in F} \ \sup_{e:\ \|e\|_{\infty}\le \epsilon} |x_j-\widetilde{x_j}|,\\
  \mm^{\av,j}(\epsilon, F,\Omega) &= \inf_{\widetilde{\mu}=(\widetilde{\av},\widetilde{\xv})} \ \sup_{\mu=(\av,\xv) \in F} \ \sup_{e:\ \|e\|_{\infty}\le \epsilon} |\alpha_j-\widetilde{\alpha_j}|.
\end{align*}
\end{defn}

{\color{blue}
\begin{remark}
    The condition $x_1<\dots<x_n$ is imposed to avoid ambiguity in representing the solution to the SR problem, as the solutions can only be unique up to an arbitrary permutation of the nodes. This condition is further assumed throughout the manuscript without any explicit mention thereof.
\end{remark}
}

As noted already in \cite{donoho1992a}, the SR becomes difficult (and, in some sense, nontrivial) when some of the $n$ nodes $\{x_j\}_{j=1}^n$ may form ``clusters'' of extent much smaller than the Rayleigh-Nyquist limit $1/\Omega$. To make this notion precise, let $\delta$ denote a-priori minimal separation between any two $x_j$ (the definition below is more general than the one used in \cite{batenkov2021b}). 

\begin{defn}[Clustered configuration]\label{def:cluster}
Given $n\geq 2$, the set of nodes $\{x_1,\dots,x_n\}$ is said to form a $(K_x,n,\zeta,\ell_*,\delta,\tau,\eta,T)$-cluster if there exists a partition $\bigcup_{s\in [\zeta]}\mathcal{C}_s=[n]$, with $\mathcal{C}_s\bigcap \mathcal{C}_{s'}=\emptyset$ for $s,s'\in [\zeta],\ s\neq s'$, such that:
\begin{enumerate}
    \item There exists a compact $K_x\subseteq \mathbb{C}$ such that $x_i\in K_x, \ i=1,\dots,n$.
    \item $\operatorname{card}\left(\mathcal{C}_s \right)=\ell_s$ for $s\in [\zeta]$ and $\ell_* := \max_{s\in [\zeta]}\ell_s$;
    \item there exist $\tau>1$ and $0<\delta<1$ such that for any $i,j\in \mathcal{C}_s, \ s\in [\zeta]$
    \begin{equation*}
    \delta \leq \left|x_i-x_j \right|\leq \tau \delta;
    \end{equation*}
    \item there exist $\eta>1$ and $T>\tau \delta$ such that for any $i\in \mathcal{C}_s,\ j\in \mathcal{C}_{s'}$ where $s,s'\in[\zeta],\ s\neq s'$ \begin{equation*}
    T\leq \left|x_i-x_j \right|\leq \eta T.
    \end{equation*}
    \end{enumerate}
    Here $\operatorname{card}\left(A \right)$ stands for the cardinality of the set $A$.
\end{defn}

In the remainder of the paper we will assume that $K_x=\mathbb{S}^1$ is the unit circle (this assumption corresponds to \eqref{eq:ft-of-measure}, which is the model of interest in this paper) and omit $K_x$ from the clustered configuration parameters.
For clustered configurations, in \cite{batenkov2021b} the worst-case bounds for the recovery problem were established as follows. Define the \emph{super-resolution factor} $\srf:=(\Omega\delta)^{-1}$.

\begin{thm}[\cite{batenkov2021b}]\label{thm:minmax}
  Let $F$ denote the set of signals whose node set forms a cluster with $\ell_1=\ell_*$ and $\ell_2=\ell_3=\dots=\ell_p=1$, and $\{|\alpha_j| \}_{j=1}^n$ bounded from below and above. For $\srf:={1\over{\Omega\delta}} \geq O(1)$, and
  $\epsilon \lessapprox (\Omega\delta)^{2\ell_*-1}$:
  \begin{align*}
    \mm^{\xv,j}(\epsilon,F,\Omega) &\asymp
                  \begin{cases}
                    \srf^{2\ell_*-1} \delta \epsilon & x_j \in \mathcal{C}_1, \\
                    {\epsilon\over\Omega} & x_j \notin \mathcal{C}_1,
                  \end{cases} \\
    \mm^{\av,j}(\epsilon,F,\Omega) &\asymp
                  \begin{cases}
                    \srf^{2\ell_*-1}  \epsilon & x_j \in \mathcal{C}_1, \\
                     \epsilon & x_j \notin \mathcal{C}_1.
                  \end{cases}
  \end{align*}
\end{thm}

For discussion of the relations of \prettyref{thm:minmax} to other works on the subject, the reader is referred to \cite[Section 1.4]{batenkov2021b}. Related results are known in the signal processing literature for Gaussian noise model: \cite{lee1992} provides similar bounds in terms of Cramer-Rao bound (in the case of a single cluster), and the expression for the threshold SNR \emph{for detection} was shown to scale like $\srf^{-2}$ for $n=2$ in \cite{stoica1995}, which is also consistent with \cite{shahram2005}.

The upper bound on the minmax error is realized by a non-tractable ``oracle-type'' algorithm, which, given a measurement function $g(s)=f_{\mu}(s)+e(s)$, produces \emph{some} signal parameters $\{\alpha_j',x_j'\}_{j=1}^n$ for which $\max_{s\in[-\Omega,\Omega]} |\sum_{j=1}^n \alpha_j' x_j'^s - g(s)| \leq \epsilon$. We are not aware of any tractable method which provably achieves the upper bound on $\mm$, although some partial results in this direction are known. The ESPRIT algorithm (originally proposed in \cite{roy1989}) was analyzed in \cite{li2020a} in the case of a discrete measurement model, showing that the nodes errors are bounded by $\srf^{2\ell_*-2}\epsilon$ provided that $\epsilon\lessapprox (\Omega\delta)^{4\ell_*-3}/\Omega$. The MUSIC algorithm was partially analyzed in \cite{li2021}, where  perturbation bounds on the noise-space correlation function were established; however this analysis does not imply an effective bound on $\Lambda$.

\subsection{Prony's method}\label{sec:prony-method-setup}

PM reduces the problem to a three-step procedure, which involves solution of two linear systems, in combination with a root-finding step, as described in \prettyref{alg:classical-prony} {\color{blue} and in the following}.

 Denote the (true) nodes and the amplitudes of Prony's problem as $\left\{x_j \right\}_{j=1}^n$ and $\left\{\alpha_j \right\}_{j=1}^n$, respectively. The so-called (unperturbed) monic Prony polynomial is given by
\begin{equation}\label{eq:prony-monic-def}
    p\left(z\right)=\prod_{i=1}^{n}\left(z-x_{i}\right) = z^n+\sum_{j=0}^{n-1}p_{j}z^{j}.
\end{equation}
    
The coefficients of $p(z)$ are obtained by solving the linear system with a Hankel matrix $H_n$:
\begin{align}
    &H_n \cdot \operatorname{col}\left\{ p_i\right\}_{i=0}^{n-1}  = -   \operatorname{col}\left\{m_i \right\}_{i=n}^{2n-1}, \quad
    H_n: = \begin{bmatrix}
m_0 & m_1 & \dots &m_{n-1}\\
\vdots & \vdots & \vdots & \vdots\\
m_{n-1} & m_n & \dots & m_{2n-2}
 \end{bmatrix}.\label{eq:UnpertHanekl}
\end{align}
We assume that the algebraic moments $m_{k}=\sum_{j=1}^{n}\alpha_{j}x_{j}^{k}$ are measured with perturbations (disturbances) of size $\epsilon\geq 0$. The latter give rise to a perturbed Hankel matrix 
\begin{align}
    &\tilde{H}_n: = H_n+\epsilon \mathrm{D}, \quad \mathrm{D} : = \begin{bmatrix}
d_0 & d_1 & \dots &d_{n-1}\\
\vdots & \vdots & \vdots & \vdots\\
d_{n-1} & d_n & \dots & d_{2n-2}
 \end{bmatrix}, \quad d_i = \text{O}(1), \ i\in \left\{0,\dots, 2n-2 \right\}.\label{eq:PertHankel}
\end{align}
The solution of the linear system 
\begin{align}
    &\tilde{H}_n\cdot \operatorname{col}\left\{q_i \right\}_{i=0}^{n-1} = -   \operatorname{col}\left\{m_i+\epsilon d_i \right\}_{i=n}^{2n-1} \label{eq:PertHanekl}
\end{align}
provides the coefficients of a perturbed monic Prony polynomial
\begin{equation}\label{eq:q-monic-def}
q\left(z;\left\{ d_i\right\},\epsilon\right)=\prod_{i=1}^{n}\left(z-\tilde{x}_{i}\right) = z^n+\sum_{j=0}^{n-1}q_{j}z^{j}.
\end{equation}
The roots of $q\left(z;\left\{ d_i\right\},\epsilon\right)$ are then used to obtain the perturbed amplitudes $\left\{\tilde{\alpha}_j \right\}_{j=1}^n$ {\color{blue} by solving a Vandermonde linear system}.

Throughout the paper, we will consider the number of nodes (resp. amplitudes) $n$ to be \emph{fixed}. We suppose that the amplitudes satisfy $\mathfrak{m}_{\alpha}\leq |\alpha_i|\leq \mathfrak{M}_{\alpha}, \ i=1,\dots,n$ for some $0<\mathfrak{m}_{\alpha}<\mathfrak{M}_{\alpha}$.

\subsection{(Apparent) instability of Prony's method}\label{sec:pm-apparent-instability}
PM is generally considered to be suboptimal for solving \eqref{eq:exp-sum} when $N\gg 2n-1$ (see e.g. \cite{kahn1992, vanblaricum1978} and references therein) due to its inability to utilize the additional measurements. It is also somewhat of a ``folk knowledge''  that PM is ``numerically unstable'', usually contributed to the fact that it involves a rootfinding step, while extracting roots is known to be ill-conditioned for root clusters (which is precisely our case of $\srf \gg 1$). While we are not aware of any rigorous numerical analysis of Prony's method in the literature, a rudimentary computation seems to confirm the above claims.

\begin{algorithm2e}[hbt]
  \SetKwInOut{Input}{Input} \SetKwInOut{Output}{Output}
  \Input{Sequence $\{\tilde{m}_k\equiv f_{\mu}(k)+\epsilon_k\},\;k=0,1,\dots,2n-1$}
  \Output{Estimates for the nodes $\{x_j\}$ and amplitudes
    $\{\alpha_j\}$}
    Construct the Hankel matrix
    $$
    \tilde{H}_n = 
    \begin{bmatrix}
\tilde{m}_{0} & \tilde{m}_{1} & \tilde{m}_{2} & \dots & \tilde{m}_{n-1}\\
\vdots\\
\\
\tilde{m}_{n-1} & \tilde{m}_{n} & \tilde{m}_{n+1} & \dots & \tilde{m}_{2n-2}
\end{bmatrix}
    $$ 

    Assuming $\operatorname{det}\tilde{H}_n \neq 0$, solve the linear system
    $$
    \tilde{H}_n \cdot \operatorname{col}\{q_i\}_{i=0}^{n-1} = -\operatorname{col}\{\tilde{m}_i\}_{i=n}^{2n-1}
    $$

	Compute the roots $\{\tilde{x}_j\}$ of the (perturbed) Prony polynomial $q(z) = z^n+\sum_{j=0}^{n-1} q_{j}z^j$ \;
	Construct $\tilde{V}:= \big[\tilde{x}_j^k\big]_{k=0,\dots,n-1}^{j=1,\dots,n}$  and solve the linear system
    $$
    \tilde{V} \cdot \operatorname{col}\{\tilde{\alpha}_i\}_{i=1}^n = \operatorname{col}\{\tilde{m}_i\}_{i=0}^{n-1}
    $$

	\Return the estimated $\tilde{x}_j$ and $\tilde{\alpha}_j$.
  \caption{The Classical Prony's method}
  \label{alg:classical-prony}
\end{algorithm2e}

\begin{prop}\label{prop:prony-naive-analysis}
    {\color{blue} Suppose $\xvec$ forms a single cluster ($\ell=n$) configuration and $\Omega=2n-1$, with the perturbations $\epsilon_k$ in the moments $m_k$ satisfying $|\epsilon_k|\leqslant \epsilon$. Then
    \begin{enumerate}
        \item If $\epsilon \lessapprox \delta^{2\ell-2}$, the coefficients of the Prony polynomial are recovered with accuracy
        
        \begin{equation}\label{eq:prony-poly-coeff-magnitude}
        |p_i-q_i|\lessapprox \delta^{2-2\ell}\epsilon.
        \end{equation}
        
        \item If $\epsilon \lessapprox \delta^{3\ell-3}$, the nodes $\{x_j\}_{j=1}^n$  are recovered  by Prony's method (Algorithm \ref{alg:classical-prony}) with accuracy $|\tilde{x}_j-x_j| \lessapprox \delta^{3-3\ell}\epsilon$.
    \end{enumerate}
    }
\end{prop}

Before embarking on the proof, let us state a well-known fact which can be checked by direct calculation.

\begin{prop}\label{prop:vand-decomp}
    The Hankel matrix $H_n$ admits the Vandermonde factorization $H_n=VCV^{\top}$ where $V=\left[ x_j^k\right]_{k=0,\dots,n-1}^{j=1,\dots,n}$ and $C=\diag\{\alpha_1,\dots,\alpha_n\}$.
\end{prop}

\begin{proof}[Proof of \prettyref{prop:prony-naive-analysis}]
    {\color{blue}
    Denote $\vec{m}=\operatorname{col}\{m_i\}_{i=n}^{2n-1}$, $\vec{\tilde{m}}=\operatorname{col}\{\tilde{m}_i\}_{i=n}^{2n-1}$. The polynomial coefficient vectors are denoted by $\vec{p}=\operatorname{col}\{p_i\}_{i=0}^{n-1}$ and $\vec{q}=\operatorname{col}\{q_i\}_{i=0}^{n-1}$, with $\Delta \vec{p} = \vec{q}-\vec{p}$. 

    We start by proving \eqref{eq:prony-poly-coeff-magnitude}.  Let $\|\cdot\|$ denote any vector norm and the associated induced matrix norm. Notice that $\|\vec{p}\| \asymp 1$ since the exact roots satisfy $|x_j|\equiv 1$. By well-known results from the literature  (see, e.g. \cite[Eq. 4.9]{batenkov2020}), we have $\sigma_{\min}(V) \gtrapprox \delta^{\ell-1}$. Plugging into \prettyref{prop:vand-decomp} we obtain:
    \begin{align*}
    \frac{1}{\|H_n^{-1}\|} &\gtrapprox \sigma_{\min}(H_n) \gtrapprox \sigma_{\min}(V)^2 \mathfrak{m}_{\alpha} \gtrapprox \mathfrak{m}_{\alpha} \delta^{2\ell-2}.
    \end{align*}

    Put $E:=\tilde{H}_n-H_n$ and $\vec{e}:=\vec{\tilde{m}}-\vec{m}$. Define $F:=EH_n^{-1}$. By the above bound, there exists a constant $c_1$ s.t. for all $\epsilon \leq c_1 \mathfrak{m}_{\alpha}\delta^{2\ell-2}$ we have $\|F\| < \frac{1}{2}$. Consequently, for all $\epsilon \leq c_1 \mathfrak{m}_{\alpha}\delta^{2\ell-2}$ we have
    \begin{align*}
        H_n\vec{p} &= \vec{m}\\
        (H_n+E)(\vec{p}+\Delta\vec{p})&=\vec{\tilde{m}}\\
        (H_n+E)\Delta\vec{p} &= \underbrace{\vec{e}-E\vec{p}}_{:=\vec{r}}\\
        \Delta\vec{p}&=(H_n+E)^{-1}\vec{r}= H_n^{-1}(I+F)^{-1}\vec{r}\\
        \|\Delta\vec{p}\| & \lessapprox \|H_n^{-1}\| \|\vec{r}\| \lessapprox \mathfrak{m}_{\alpha}^{-1} \delta^{2-2\ell} \epsilon.
    \end{align*}
    This proves \eqref{eq:prony-poly-coeff-magnitude}. 
    }


    

    {\color{blue}
    To derive the second bound in \prettyref{prop:prony-naive-analysis} we shall utilize Rouche's theorem \cite{conway1978functions}. We will also obtain quantitative estimate on the size of the perturbation $\epsilon$ required to attain the bound. 

    To apply Rouche's theorem, we pick $j=1,\dots,n$ and set $z=z(\theta)=x_j+\rho_* e^{i\theta}$, where $\theta$ is an arbitrary angle. We will seek a condition on $\rho_*$ which ensures that, for all $\theta$, the following inequality holds:
    $$
    |p(z(\theta))| > |q(z(\theta))-p(z(\theta))|.
    $$
    By Rouche's theorem, this would imply that $q(z)$ has a root in the ball $B(x_j,\rho_*)$ and therefore $|x_j-\tilde{x}_j| \leq \rho_*$.

    Due to the explicit form of $p(z)$ as in \prettyref{eq:prony-monic-def}, we have $|p(z(\theta))| \geqslant c_1 \rho_* \delta^{\ell-1}$. On the other hand, using the estimate \eqref{eq:prony-poly-coeff-magnitude} and assuming $\rho_* < 1/3$ (for simplicity only), we have
    \begin{equation}\label{eq:r-magnitude-naive-bound}
    |q(z(\theta))-p(z(\theta))| \leqslant c_2 \delta^{2-2\ell}\epsilon.
    \end{equation}

    Now choose $\rho_*=\frac{2c_2}{c_1}\delta^{3-3\ell}\epsilon$. To ensure the condition $\rho_*<1/3$ we require $\epsilon \leqslant \frac{c_1}{6c_2}\delta^{3\ell-3}$. 

    This completes the proof. \qedhere
    }
\end{proof}

{\color{blue} As the following proposition demonstrates,} apparent difficulties arise also when estimating the errors in recovering the $\alpha_j$'s, even assuming that the nodes have been recovered with optimal accuracy. 

\begin{prop}\label{prop:ampl-naive}
    Let $\vec{m}=\operatorname{col}\{m_i\}_{i=0}^{n-1}$, $\vec{\tilde{m}}=\operatorname{col}\{\tilde{m}_i\}_{i=0}^{n-1}$. Under the assumptions of \prettyref{prop:prony-naive-analysis}, for $\epsilon \lessapprox \delta^{3\ell-3}$, let $\tilde{\xvec}$ be a node vector satisfying $|\tilde{x}_j-x_j|\lessapprox \delta^{2-2\ell}\epsilon$. Let $\tilde{\av}=\operatorname{col}\{\tilde{\alpha}_i\}_{i=1}^{n}$ be the solution of the linear system $\tilde{V} \tilde{\av} = \tilde{\vec{m}}$ where $\tilde{V}$ is the Vandermonde matrix corresponding to $\tilde{\xvec}$, and $\|\tilde{\vec{m}}-\vec{m}\|_{\infty} \leq \epsilon$. Then $|\tilde{\alpha}_j-\alpha_j|\lessapprox \delta^{3-3\ell}\epsilon$ for all $j=1,\dots,n$.
\end{prop}

\begin{proof}
    Denote $\av=\operatorname{col}\{\alpha_i\}_{i=1}^{n}$, and $\Delta\vec{\alpha}=\tilde{\av}-\av$. {\color{blue}
    A well-known bound (see e.g., \cite[eq.(2.4)]{demmel1997}, or any classical numerical analysis textbook) states that for any vector norm $\|\cdot\|$ and associated induced matrix norm, we have
    \begin{equation}\label{eq:classical-error-lin}    
    \frac{\|\Delta \vec{\alpha}\|}{\|\vec{\alpha}\|} \leq \frac{\kappa(V)}{1-\kappa(V) \frac{\|\Delta V\|}{\|V\|}}  \left( \frac{\|\Delta V\|}{\|V\|} + \frac{\|\Delta \vec{m}\|}{\|\vec{m}\|}\right),\qquad \Delta V = \tilde{V}-V,\;\Delta\vec{m} = \tilde{\vec{m}}-\vec{m},
    \end{equation}
    where $\kappa(A) = \|A\|\|A^{-1}\|$.
    By \prettyref{prop:vand-decomp} and Gautschi's bound for $\|V^{-1}\|$ \cite{gautschi1974norm} we have $\kappa(V)\lessapprox \max_{i=1,\dots,n}\prod_{j\neq i} |x_i-x_j|^{-1} \lessapprox \delta^{1-\ell}$.
    }
    
     Since $\sigma_{\min}(V)\asymp \delta^{\ell-1}$ (as in the proof of \prettyref{prop:prony-naive-analysis}) and $\|\tilde{V}-V\|_{\infty} \approx \delta^{2-2\ell}\epsilon$, the condition on $\epsilon$ ensures that $\|\Delta V\| \kappa(V)<\frac{1}{2}$, and therefore \eqref{eq:classical-error-lin} yields in this case
    $$
    \frac{\|\Delta\vec{\alpha}\|}{\|\vec{\alpha}\|} \lessapprox \delta^{1-\ell} (\delta^{2-2\ell}\epsilon + \delta^{1-\ell}\epsilon) \lessapprox \delta^{3-3\ell} \epsilon.\qedhere
    $$
\end{proof}

Now let us consider the setting of \prettyref{sub:superres-intro} again, and suppose that $\ell=n$, $\Omega$ is constant and $\delta \rightarrow 0$. The above analysis suggests that Prony's method should not be used for solving \eqref{eq:exp-sum} already in this simplified setting, as the estimates derived in Propositions~\ref{prop:prony-naive-analysis} and \ref{prop:ampl-naive} are clearly suboptimal with respect to the min-max error of Theorem \ref{thm:minmax}. However, the actual numerical performance of Prony's method turns out to be much more accurate. Indeed, a simple numerical experiment (Figure \ref{fig:prony-initial-check}) demonstrates that Prony's method exhibits cancellation of errors in both steps 2 and 3, resulting in errors of the order $\delta^{2-2\ell}\epsilon$ for both the coefficients of the Prony polynomial and the roots themselves. On the other hand, replacing the error vector {\color{blue}$\Delta\vec{p}$} with a random complex vector $\Delta\vec{r}$ such that $|\Delta r_j|=|\Delta p_j|=|q_j-p_j|$, we observe that the corresponding root perturbations are of the order $\delta^{3-3\ell}\epsilon$, {\color{blue} exactly as predicted by \prettyref{prop:prony-naive-analysis}. Thus, a refined analysis, which is aimed at discovering the true error tolerance of Prony's method, should take into account the inter-relations between the errors in different coefficients in step 3.}

The bound in \prettyref{prop:ampl-naive} turns out to be extremely pessimistic as well, as evident from Figure \ref{fig:prony-initial-check} (right panel). Here again the bound is not attained as described, the Vandermonde structure of $V$ (resp. $\tilde{V}$) evidently playing a crucial role with regards to stability. Further experiments suggesting the asymptotic optimality of Prony's method both in terms of stability coefficient, and the threshold SNR, have been recently reported in \cite{decimatedProny} in the clustered geometry. Motivated by the numerical evidence as described, in this paper we close the aforementioned gap and derive the true error tolerance of Prony's method.

\begin{figure}
\begin{center}
\includegraphics[width=0.32\linewidth]{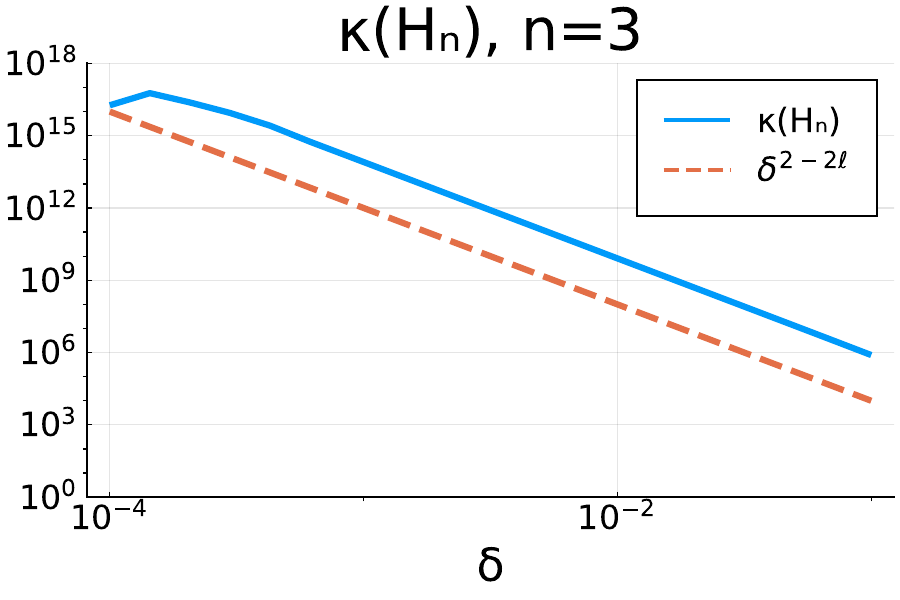}
\includegraphics[width=0.32\linewidth]{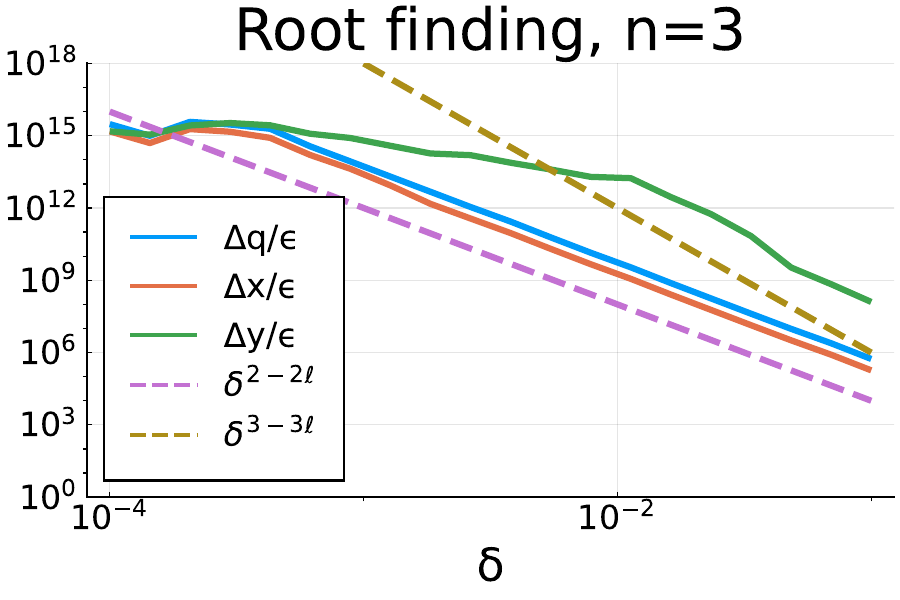}
\includegraphics[width=0.32\linewidth]{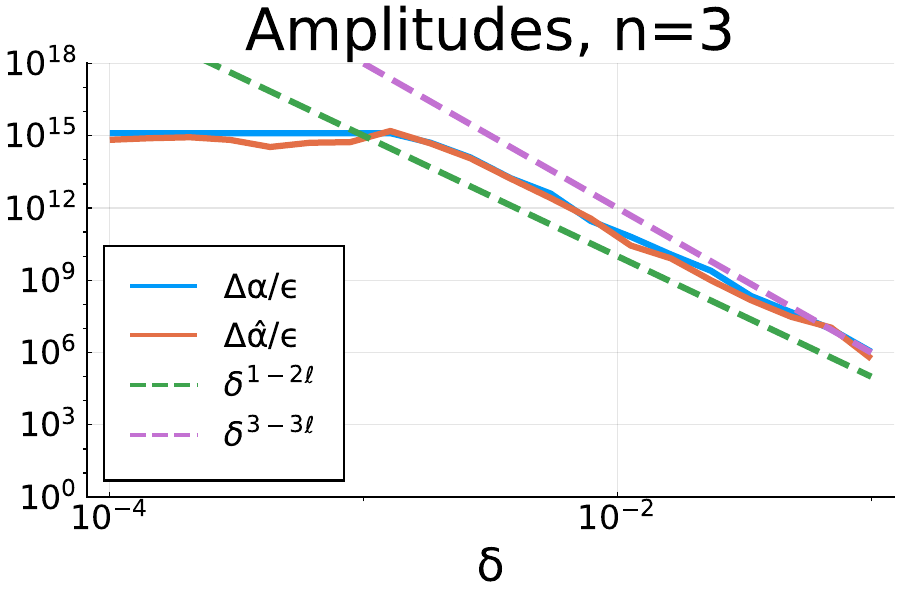}
\end{center}
\caption{Left: the condition number $\kappa(H_n)$ of the Hankel matrix scales as $\delta^{2-2\ell}$. Middle: both $\Delta\vec{q}$ and $|x_j-\tilde{x}_j|$ scale as $\delta^{2-2\ell}\epsilon$, showing that the errors in the coefficients of the Prony polynomial are not independent. For comparison, choosing a random perturbation of $\vec{p}$ with same magnitude as $\Delta\vec{q}$ and computing the roots $y_j$ of the resulting polynomial, we observe that $|y_j-x_j| = O(\delta^{3-3\ell}\epsilon)$, as predicted by \prettyref{prop:prony-naive-analysis}. Right: the amplitude errors committed by Prony's method (blue) and by replacing the recovered nodes with random perturbations (red). Here, in contrast, the bound of \prettyref{prop:ampl-naive} is not attained. \emph{All computations are done in floating-point arithmetic, with $\epsilon=10^{-15}$.}}
\label{fig:prony-initial-check}
\end{figure}

\section{Main results}\label{sec:main-results}
  Consider the Prony's method in Algorithm \ref{alg:classical-prony}, {\color{blue} and employ the assumptions and the notations of \prettyref{sec:prony-method-setup} and \prettyref{sec:pm-apparent-instability}. We further assume that the nodes form a clustered configuration as in Definition \ref{def:cluster}.}
   
 The goal of this work is to derive efficient bounds on the errors $\left\{\left|x_j -\tilde{x}_j \right| \right\}_{j=1}^n$ and $\left\{\left|\alpha_j -\tilde{\alpha}_j \right| \right\}_{j=1}^n$, depending on $\epsilon$ and to derive the condition on $\epsilon$ which ensures the bounds.

In this paper we use the  ``homogeneous'' version given in Algorithm \ref{alg:homo-prony}, which is computationally equivalent to \prettyref{alg:classical-prony}. Let $\bar{p}(z)$ denote the corresponding homogenized version of $p(z)$:

\begin{align}
\bar{p}\left(z\right)=\operatorname{det}\begin{bmatrix}1 & z & z^{2} & \dots & z^{n-1} & z^{n}\\
m_{0} & m_{1} & m_{2} & \dots & m_{n-1} & m_{n}\\
\vdots\\
\\
m_{n-1} & m_{n} & m_{n+1} & \dots & m_{2n-2} & m_{2n-1}
\end{bmatrix}\label{eq:HomogPronyPol}.
\end{align}

\begin{algorithm2e}[hbt]
  \SetKwInOut{Input}{Input} \SetKwInOut{Output}{Output}
  \Input{Sequence $\{\tilde{m}_k\}_{k=0}^{2n-1} $}
  \Output{Estimates for the nodes $\{x_j\}_{j=1}^n$ and amplitudes
    $\{\alpha_j\}_{j=1}^n$}
	Compute the roots $\{\tilde{x}_j\}_{j=1}^n$ of the Prony polynomial
	\begin{equation}\label{eq:qbar-def}
	\bar{q}(z)=\operatorname{det}\begin{bmatrix}1 & z & z^{2} & \dots & z^{n-1} & z^{n}\\
\tilde{m}_{0} & \tilde{m}_{1} & \tilde{m}_{2} & \dots & \tilde{m}_{n-1} & \tilde{m}_{n}\\
\vdots\\
\\
\tilde{m}_{n-1} & \tilde{m}_{n} & \tilde{m}_{n+1} & \dots & \tilde{m}_{2n-2} & \tilde{m}_{2n-1}
\end{bmatrix}.
\end{equation}

Solve the linear system $\tilde{V} \cdot \operatorname{col}\{\tilde{\alpha}_j\}_{j=1}^{n}  = \operatorname{col}\{\tilde{m}_j\}_{j=0}^{n-1}$ \;
\Return the estimated $\tilde{x}_j$ and $\tilde{\alpha}_j$ for $j=1,\dots,n$.
\caption{The Homogeneous Prony's method}
 \label{alg:homo-prony}
\end{algorithm2e}

We now state the main theorems of this work. 

The node errors $\left\{\left|x_i-\tilde{x}_i \right|\right\}_{i=1}^n$, are bounded as follows.

\begin{thm}\label{thm:node-accuracy-Dima}
{\color{blue}Let $\left\{x_i \right\}_{i=1}^n$ form a clustered configuration as in \prettyref{def:cluster}, with largest cluster of size $\ell_*$.} Denote $m_k=\sum_{s=1}^n \alpha_s x_s^k$, and let $\left\{d_s\right\} =O(1)$ be tolerance coefficients. For each $\epsilon$, let $\tilde{x}_i=\tilde{x}_i(\left\{ d_s\right\},\epsilon), \ i=1,\dots,n$ be the exact roots of the Prony polynomial $\bar{q}\left(z;\left\{ d_s\right\},\epsilon\right)$ in \eqref{eq:qbar-def}. Then for $\epsilon \lessapprox \delta^{2\ell_*-1}$ {\color{blue} the following holds: if $j\in \mathcal{C}_t$ (i.e. $x_j$ belongs to a cluster of size $\ell_t$), then} we have $|x_j-\tilde{x}_j| \lessapprox \delta^{2-2\ell_t}\epsilon$.
\end{thm}
{\color{blue} Note that the estimate on the node errors, which appears in Theorem \ref{thm:node-accuracy-Dima}, holds for all node approximations simultaneously, as the upper bound on the noise is independent of a particular node selection.} Theorem \ref{thm:node-accuracy-Dima} then leads to the following result for the amplitude errors $\left\{\left|\alpha_i-\tilde{\alpha}_i \right|\right\}_{i=1}^n$.

\begin{thm}\label{thm:coeffs-accuracy-Dima}
Let $\left\{d_s \right\}, \left\{\breve{d}_s\right\} = O(1)$ be two sets of tolerance coefficients. Under the assumptions of \prettyref{thm:node-accuracy-Dima}, and for each $\epsilon$ satisfying the condition $\epsilon\lessapprox\delta^{2\ell_*-1}$, let $\tilde{x}_s=\tilde{x}_s(\left\{ d_i\right\},\epsilon), \ s=1,\dots,n$ denote the exact roots of the perturbed Prony polynomial $\bar{q}\left(z;\left\{ d_i\right\},\epsilon\right)$, and consider $\operatorname{col}\left\{\tilde{\alpha}_i \right\}_{i=1}^n$ which satisfies
\begin{align*}
&\tilde{V}\cdot\operatorname{col}\left\{\tilde{\alpha}_i \right\}_{i=1}^n = \operatorname{col}\left\{m_i+\epsilon \breve{d}_i \right\}_{i=0}^{n-1},\quad  \tilde{V} = \begin{bmatrix}
1&  \dots &1 &1\\
\tilde{x}_1 &  \dots & \tilde{x}_{n-1}& \tilde{x}_n\\
\vdots  & \vdots & \vdots&\vdots\\
\tilde{x}_1^{n-1} &  \dots & \tilde{x}_{n-1}^{n-1}& \tilde{x}_n^{n-1}
\end{bmatrix}.
\end{align*}
{\color{blue} Given any $j\in \mathcal{C}_t$ (which corresponds to a cluster of size $\ell_t$), we have }
$$
|\alpha_j-\tilde{\alpha}_j| \lessapprox
\begin{cases}
    \delta^{1-2\ell_t}\epsilon & \ell_t > 1;\\
    \epsilon & \ell_t=1.
\end{cases}
$$
\end{thm}
{\color{blue} Similarly to Theorem \ref{thm:node-accuracy-Dima}, the latter theorem holds for all amplitudes simultaneously}. As a consequence, we show in \prettyref{thm:finite-precision} that \prettyref{alg:classical-prony} (and thus also \prettyref{alg:homo-prony}) retains the stability bounds when executed in finite precision arithmetic. However, we require several auxiliary definitions in order to state the precise result, and thus we defer these developments to \prettyref{sec:finite-precision}.

Some preliminary results, which are needed for the proofs, are given in Section \ref{Sec:Prelims}. The proof of the Theorem \ref{thm:node-accuracy-Dima} is the subject of Section \ref{Sec:PfMainthm}, while the proof of Theorem \ref{thm:coeffs-accuracy-Dima} is given in Section \ref{Sec:PfAmplitude}. Some numerical results are presented in \prettyref{sec:numerics}.

\subsection{Proof outline}
{\color{blue}
Before embarking on the technical details of the proofs of the main results, here we provide a simplified version of the key arguments involved.
\subsubsection{\prettyref{thm:node-accuracy-Dima}}
Examining the proof of the second estimate of \prettyref{prop:prony-naive-analysis} in \prettyref{sec:pm-apparent-instability}, we notice that it did not require any particular structure for the errors $q_i-p_i$, but only their magnitude via the estimate \eqref{eq:r-magnitude-naive-bound}. This resulted in a suboptimal bound.

Our strategy to improve the bound \eqref{eq:r-magnitude-naive-bound} on $q(z)-p(z)$ (more precisely, on the difference of the scaled polynomials, $\bar{r}(z)=\bar{q}(z)-\bar{p}(z)$) is to explicitly obtain an expansion of the form
$$
\bar{r}(z) = \sum_{k=1}^n \epsilon^k s_k(z)
$$
where $\{s_k(z)\}_{k=1}^n$ will be well-defined polynomials, depending also on $x_1,\dots,x_n$ (in addition to $z$), thereby making them multivariate, which in turn depend on the perturbation matrix $\mathrm{D}$ (as in \eqref{eq:PertHankel}). The key observation is that $s_k(z)$ can be further written as linear combination of terms of the form
$$
s_k(z;x_1,\dots,x_n) = \sum_{j\in I_k} D_j (z;x_1,\dots,x_n;\mathrm{D}) \prod_{i\in S_j} (x_i-z)\prod_{s,t\in T_j, s<t} (x_s-x_t)^2,
$$
where the sets $\left\{S_j\right\}, \left\{T_j\right\}$ as well as the index sets $\{I_k\}$ are defined combinatorially, while the coefficients $\{D_j\}$ are bounded uniformly (and depend on $\mathrm{D}$). This representation of $s_k$ is obtained by a careful consideration of symmetry properties of the associated Vandermonde type determinants resulting from the determinantal formulas \eqref{eq:HomogPronyPol} and \eqref{eq:qbar-def} of $\bar{p}(z)$ and $\bar{q}(z)$, respectively, while further employing \eqref{eq:PertHankel}.

To demonstrate the argument, consider the first order term $k=1$. Our calculations show that in this case
$$
s_1(z;x_1,\dots,x_n)=\sum_{j=1}^n D_j \prod_{i\neq j} (x_i-z) \prod_{i\neq k;i,k\in[n]\setminus\{j\}} (x_i-x_k)^2.
$$

Let us replace $\bar{r}(z)$ with the first order approximation $\epsilon s_1(z)$, and consider the Rouche's theorem argument from the proof of \prettyref{prop:prony-naive-analysis}. Also, for simplicity, consider the case of a single cluster $\ell=n$. Set $\rho_*=\bar{\rho}_* \delta,\;\epsilon=\bar{\epsilon}\delta^{2n-1}$ where $\bar{\rho}_*,\bar{\epsilon}$ are to be determined. Set $z(\theta)=x_m+\rho_*e^{\imath\theta}$ as before for some $m=1,\dots,n$. On the one hand:
$$
|\bar{p}(z(\theta))| \gtrapprox \delta^{n(n-1)}\rho_* \delta^{n-1}= c_1\bar{\rho}_*\delta^{n^2}
$$
for some constant $c_1$. On the other hand,
\begin{align*}
    |\epsilon s_1(z(\theta))| & \lessapprox \epsilon \biggl\{ \sum_{j\neq m} \delta^{n-2} \rho_* \delta^{(n-1)(n-2)} + \delta^{n-1}\delta^{(n-1)(n-2)} \biggr\}\\
    &\lessapprox \epsilon\delta^{(n-1)^2}=c_2\bar{\epsilon} \delta^{n^2}
\end{align*}
for some constant $c_2$. Therefore, the condition of Rouche's theorem can always be satisfied by choosing $\bar{\rho}_* < \frac{c_2}{c_1}\bar{\epsilon}$.

The general argument will be similar in spirit to what is presented above. However, significant technical elaboration is required due to the combinatorial structures arising in the analysis. Indeed, to complete the argument, we show that
\begin{enumerate}
    \item The arguments can be generalized to the multi-cluster geometry, and
    \item The restriction to first order term is not essential, as the higher order terms exhibit a recursive property, which allows to bound them  in terms of the first order term in $\epsilon$.
\end{enumerate}

\subsubsection{\prettyref{thm:coeffs-accuracy-Dima}}\label{sec:proof-outline-coeffs}
Recall that $\av=\operatorname{col}\{\alpha_i\}_{i=1}^{n}$, $\tilde{\av}=\operatorname{col}\{\tilde{\alpha}_i\}_{i=1}^{n}$, and $\Delta\vec{\alpha}=\tilde{\av}-\av$. Let $\vec{m}$, $\tilde{\vec{m}}$ and $\Delta \vec{m}$ be defined similarly for the first $n$ algebraic moments. The improvement of the bound on the amplitude errors $\left\{\left|\alpha_i-\tilde{\alpha}_i \right|\right\}_{i=1}^n$ in \prettyref{thm:coeffs-accuracy-Dima} is achieved by a careful analysis of the Vandermonde structure of the linear system $(V+\Delta V)(\vec{\alpha}+\Delta\vec{\alpha})=\vec{\tilde{m}}$. Substituting $\vec{\alpha}=V^{-1}\vec{m}$ we can extract $\Delta\vec{\alpha}$ as
$$
\Delta\vec{\alpha} = \tilde{V}^{-1}\Delta\vec{m}-(I-\tilde{V}^{-1}V)\vec{\alpha}.
$$
Our key observation here is that the entries of $\tilde{V}^{-1} V$ are given by the (perturbed) Lagrange interpolation basis polynomials evaluated at the original nodes $\{x_j\}_{j=1}^n$. Due to the specific structure of these polynomials, the correct estimate on $|\Delta\vec{\alpha}|$ can be obtained.

Additional complication arises from the multi-cluster case, as the above analysis no longer provides an accurate estimate; however utilizing correlations between the errors in the perturbed nodes $\tilde{x}_j-x_j$, we are able to obtain the correct bound.
}


\section{Preliminary results}\label{Sec:Prelims}
The following lemma will be useful for obtaining an explicit representation of the Prony polynomials $\bar{p}(z)$, $\bar{q}(z)$ as in \eqref{eq:HomogPronyPol}, resp. \eqref{eq:qbar-def}.
\begin{lem}\label{Lem:HomogPronyPoly}

Let $\tilde{m}_j=m_j+\epsilon d_j, \ j=0,\dots ,2n-1$, let $\bar{q}(z)$ be given by \eqref{eq:qbar-def}, and let $p(z),q(z)$ be the monic Prony polynomials given in \eqref{eq:prony-monic-def}, \eqref{eq:q-monic-def}. Then
\begin{align}
    \bar{p}(z) &= (-1)^n\operatorname{det}\left(H_n \right)p(z)\label{eq:HomogPronyPol1}\\
    \bar{q}(z) &= (-1)^n\operatorname{det}\left(\tilde{H}_n \right)q(z)\label{eq:HomogPronyPol2}
\end{align}
\end{lem}
\begin{proof}
See Appendix \ref{Lem:HomogPronyPolyPf}.
\end{proof}

By Lemma \ref{Lem:HomogPronyPoly} and \prettyref{prop:vand-decomp} we immediately have
\begin{align}\label{eq:PbarExplicitExpr}
&\bar{p}(z) = (-1)^n\left(\prod_{k=1}^n \alpha_k\right) \cdot \left(\prod_{1\leq m<\ell\leq n}\left(x_{\ell}-x_m \right)^2\right)\cdot \left(\prod_{m=1}^n(z-x_m) \right).
\end{align}

{ We introduce several definitions that are essential for deriving an explicit $\epsilon$ expansion of the perturbed Prony polynomial $\bar{q}\left(z;\left\{ d_i\right\},\epsilon\right)$. The notations correspond to\footnote{In the original version of \cite{horn2012matrix} there are several typos in Section 0.8.12, which are fixed in the official errata.} \cite[Section 0.8.12]{horn2012matrix}. Given any $\kappa \in \left\{1,\dots,n+1 \right\}$, let $\mathcal{Q}_{\kappa}^{n+1}$ be the set of all increasing sequences of elements from $[n+1]$ of length $\kappa$, namely
\begin{equation}\label{eq:QkapDef}
\mathcal{Q}_{\kappa}^{n+1}=\left\{\left(i_1,\dots,i_{\kappa} \right) \ | \  1\leq i_1<\dots<i_{\kappa}\leq n+1\right\}.
\end{equation}
The latter is an ordered set with the standard lexicographic order. Given a matrix $A\in\mathbb{C}^{n\times n}$, let {\color{blue} $C_r(A)\in \mathbb{C}^{\binom{n}{r}\times \binom{n}{r}}, \ r\in [n]$} be the $r$-th multiplicative compound of $A$, which consists of all $r\times r$ minors of $A$. Namely, the rows and columns of $C_r(A)$ are indexed by $\beta,\gamma\in \mathcal{Q}^n_{r}$ with the entry $\left[ C_r(A)\right]_{\beta,\gamma}$ being the determinant of the $r\times r$ submatrix obtained by choosing from $A$ the rows in $\beta$ and columns in $\gamma$. We further define $C_0\left(A\right) = 1$. Let   $\operatorname{adj}_r(A), \ r\in [n-1] $ be the $r$-th adjugate of $A$, where we define $\operatorname{adj}_n(A)=1$ and $\operatorname{adj}_0(A)=\operatorname{det}{A}$. In particular, $\operatorname{adj}_1(A)$ is the standard adjugate of $A$.

Introducing
\begin{equation}\label{PronPert1}
\begin{array}{lll}
&G(z)=\left[\begin{array}{ccccc}
1 & z & \ldots & z^{n-1} & z^n \\
m_0 & m_1  & \ldots & m_{n-1} & m_n \\
& &\vdots & & \\
m_{n-1} & m_n &  \ldots & m_{2 n-2} & m_{2 n-1}
\end{array}\right], \ D = \begin{bmatrix}
0 & \ldots & 0 \\
d_0 & \dots & d_n \\
\vdots\\
d_{n-1} & \dots & d_{2 n-1}
\end{bmatrix}
\end{array}
\end{equation}
and using \eqref{eq:PertHankel} and \eqref{eq:HomogPronyPol}, we see that the perturbed Prony polynomial satisfies 
$\bar{q}\left(z;\left\{ d_i\right\},\epsilon\right) = \operatorname{det}(G(z)+\epsilon D)$. By \cite[Formula 0.8.12.3]{horn2012matrix}, we have the expansion
\begin{equation}\label{eq:CompundExpansion}
\bar{q}\left(z;\left\{ d_i\right\},\epsilon\right) =  \sum_{\kappa=0}^{n+1} \epsilon^{\kappa}\theta_{n+1-\kappa}(z), \quad \theta_{n+1-\kappa}(z) = \operatorname{tr}\left( \operatorname{adj}_{n+1-\kappa}(D) C_{n+1-\kappa}(G(z))\right).
\end{equation}
Fixing $\kappa\in [n+1]\cup \{0\}$, the following lemma provides an explicit description of the coefficient of $\epsilon^{\kappa}$ in \eqref{eq:CompundExpansion}.
\begin{lem}\label{lem:qbarAsymp}
The coefficient of $\epsilon^{\kappa}$ in \eqref{eq:CompundExpansion} is given by 
\begin{equation*}
\begin{array}{lll}
&\theta_{n+1-\kappa}(z) = \begin{cases}
  \bar{p}(z), \qquad \kappa = 0 \vspace{0.1cm} \\
  \sum_{\gamma \in \mathcal{Q}_{n+1-\kappa}^{n+1}}\sum_{\beta \in \mathcal{Q}^{n+1}_{n+1-\kappa}: 1\in \beta}\left[\operatorname{adj}_{n+1-\kappa}(D) \right]_{\gamma, \beta} \Gamma_{\beta,\gamma}(z),\qquad 1\leq \kappa \leq n-1 \vspace{0.1cm} \\
  \sum_{i=1}^{n+1}\left[\operatorname{adj}(D) \right]_{i, 1}z^{i-1}, \qquad \kappa = n \vspace{0.1cm} \\
  \operatorname{det}(D), \qquad \kappa = n+1
\end{cases} ,
\end{array}
\end{equation*}
where for any $\gamma,\beta\in \mathcal{Q}_{n+1-\kappa}^{n+1}, \ 1\leq \kappa \leq n-1$ such that $1\in \beta$, the term $\Gamma_{\beta,\gamma}(z)$ is given by 
\begin{equation*}
\begin{array}{lll}
\Gamma_{\beta,\gamma}(z) = \operatorname{det}\begin{bmatrix}
z^{\mathrm{a}} & z^{\mathrm{a}+\mathrm{k}_1} & \ldots & z^{\mathrm{a}+\mathrm{k}_{n-\kappa}}  \\
m_{\mathrm{b}} & m_{\mathrm{b}+\mathrm{k}_1} & \ldots & m_{\mathrm{b}+\mathrm{k}_{n-\kappa}} \\
m_{\mathrm{b}+\mathrm{l}_1} & m_{\mathrm{b}+\mathrm{l}_1+\mathrm{k}_1} & \ldots & m_{\mathrm{b}+\mathrm{l}_1+\mathrm{k}_{n-\kappa}} \\
\vdots\\
m_{\mathrm{b}+\mathrm{l}_{n-\kappa-1}} & m_{\mathrm{b}+\mathrm{l}_{n-\kappa-1}+\mathrm{k}_{1}} & \ldots & m_{\mathrm{b}+\mathrm{l}_{n-\kappa-1}+\mathrm{k}_{n-\kappa}}
\end{bmatrix}_{n-\kappa+1}
\end{array}
\end{equation*}
where 
\begin{equation}\label{eq:GammaParam}
\begin{array}{lll}
& \gamma = \left(i_1,\dots,i_{n+1-\kappa} \right), \quad \beta = \left(1,j_1,\dots,j_{n-\kappa} \right),\\
& \mathrm{a} = i_1-1, \quad \mathrm{b} = (j_1 -2) + (i_1-1),\\
   &\mathrm{k}_1 = i_2-i_1, \dots, \mathrm{k}_{n-\kappa}=i_{n+1-\kappa}-i_1,\\
   &\mathrm{l}_1 = j_2-j_1,\dots, \mathrm{l}_{n-\kappa-1} = j_{n-\kappa}-j_1.
\end{array}
\end{equation}
\end{lem}
\begin{proof}
See Appendix \ref{lem:qbarAsympPf}.
\end{proof}

According to Lemma \ref{lem:qbarAsymp}, we have 
\begin{equation}\label{eq:rdef}
\begin{array}{lll}
r(z)&=r\left(z;\left\{ d_i\right\},\epsilon\right) := \bar{q}\left(z;\left\{ d_i\right\},\epsilon\right) - \bar{p}(z) \\
&=\sum_{\kappa=1}^{n-1}\epsilon^{\kappa}\left[\sum_{\gamma \in \mathcal{Q}^{n+1}_{n+1-\kappa}}\sum_{\beta \in \mathcal{Q}^{n+1}_{n+1-\kappa}: 1\in \beta}\left[\operatorname{adj}_{n+1-\kappa}(D) \right]_{\gamma, \beta} \Gamma_{\beta,\gamma}(z) \right]\\
&\hspace{3mm}+\epsilon^n\sum_{i=1}^{n+1}\left[\operatorname{adj}(D) \right]_{i, 1}z^{i-1} + \epsilon^{n+1}\operatorname{det}(D). 
\end{array}
\end{equation}
We aim to obtain an upper bound on $r(z)$ in terms of the clustered configuration parameters (see Definition \ref{def:cluster}). For that purpose, we consider the expansion \eqref{eq:rdef} and proceed with analyzing $\Gamma_{\beta,\gamma}(z)$ (subject to \eqref{eq:GammaParam}), in the case $1\leq \kappa\leq n-1$.
\begin{thm}\label{Lem:GammaEval}
Let $1\leq \kappa \leq n-1$. Consider $\Gamma_{\beta,\gamma}(z)$ in Lemma \ref{lem:qbarAsymp}, subject to \eqref{eq:GammaParam}. There exist polynomials $\phi^{\beta,\gamma}_{(\omega_1,\dots,\omega_{n-\kappa})}(z), \ (\omega_1,\dots,\omega_{n-\kappa})\in \mathcal{Q}_{n-\kappa}^n$ (see \eqref{eq:phi-explicit} in Appendix \ref{lem:GammBetAlphExp} for explicit description of $\phi^{\beta,\gamma}_{(\omega_1,\dots,\omega_{n-\kappa})}(z)$) such that 
\begin{equation}\label{eq:GammabetgamExp}
\begin{array}{lll}
\Gamma_{\beta,\gamma}(z)&=\sum_{(\omega_1,\dots,\omega_{n-\kappa})\in \mathcal{Q}_{n-\kappa}^n}\left\{\left(\prod_{s=1}^{n-\kappa} (x_{\omega_s}-z) \right)\left(\prod_{1\leq s<t\leq n-\kappa } (x_{\omega_t}-x_{\omega_s})^2 \right)\phi^{\beta,\gamma}_{(\omega_1,\dots,\omega_{n-\kappa})}(z)\right\}.
\end{array}
\end{equation}
Here $\left\{x_j \right\}_{j=1}^n$ and $\left\{\alpha_j \right\}_{j=1}^n$ are the nodes and amplitudes, respectively, of the noiseless Prony's problem. 
\end{thm}
\begin{proof}
See Appendix \ref{lem:GammBetAlphExp}.
\end{proof}

\begin{remark}
 In Theorem \ref{Lem:GammaEval} with $\kappa = n-1$ we use the convention that $\prod_{1\leq s<t\leq n-\kappa } (x_{\omega_t}-x_{\omega_s})^2=1$.  
\end{remark}

Combining \eqref{eq:rdef} and Theorem \ref{Lem:GammaEval}, we obtain the following presentation for $r(z)$
\begin{equation}\label{eq:rdef1}
\begin{array}{lll}
r(z)&=\epsilon^{n+1}\overbrace{\operatorname{det}(D)}^{=0}+\epsilon^n\sum_{i=1}^{n+1}\left[\operatorname{adj}(D) \right]_{i, 1}z^{i-1}\\
&+\sum_{\kappa=1}^{n-1}\epsilon^{\kappa}\left[\sum_{\gamma \in \mathcal{Q}^{n+1}_{n+1-\kappa}}\sum_{\beta \in \mathcal{Q}^{n+1}_{n+1-\kappa}: 1\in \beta}\sum_{(\omega_1,\dots,\omega_{n-\kappa})\in \mathcal{Q}_{n-\kappa}^n}\left[\operatorname{adj}_{n+1-\kappa}(D) \right]_{\gamma, \beta} \right.\\
&\hspace{18mm}\left.\times \left(\prod_{s=1}^{n-\kappa} (x_{\omega_s}-z) \right)\left(\prod_{1\leq s<t\leq n-\kappa } (x_{\omega_t}-x_{\omega_s})^2 \right)\phi^{\beta,\gamma}_{(\omega_1,\dots,\omega_{n-\kappa})}(z) \right].
\end{array}    
\end{equation}
{\color{blue} Here we used the fact that $D$ has a row of zeros (see \eqref{PronPert1}), whence $\operatorname{det}(D)=0$}. Recall the definition of a clustered configuration (Definition \ref{def:cluster}) and consider the expansion \eqref{eq:rdef1}. Since we assume that $\left\{d_i\right\}_{i=0}^{2n-1}=O(1)$ (see Theorem \ref{thm:node-accuracy-Dima}), we have that $\left[\operatorname{adj}(D) \right]_{i, 1}=O(1), \ 1\leq i \leq n+1$ and $\left[\operatorname{adj}_{n+1-\kappa}(D) \right]_{\gamma, \beta} = O(1)$ for any $\gamma,\beta\in \mathcal{Q}_{n+1-\kappa}^{n+1},\ 1\leq \kappa \leq n-1, \ 1\in \beta$. Furthermore, taking into account the continuity of the polynomials $\phi^{\beta,\gamma}_{(\omega_1,\dots,\omega_{n-\kappa})}(z)$, we have that for any $K\subset \mathbb{C}$ compact, we can find a constant $C_1 = C_1(K,\mathfrak{M}_{\alpha},n,\eta,T)$ (i.e., depending on the compact $K$, the clustered configuration parameters $n,\eta,T$ and the upper bound on the amplitudes) such that for all $\gamma,\beta\in \mathcal{Q}_{n+1-\kappa}^{n+1},\ 1\leq \kappa \leq n-1, \ 1\in \beta$ and $(\omega_1,\dots,\omega_{n-\kappa})\in \mathcal{Q}_{n-\kappa}^n$
\begin{equation}\label{eq:phibound} \left\|\phi^{\beta,\gamma}_{(\omega_1,\dots,\omega_{n-\kappa})} \right\|_{L^{\infty}(K)}\leq C_1
\end{equation}
uniformly in $\delta<1$. {\color{blue} The constant $C_1$ will be taken into account below when bounding the difference $r(z)= \bar{q}\left(z\right) - \bar{p}(z)$ as part of applying Rouche's theorem (see the beginning of Section 5).}

{
\subsection{First-order asymptotic constant}

It may be of interest to have an explicit expression of the first-order in $\epsilon$ term in the node error $\eta_j:=x_j-\tilde{x}_j$ as $\epsilon\to 0^+$. The following lemma provides this first-order term.

\begin{lem}\label{lem:etaj}
The following holds for all $j=1,\dots,n$:
\begin{equation}
\eta_j:=x_j-\tilde{x}_j=(-1)^{n+1} \frac{\Psi_j(D)}{\alpha_j \prod_{m\in[n]\setminus\{j\}} (x_j-x_m)^2} \epsilon + O(\epsilon^2), \quad \epsilon\to 0^+.\label{eq:Etaj}
\end{equation}
Here $\left\{\Psi_j(D)\right\}_{j=1}^n$ are polynomials in $\mathfrak{X}:=\{x_1,\dots,x_n\}$ given in \eqref{eq:psi-j-explicit} below.
\end{lem}

\begin{remark}
    Note that if $j\in\mathcal{C}_t$ with $\operatorname{card}(\mathcal{C}_t)=\ell_t$, then \eqref{eq:Etaj} implies that $|\eta_j| \lessapprox |\alpha_j|^{-1}\delta^{2-2\ell_t}\epsilon$ as $\epsilon \to 0^+$, giving the infinitesimal version of \prettyref{thm:node-accuracy-Dima}, and by itself already a vast improvement upon \prettyref{prop:prony-naive-analysis}. Furthermore, the dependence on $1/|\alpha_j|$ in \eqref{eq:Etaj} expresses the intuitive fact that it is harder to accurately recover nodes with smaller amplitudes.
\end{remark}

\begin{proof}[Proof of \prettyref{lem:etaj}]
Recall that $\bar{q}(\tilde{x}_j)=0$. To first order in $\epsilon$, \eqref{eq:rdef1} implies
\begin{equation}\label{eq:first-order-basic-identity}
0 = \bar{q}(\tilde{x}_j)=\bar{p}(\tilde{x}_j) + \epsilon \sum_{\gamma \in \mathcal{Q}^{n+1}_{n}}\sum_{\beta \in \mathcal{Q}^{n+1}_{n}: 1\in \beta} \left[\operatorname{adj}_{n}(D) \right]_{\gamma, \beta}\Gamma_{\beta,\gamma}(\tilde{x}_j) + O(\epsilon^2),
\end{equation}
where, denoting $\qs_s=(1,2,\dots,s-1,s+1,\dots,n)\in\mathcal{Q}^n_{n-1}$, one has
$$
\Gamma_{\beta,\gamma}(\tilde{x}_j) = \sum_{s=1}^n \Biggl\{ \biggl(\prod_{i,k\in[n]\setminus\{s\}} (x_i-x_k)^2 \biggr) \biggl(\prod_{m\in[n]\setminus\{s\}}(x_m-\tilde{x}_j) \biggr) \phi^{\beta,\gamma}_{\qs_s}(\tilde{x}_j) \Biggr\},
$$
with $\phi^{\beta,\gamma}_{\qs_s}$ given by \eqref{eq:phi-explicit} in Appendix~\ref{lem:GammBetAlphExp}. As $\epsilon\to 0^+$ with all other parameters fixed, the only term in $\Gamma_{\beta,\gamma}(\tilde{x}_j)$ which is of order $O(1)$ corresponds to $s=j$. Further substitution of \eqref{eq:PbarExplicitExpr} into  \eqref{eq:first-order-basic-identity} implies
\begin{align*}
&(-1)^n\left(\prod_{k=1}^n \alpha_k\right) \cdot \left(\prod_{1\leq m<\ell\leq n}\left(x_{\ell}-x_m \right)^2\right)\cdot \left(\prod_{m=1}^n(\tilde{x}_j-x_m) \right) \\
& =-\epsilon \biggl(\prod_{i,k\in[n]\setminus\{j\}} (x_i-x_k)^2 \biggr) \biggl(\prod_{m\in[n]\setminus\{j\}}(x_m-\tilde{x}_j) \biggr) \sum_{\gamma \in \mathcal{Q}^{n+1}_{n}}\sum_{\beta \in \mathcal{Q}^{n+1}_{n}: 1\in \beta}\left[\operatorname{adj}_n(D)\right]_{\gamma,\beta}   \phi^{\beta,\gamma}_{\qs_j}(\tilde{x}_j) + O(\epsilon^2)
\end{align*}
Using the explicit form \eqref{eq:phi-explicit} leads to \eqref{eq:Etaj}  where 
\begin{equation}\label{eq:psi-j-explicit}
\Psi_j(D) = \sum_{\gamma \in \mathcal{Q}^{n+1}_{n}}\sum_{\beta \in \mathcal{Q}^{n+1}_{n}: 1\in \beta} \left[\operatorname{adj}_n(D)\right]_{\gamma,\beta} s_{\underline{\lambda}}\left(\mathfrak{X} \right)\frac{ { x_j^{\textrm{a}}}\psi_{\underline{\lambda},\qs_j}}{\prod_{ s<t; s,t\in[n]\setminus\{j\} } (x_{s}-x_{t})}.
\end{equation}
To finish the proof, note that the rightmost factor in \eqref{eq:psi-j-explicit} is a polynomial (cf. Appendix~\ref{lem:GammBetAlphExp}, \eqref{eq:gamma-beta-gamma-explicit} - \eqref{eq:phi-explicit}).
\end{proof}

}

\section{Proof of Theorem \ref{thm:node-accuracy-Dima}}\label{Sec:PfMainthm}

{
Recall the definition of clustered configuration (Definition \ref{def:cluster}). Let $j_*\in [n]$ and assume that $0<|z-x_{j_*}| = \rho_*$. {\color{blue} Assume further that $j_*\in \mathcal{C}_{\mu}$ where the cluster size is $\ell_{\mu}=\operatorname{card}(\mathcal{C}_{\mu})$.} Henceforth, we fix $\epsilon = \bar{\epsilon} \delta^{2\ell_*-1}$ and $\rho_* = \bar{\rho}_*\delta^{2(\ell_*-\ell_{\mu})+1}$, where $\bar{\epsilon}<1$ and {\color{blue} $\bar{\rho}_*<\frac{1}{3}\min(1, T)$} are independent of $\delta$ and will be determined and interconnected subsequently {\color{blue} (see  \eqref{eq:Rouche1} below). Notice that since $\delta<1$ is assumed throughout,  we automatically have that $\rho_*<\delta/3<\delta$.}  Finally, since $|z-x_{j_*}| = \rho_*$, we have that $z$ lies in  $\overline{\mathbb{S}^1+B\left(0,\frac{1}{3}\right)}$. Let $C_1$ be the constant in \eqref{eq:phibound}, which corresponds to this closed neighborhood.

In this section we provide a proof of Theorem \ref{thm:node-accuracy-Dima}. The underlying idea of the proof is to show that subject to the notations above, we can find a pair $(\bar{\epsilon},\bar{\rho}_*)$ (which is independent of $\delta$) for which $\left|\bar{p}\left(z\right)\right| -  |r(z)|>0$, for arbitrarily small $\delta<1$. Then, Rouche's theorem \cite{conway1978functions} guarantees that the perturbed polynomial $\bar{q}\left(z;\left\{ d_i\right\},\epsilon\right)$ has a root $\tilde{x}_{j_*}$ such that $\left| x_{j_*}-\tilde{x}_{j_*}\right|\leq \rho_*$. The latter will readily yield the result of Theorem \ref{thm:node-accuracy-Dima}. 

The proof requires several auxiliary results. From \eqref{eq:PbarExplicitExpr} we find 
\begin{align}
\left|\bar{p}\left(z\right)\right|&= \left(\prod_{k=1}^n |\alpha_k|\right)\cdot \left(\prod_{1\leq m<l\leq n}\left|x_{l}-x_m \right|^2\right) \cdot \left(\prod_{m=1}^n|z-x_m| \right)\nonumber \\
&\geq \mathfrak{m}_{\alpha}^n \cdot \left(\prod_{1\leq m<l\leq n}\left|x_{l}-x_m \right|^2\right) \cdot \left(\prod_{m=1}^n|z-x_m| \right). \label{eq:pLowerBdMultiClus11}
\end{align}
We bound the two rightmost terms separately. First, given $m,\ell\in [n],\ m\neq \ell$, the nodes $x_m,x_{\ell}$ satisfy either $m,\ell\in \mathcal{C}_s$ for some $s\in [\zeta]$ or $m\in \mathcal{C}_{s},\ \ell\in \mathcal{C}_{s'}$ where $s,s'\in[\zeta],\ s\neq s'$. The number of ways to choose two nodes which \emph{do not} belong to the same cluster is  
\begin{align}
    \varrho &= \binom{n}{2}-\sum_{s\in [\zeta]}\binom{\ell_{s}}{2}. \label{eq:xiDef11}
\end{align}
\begin{remark}\label{rem:GenBin}
The latter formula employs \emph{generalized} binomial coefficients. In particular, if for some $s\in [\zeta]$, we have $\ell_s = 1$ (i.e., $\mathcal{C}_s$ consists of a single isolated node), then $\binom{\ell_s}{2}=0$. The same convention will be used in \eqref{eq:varpiiota} below.
\end{remark}
Recalling the definition of a clustered configuration (Definition \ref{def:cluster}) we have 
\begin{align}
&\prod_{1\leq m<l\leq n}\left|x_{l}-x_m \right|^2 \geq T^{2\varrho} \delta^{n(n-1)-2\varrho}.\label{eq:rightterm111}
\end{align}
Moreover,
\begin{align}\label{eq:rightterm211}
\prod_{m=1}^n|z-x_m| &= |z-x_{j_*}|\cdot \left[\prod_{m\notin \mathcal{C}_{\mu}}|z-x_m| \right] \cdot \left[\prod_{m\in \mathcal{C}_{\mu}\setminus\left\{j_* \right\}}|z-x_m| \right]
\end{align}
Using \eqref{eq:rightterm111} and \eqref{eq:rightterm211}, we obtain {\color{blue} (recall that $\rho_*<\delta/3$)}
\begin{equation}\label{eq:pLowerBdMultiClus111}
\begin{array}{lll}
\left|\bar{p}\left(z\right)\right|&\geq  \mathfrak{m}_{\alpha}^n \rho_* \left(\delta-\rho_*\right)^{\ell_{\mu}-1} T^{n+2\varrho-\ell_{\mu}} \delta^{n(n-1)-2\varrho}\\
&\geq \mathfrak{m}_{\alpha}^n \left(\frac{2}{3} \right)^{\ell_{\mu}-1}T^{n+2\varrho-\ell_{\mu}}\bar{\rho}_*\delta^{n(n-1)+2\ell_*-\ell_{\mu}-2\varrho}\\
& = C_2 \bar{\rho}_
* \delta^{n(n-1)+2\ell_*-\ell_{\mu}-2\varrho} 
\end{array}
\end{equation}
with the constant $C_2 = C_2\left(n,\mathfrak{m}_{\alpha},T,\left\{\mathcal{C}_s\right\}_{s=1}^{\zeta} \right)=\mathfrak{m}_{\alpha}^n \left(\frac{2}{3} \right)^{\ell_{\mu}-1}T^{n+2\varrho-\ell_{\mu}}$.

Next, consider the expansion \eqref{eq:rdef1} and write it as 
\begin{equation}\label{eq:rdef2}
\begin{array}{lll}
&r(z)=\epsilon^{n+1}\Omega_{n+1}(z)+\epsilon^n\Omega_n(z)+\sum_{\kappa=1}^{n-1}\epsilon^{\kappa}\Omega_{\kappa}(z),\\
&\Omega_{n+1}(z) =\operatorname{det}(D)  {\color{blue} \overset{\eqref{PronPert1}}{=}0},\\
&\Omega_{n}(z) = \sum_{i=1}^{n+1}\left[\operatorname{adj}(D) \right]_{i, 1}z^{i-1}, \\
&\Omega_{\kappa}(z) = \sum_{\gamma \in \mathcal{Q}^{n+1}_{n+1-\kappa}}\sum_{\beta \in \mathcal{Q}^{n+1}_{n+1-\kappa}: 1\in \beta}\sum_{(\omega_1,\dots,\omega_{n-\kappa})\in \mathcal{Q}_{n-\kappa}^n}\left(\left[\operatorname{adj}_{n+1-\kappa}(D) \right]_{\gamma, \beta} \right. \\
&\hspace{13mm}\left.\times \left(\prod_{s=1}^{n-\kappa} (x_{\omega_s}-z) \right)\left(\prod_{1\leq s<t\leq n-\kappa } (x_{\omega_t}-x_{\omega_s})^2 \right)\phi^{\beta,\gamma}_{\omega_1,\dots,\omega_{n-\kappa}}(z)\right), \quad 1\leq \kappa \leq n-1.
\end{array}    
\end{equation}
We introduce the following notation for $1\leq \kappa \leq n-1$
\begin{equation}\label{eq:PcalDef}
\mathcal{P}_{(\omega_1,\dots,\omega_{n-\kappa})}:=\left(\prod_{s=1}^{n-\kappa} |x_{\omega_s}-z| \right)\left(\prod_{1\leq s<t\leq n-\kappa } |x_{\omega_t}-x_{\omega_s}|^2 \right), \quad (\omega_1,\dots,\omega_{n-\kappa})\in \mathcal{Q}_{n-\kappa}^n
\end{equation}
Then, recalling \eqref{eq:phibound} and  $\left\{d_i\right\}_{i=0}^{2n-1}=O(1)$, we have {\color{blue} for $1\leq \kappa \leq n-1$}
\begin{equation}\label{eq:OmegaPhi}
\begin{array}{lll}
&\epsilon^{\kappa}\left|\Omega_{\kappa}(z) \right|\leq C_1 \max_{\gamma,\beta \in \mathcal{Q}_{n+1-\kappa}^{n+1}}\left|\left[\operatorname{adj}_{n+1-\kappa}(D) \right]_{\gamma, \beta} \right|\Phi_{\kappa}\\
&\Phi_{\kappa} =\sum_{\gamma \in \mathcal{Q}^{n+1}_{n+1-\kappa}}\sum_{\beta \in \mathcal{Q}^{n+1}_{n+1-\kappa}: 1\in \beta}\sum_{(\omega_1,\dots,\omega_{n-\kappa})\in \mathcal{Q}_{n-\kappa}^n} \left[\epsilon^{\kappa}\mathcal{P}_{(\omega_1,\dots,\omega_{n-\kappa})}\right].
\end{array}
\end{equation}

{\color{blue} Note that $\Phi_{\kappa}$ is defined \emph{to contain} $\epsilon^{\kappa}$. Our next step is to derive recursive bounds on $\left\{\Phi_{\kappa} \right\}_{\kappa=2}^{n-1}$ in terms of $\Phi_1$. The latter will simplify the upper bounding of $r(z)$.} To proceed, we derive the following relations on the expressions $\mathcal{P}_{(\omega_1,\dots,\omega_{n-\kappa})}$.
\begin{lem}\label{lem:PcalRecur}
Let $(\omega_1,\dots,\omega_{n-\kappa})\in \mathcal{Q}_{n-\kappa}^n, \ 1\leq \kappa\leq n-2$ and fix $1\leq s\leq n-\kappa$. Then
\begin{equation}\label{eq:PcalDivision}
\frac{\mathcal{P}_{(\omega_1,\dots,\omega_{n-\kappa})}}{\mathcal{P}_{(\omega_1,\dots,\omega_{s-1},\omega_{s+1},\dots,\omega_{n-\kappa})}} \gtrapprox \min \left( \rho_* \delta^{2(\ell_{\mu}-1)},( \delta - \rho_*) \delta^{2(\ell_{\mu}-1)},  \delta^{2(\ell_{*}-1)}\right).
\end{equation}
Moreover, 
\begin{equation}\label{eq:PcalDiv}
    \frac{\epsilon^{\kappa}\mathcal{P}_{(\omega_1,\dots,\omega_{n-\kappa})}}{\epsilon^{\kappa+1}\mathcal{P}_{(\omega_1,\dots,\omega_{s-1},\omega_{s+1},\dots,\omega_{n-\kappa})}}   \gtrapprox \frac{\bar{\rho}_*}{\bar{\epsilon}}.
\end{equation}
\end{lem}
\begin{proof}
See Appendix \ref{lem:PcalRecurPf}.
\end{proof}

Using Lemma \ref{lem:PcalRecur} and \eqref{eq:OmegaPhi} we conclude that there exist positive constants $\Xi_{\kappa}, \ 1\leq \kappa\leq n-1$  that are independent of $\delta$ and satisfy
\begin{equation}\label{eq:OmegaPhi1}
\epsilon^{\kappa}\left|\Omega_{\kappa}(z) \right|\leq \left(\frac{\bar{\epsilon}}{\bar{\rho}_*} \right)^{\kappa-1}\Xi_{\kappa}\Phi_1, \quad 1\leq \kappa\leq n-1.
\end{equation}
{\color{blue} Note that $\Xi_1=1$. Recalling that $\Omega_{n+1}(z)=0$,} introducing  $\Xi_{n} = \sum_{i=1}^{n+1}\left| \left[\operatorname{adj}(D) \right]_{i, 1}\right|\left(\frac{4}{3} \right)^{i-1}$ and employing \eqref{eq:OmegaPhi1}, we obtain the following upper bound
\begin{equation}\label{eq:rdef3}
\begin{array}{lll}
|r(z)|&\leq  \epsilon^n\Xi_n+\left(1+\frac{\bar{\epsilon}}{\bar{\rho}_*}\Xi_1+\dots +  \left(\frac{\bar{\epsilon}}{\bar{\rho}_*} \right)^{n-2}\Xi_{n-1}\right)\Phi_1.
\end{array}    
\end{equation}
Consider first the leftmost term on the right-hand side. {\color{blue} We derive an upper bound which will be used in the application of Rouche's theorem below (see \eqref{eq:pLowerBdMultiClus111} and \eqref{eq:Rouche1}).}
\begin{prop}\label{prop:HighPowerBds}
The following holds 
\begin{equation}\label{eq:HighPoweBd}
(\epsilon/\bar{\epsilon})^n \leq \delta^{n(n-1)+2\ell_*-\ell_\mu-2\varrho}
\end{equation}
\end{prop}
\begin{proof}
See Appendix \ref{prop:HighPowerBdsPf}.
\end{proof}

To continue with \eqref{eq:rdef3}, we would like to upper bound $\Phi_1$ (see \eqref{eq:OmegaPhi} for the definition). 
\begin{lem}\label{lem:PcalEst}
Let  $\epsilon = \bar{\epsilon} \delta^{2\ell_*-1}$ and $\rho_* = \bar{\rho}_*\delta^{2(\ell_*-\ell_{\mu})+1}$, where $\delta<1$. Assume that  $\bar{\epsilon}<1$ and $\bar{\rho}_*<\frac{1}{3}\min(1, T)$ are independent of $\delta$. Then,
\begin{equation*}
\begin{array}{lll}
&\frac{\epsilon\mathcal{P}_{(\omega_1,\dots,\omega_{n-1})}}{\bar{\epsilon}\delta^{n(n-1)+2\ell_*-\ell_\mu-2\varrho}}\lessapprox \begin{cases}
    1, & \omega_s \neq j_* \text{ for all } s\in [n-1]\\
    \bar{\rho}_*\delta^{2(\ell_*-\ell_{\mu})}, & \omega_s\neq t \text{ for some } t\in \mathcal{C}_{\mu}\setminus \left\{j_* \right\},\\
    \bar{\rho}_*\delta^{2(\ell_*-\ell_{\iota})+1}, & \omega_s\neq b \text{ for some } b\in \mathcal{C}_{\iota},\ \iota\neq \mu
\end{cases} 
\end{array}
\end{equation*}
with a constant that is \emph{independent} of $\delta$.
\end{lem}
\begin{proof}
See Appendix \ref{lem:PcalEstPf}
\end{proof}

We are now ready to finish the proof of Theorem \ref{thm:node-accuracy-Dima}. By combining \eqref{eq:OmegaPhi} in the case $\kappa=1$ and Lemma \ref{lem:PcalEst}, we obtain that there exists some $\Xi_0>0$, independent of $\delta$, such that 
\begin{equation}\label{eq:Phi1bound}
\Phi_1\leq \Xi_0 \left(1+ \bar{\rho}_*\delta^{2(\ell_*-\ell_{\mu})}+{\color{blue} \sum_{\iota \neq \mu }}\bar{\rho}_*\delta^{2(\ell_*-\ell_{\iota})+1}\right)\bar{\epsilon} \delta^{n(n-1)+2\ell_*-\ell_\mu-2\varrho}.
\end{equation}
Recalling \eqref{eq:pLowerBdMultiClus111} and \eqref{eq:rdef3}, we have
\begin{equation}\label{eq:Rouche1}
\begin{array}{lll}
{\color{blue}R:=}\frac{\left|\bar{p}\left(z\right)\right| -  |r(z)|}{\delta^{n(n-1)+2\ell_*-\ell_\mu-2\varrho}}&\geq C_2 \bar{\rho}_
* - \frac{ \epsilon^n\Xi_n+\left(1+\frac{\bar{\epsilon}}{\bar{\rho}_*}\Xi_1+\dots +  \left(\frac{\bar{\epsilon}}{\bar{\rho}_*} \right)^{n-2}\Xi_{n-1}\right)\Phi_1}{\delta^{n(n-1)+2\ell_*-\ell_\mu-2\varrho}}\\
&\overset{\eqref{eq:HighPoweBd},\eqref{eq:Phi1bound}}{\geq} C_2 \bar{\rho}_
*  -\bar{\epsilon}^n\Xi_n-\bar{\epsilon}\left(1+\frac{\bar{\epsilon}}{\bar{\rho}_*}\Xi_1+\dots +  \left(\frac{\bar{\epsilon}}{\bar{\rho}_*} \right)^{n-2}\Xi_{n-1}\right)\\
&\hspace{12mm} \times \Xi_0 \left(1+ \bar{\rho}_*\delta^{2(\ell_*-\ell_{\mu})}+{\color{blue} \sum_{\iota \neq \mu }}\bar{\rho}_*\delta^{2(\ell_*-\ell_{\iota})+1}\right).
\end{array}
\end{equation}

{\color{blue}

At this point we impose a linear relationship between $\bar{\rho}_*$ and $\bar{\epsilon}$ as follows:
$$
\bar{\epsilon} = \alpha \bar{\rho}_*, \qquad \alpha < 3.
$$
The condition $\alpha<3$ ensures that $\bar{\epsilon} < 1$ for all $\bar{\rho}_* < \frac{1}{3}$. We can further bound the right-hand side of \eqref{eq:Rouche1}:
\begin{align}\label{eq:Rouche2}
    &\frac{R}{\bar{\rho}_*}  \geq  T(\alpha),\\
    &T(\alpha) := C_2 -\frac{\alpha^n \Xi_n}{3^{n-1}}-\alpha\left(1+\alpha \Xi_1+\dots +  \alpha^{n-2}\Xi_{n-1} \right)\Xi_0 \left(1+ \frac{\delta^{2(\ell_*-\ell_{\mu})}}{3}+\frac{ \sum_{\iota \neq \mu }\delta^{2(\ell_*-\ell_{\iota})+1}}{3}\right).\nonumber
\end{align}


Since $T(0)=C_2>0$ and $T(\alpha)$ is continuous, there exists $\alpha_0$, independent of $\bar{\rho}_*$, such that $T(\alpha)>0$ for all $\alpha\in(0,\alpha_0)$. Now put $\alpha_1:=\min(\alpha_0,3)$. We have shown that for all $\alpha\in(0,\alpha_1)$ and all $\bar{\rho}_*<\frac{1}{3}$ we have $R>0$. Therefore, by applying Rouche's theorem \cite{conway1978functions} and using the fact that $p(z)$ and $\bar{p}\left(z\right)$ differ by a multiplicative constant the following implication is true:
$$
\forall \alpha\in (0,\alpha_1),\;\forall \bar{\rho}_*<\frac{1}{3}: \quad \epsilon=\alpha\bar{\rho}_*\delta^{2\ell_*-1} \implies |x_{j_*}-\tilde{x}_{j_*}|\leq \rho_*=\bar{\rho}_*\delta^{2(\ell_*-\ell_{\mu})+1}=\bar{\rho}_*\frac{\epsilon}{\bar{\epsilon}}\delta^{2-2\ell_{\mu}}=\frac{1}{\alpha}\delta^{2-2\ell_{\mu}}\epsilon.
$$
By \emph{fixing} some $0 < \alpha_2 < \alpha_1$ and letting $\bar{\rho}_*<\frac{1}{3}$ be arbitrary, we rewrite the above as
$$
\forall \bar{\rho}_*<\frac{1}{3}: \quad \epsilon=\bar{\rho}_*\alpha_2 \delta^{2\ell_*-1} \implies |x_{j_*}-\tilde{x}_{j_*}|\leq \frac{1}{\alpha_2}\delta^{2-2\ell_{\mu}}\epsilon,\quad j_*\in \mathcal{C}_{\mu}, \ \ell_{\mu}=\operatorname{card}(\mathcal{C}_{\mu}).
$$
Since the above holds for arbitrary $\delta<1$, this completes the proof of Theorem \ref{thm:node-accuracy-Dima}.
}

}

\section{Proof of Theorem \ref{thm:coeffs-accuracy-Dima}}\label{Sec:PfAmplitude}
{
In this section we consider the amplitude approximation error $\left\{\left|\alpha_{j}-\tilde{\alpha}_{j} \right|\right\}_{j=1}^n$. Throughout the section we assume that the conditions of Theorem \ref{thm:node-accuracy-Dima} hold. In particular, we have $\epsilon = \bar{\epsilon} \delta^{2\ell_*-1}$, whereas given any $j\in [n], \ j\in \mathcal{C}_t$ we denote by $\rho_j = \bar{\rho}_j\delta^{2(\ell_*-\ell_{t})+1}$ {\color{blue} the upper bound on the node recovery error} $\left| x_j-\tilde{x}_j\right|$. Here,  $\bar{\epsilon}<1$ and $\bar{\rho}_j<\frac{1}{3}\min(1, T)$ are independent of $\delta$ {\color{blue} and chosen as in the proof of Theorem \ref{thm:node-accuracy-Dima}}. In particular, note that $\rho_j\lessapprox \frac{\epsilon}{\delta^{2\ell_{t}-2}}$, as obtained from Theorem \ref{thm:node-accuracy-Dima}. 
Let  $\left\{d_i \right\}, \left\{\breve{d}_i \right\} = O(1)$ be two sets of tolerance coefficients (see Theorem \ref{thm:coeffs-accuracy-Dima}).

Recall that the algebraic moments are given by $m_{k}=\sum_{j=1}^{n}\alpha_{j}x_{j}^{k}$. {\color{blue}Consistent with previous notation in \prettyref{prop:ampl-naive} and \prettyref{sec:proof-outline-coeffs}, set $\av=\operatorname{col}\{\alpha_i\}_{i=1}^{n}$, $\tilde{\av}=\operatorname{col}\{\tilde{\alpha}_i\}_{i=1}^{n}$, and $\Delta\vec{\alpha}=\tilde{\av}-\av$. Introducing
\begin{equation}
    \vec{m} = \operatorname{col} \left\{m_i \right\}_{i=0}^{n-1} , \quad \vec{\tilde{m}} = \operatorname{col}\left\{m_i+\epsilon \breve{d}_i \right\}_{i=0}^{n-1}, \quad \Delta\vec{m} = \vec{\tilde{m}}-\vec{m} \label{eq:FreeVars}
\end{equation}
we have that both the perturbed and unperturbed amplitudes satisfy
\begin{align}
  V\cdot \av = \vec{m}, \quad \tilde{V}\cdot \tilde{\av}=\vec{\tilde{m}} \label{eq:Linsyst}
\end{align}
}
with $\tilde{V}$ given in the formulation of Theorem \ref{thm:coeffs-accuracy-Dima} and $V$ having the same form as $\tilde{V}$ with the perturbed nodes $\left\{\tilde{x}_j\right\}_{j=1}^n$ replaced by the unperturbed nodes $\left\{x_j\right\}_{j=1}^n$. Denote further {\color{blue}$\tilde{V} - V =: \Delta V$}. Then
{\color{blue}
\begin{align}
    & \tilde{V} \cdot \Delta\av =  \Delta\vec{m}-\Delta V\cdot \av\Rightarrow \Delta\av = \tilde{V}^{-1}\cdot\left( \Delta \vec{m} - \Delta V\cdot \av\right) = \tilde{V}^{-1}\cdot \Delta\vec{ m} - \left(I - \tilde{V}^{-1}V\right)\cdot \av .\label{eq:VDMsystem1}
\end{align}
}
For the rightmost term in \eqref{eq:VDMsystem1}, let 
\begin{equation}
    \tilde{L}_i(z) = \prod_{k\neq i}\frac{z-\tilde{x}_k}{\tilde{x}_i-\tilde{x}_k} = \sum_{b=0}^{n-1} \tilde{l}_{i,b}z^b,\quad i=1,\dots,n \label{eq:PertLagr}
\end{equation}
be the Lagrange basis corresponding to $\tilde{x}_1,\dots, \tilde{x}_n$. Since the the columns of $\tilde{V}^{-T}$ contain the coefficients of the latter Lagrange basis, we have 
\begin{equation}
    \tilde{V}^{-1} = \begin{bmatrix}
    \tilde{l}_{1,0} & \dots & \tilde{l}_{1,n-1}\\
    \vdots & \vdots & \vdots \\
    \tilde{l}_{n,0} & \dots & \tilde{l}_{n,n-1}
    \end{bmatrix}\Longrightarrow \left[\tilde{V}^{-1}V\right]_{s,b} = \sum_{k=0}^{n-1}\tilde{l}_{s,k}x_b^k = \tilde{L}_{s}(x_b) = \prod_{m\neq s}\frac{x_b-\tilde{x}_m}{\tilde{x}_s-\tilde{x}_m}, \quad s,b\in [n].\label{eq:VDMEval}
\end{equation}

We begin by considering the term {\color{blue}$\tilde{V}^{-1}\cdot \Delta\vec{m}$} in \eqref{eq:VDMsystem1}.
\begin{prop}\label{eq:Clust1Sum}
Fix $j\in [n]$ such that $j\in \mathcal{C}_t$. The following holds
{\color{blue}
\begin{align*}
    &\left|\left[\tilde{V}^{-1} \cdot \Delta\vec{m} \right]_j\right|\lessapprox \frac{\epsilon}{\delta^{\ell_t-1}}.
\end{align*}
}
\end{prop}
\begin{proof}
See Appendix \ref{eq:Clust1SumPf}.
\end{proof}

Consider now the term $\left(I - \tilde{V}^{-1}V\right)$, appearing in \eqref{eq:VDMsystem1}.
\begin{prop}\label{eq:Clust2Sum}
Fix $j\in [n]$ such that $j\in \mathcal{C}_t$. The following estimate holds:
\begin{align}
    \left|\boldsymbol{\delta}_{j,s}-\left[\tilde{V}^{-1}V\right]_{j,s} \right| \lessapprox \begin{cases}
    \frac{\epsilon}{\delta^{2\ell_t-1}}, \quad s\in \mathcal{C}_t  ,\\
    \frac{\epsilon}{\delta^{\ell_t + \max_{a\neq t}\ell_a-2}}, \quad s\notin \mathcal{C}_t
    \end{cases}\label{eq:TildVVBdClus}
\end{align}
where $\boldsymbol{\delta}_{j,s}$ denotes the Kronecker delta.
\end{prop}
\begin{proof}
See Appendix \ref{eq:Clust2SumPf}.
\end{proof}
From Propositions \ref{eq:Clust1Sum} and \ref{eq:Clust2Sum}, and $\mathfrak{m}_{\alpha}\leq \alpha_i \leq \mathfrak{M}_{\alpha},\ i\in [n]$, we obtain the following estimate for {\color{blue}$(\Delta\av)_j$} (see \eqref{eq:VDMsystem1}), where $j\in \mathcal{C}_t$:
\begin{align}
 {\color{blue}\left|(\Delta\av)_j \right|} & {\color{blue}\leq \left|\left[\tilde{V}^{-1} \cdot\Delta\vec{m} \right]_j\right|}+ \sum_{s=1}^n\left(\left|\boldsymbol{\delta}_{j,s}-\left[\tilde{V}^{-1}V\right]_{j,s} \right|\cdot \left|\alpha_s \right|\right) \nonumber\\
&\overset{}{\lessapprox} \frac{\epsilon}{\delta^{\ell_t-1}}+\frac{\epsilon}{\delta^{2\ell_t-1}}+\frac{\epsilon}{\delta^{\ell_t + \max_{a\neq t}\ell_a-2}}\lessapprox\frac{\epsilon}{\delta^{2\ell_*-1}}.\label{eq:alphaeClust}
\end{align}
Note that the analysis in Proposition \ref{eq:Clust2Sum} yields an upper bound which couples the different clusters together via their sizes $\left\{\ell_{a} \right\}_{a\in [\zeta]}$. Therefore, in \eqref{eq:alphaeClust}, we see that clusters of larger size may influence accuracy of coefficients belonging to cluster of smaller size. This leads to a non-sharp upper bound on the amplitude errors, which will be improved in the next section.

\begin{remark}\label{eq:SingleNodeAmp}
 If $x_j, \ j\in \mathcal{C}_t$ is an isolated node, the estimate \eqref{eq:alphaeClust} is replaced by 
\begin{align}
    {\color{blue} \left|(\Delta\av)_j \right|} &\lessapprox \epsilon +\frac{\epsilon}{\delta^{\max_{a\neq t}\ell_a-1}} \lessapprox \frac{\epsilon}{\delta^{\ell_*-1}}.   \label{eq:alphaeClustIsolated}
\end{align}
See Remark \ref{Rem:SingClus} for the necessary modifications.
\end{remark}
}

\subsection{An improved bound on the amplitude error}\label{sec:impBound}

\newcommand{\kk}{\ensuremath{\kappa}}

{\color{blue} Considering $(x_j,\alpha_j)$ with $j\in \mathcal{C}_t$,} from \eqref{eq:TildVVBdClus} and \eqref{eq:alphaeClust} it can be seen that the leading  contribution to the amplitude error stems
from {\color{blue} pairs $(x_s,\alpha_s)$} satisfying $s\in \mathcal{C}_a$ with $a\neq t$ (i.e., pairs \emph{out of the cluster  $\mathcal{C}_t$}), if $\ell_a > \ell_t$. Here we aim to improve the previous analysis, in order to achieve a better bound on the amplitude error.

Let $b\in [\zeta], \ b\neq t$. We consider the contribution of the nodes in $\mathcal{C}_b= \left\{x_{i_1},\dots,x_{i_{\ell_b}} \right\}$ to {\color{blue}$\left|(\Delta\av)_j \right|,\ j\in \mathcal{C}_t$} (leading to the suboptimal bound $\frac{\epsilon}{\delta^{\ell_t + \max_{a\neq t}\ell_a-2}}$ in \eqref{eq:alphaeClust}) through {\color{blue}$\left(I - \tilde{V}^{-1}V\right)\cdot \av$} (see \eqref{eq:VDMsystem1}). For simplicity of notations, we assume henceforth that $b=1$ and $x_{i_k}=x_k,\ k\in [\ell_1]$. This can always be achieved by index permutation.

Since $t\neq 1$, the contribution of $\mathcal{C}_1$ is given by 
\begin{align}
&\mathcal{V}_{1}:=\sum_{\nu=1}^{\ell_1} \left[\tilde{V}^{-1} V\right]_{j,\nu} \alpha_{\nu}  \overset{\eqref{eq:VDMEval}}{=}\sum_{\nu=1}^{\ell_1} \alpha_{{\nu}} \prod_{m \neq j} \frac{x_{{\nu}}-\tilde{x}_{m}}{\tilde{x}_{j}-\tilde{x}_{m} }.\label{eq:ContribCc}
\end{align}
The following theorem bounds $\mathcal{V}_1$. {Although we assume that $b=1$ for ease of notation, note that the theorem remains valid for any cluster $\mathcal{C}_{b}$ where $b\neq t$ (the proof remains the same subject to more tedious notations)}.

\begin{thm}\label{Thm:ImprovedClust}
    Assume the conditions of \prettyref{thm:coeffs-accuracy-Dima}. The following estimate holds:
    \begin{align*}
        \mathcal{V}_1 \lessapprox \delta^{1-\ell_t}\epsilon.
    \end{align*}
\end{thm}

\begin{proof}
Denote
   \begin{align*}
   A&:=\det \tilde{H}_n = (-1)^n \biggl(\prod_{k=1}^{n}\Tilde{\alpha}_k\biggr)\bigg(\prod_{1\leq m < s \leq n}(\tilde{x}_{m}-\tilde{x}_s)^2\bigg)\\
   B&:= \prod_{m\neq j}\left(\tilde{x}_j-\tilde{x}_m\right),\quad   C:= \prod_{\nu=1}^{\ell_1} (x_{{\nu}}-\tilde{x}_j).
   \end{align*}

   These quantities can be effectively bounded \emph{from below}, using the following facts:
   \begin{enumerate}
    \item  $|\tilde{x}_j-\tilde{x}_m|\gtrapprox \delta$ by \prettyref{thm:node-accuracy-Dima};
    \item $|\tilde{\alpha}_j-\alpha_j| \lessapprox 1$ by \eqref{eq:alphaeClust}. Hence, by decreasing $\alpha $ in the proof of \prettyref{thm:node-accuracy-Dima} (see definition above \eqref{eq:Rouche2}), this quantity can be made smaller than $\frac{\mathfrak{m}_{\alpha}}{2} \leq \left|\frac{\alpha_j}{2} \right|$;
    \item $x_j\notin \mathcal{C}_1$.
   \end{enumerate}

   Therefore we have
   \begin{align}
       A & \gtrapprox \delta^{2\sum_{s\in[\zeta]} \binom{\ell_s}{2}},\quad B \gtrapprox \delta^{\ell_t-1}, \quad 
       C \gtrapprox 1 \label{eq:ABCEst}
   \end{align}   

   Note that $\bar{q}(x_{{\nu}}) =  A \prod_{m=1}^n \left(x_{{\nu}}-\tilde{x}_m\right) = \underbrace{\bar{p}(x_{{\nu}})}_{=0}+r(x_{{\nu}})$ for any $\nu \in [\ell_1]$. Recalling \eqref{eq:rdef} and \eqref{eq:ContribCc}, we have
   \begin{align*}
   A\cdot B\cdot C \cdot \mathcal{V}_1  &= \sum_{\nu=1}^{\ell_1} \alpha_{{\nu}} \biggl[ \prod_{r\in [\ell_1]\setminus\{\nu\}} (x_{r}-\tilde{x}_j) \biggr] \bar{q}(x_{{\nu}})\\
   &= \sum_{\kk=1}^{n-1}\epsilon^\kk \sum_{\gamma \in \mathcal{Q}^{n+1}_{n+1-\kappa}}\sum_{\beta \in \mathcal{Q}^{n+1}_{n+1-\kappa}: 1\in \beta}\left[\operatorname{adj}_{n+1-\kappa}(D) \right]_{\gamma, \beta} \sum_{\nu=1}^{\ell_1} \alpha_{{\nu}} \biggl[ \prod_{r\in [\ell_1]\setminus\{\nu\}} (x_{r}-\tilde{x}_j) \biggr] \Gamma_{\beta,\gamma}(x_{{\nu}}) \\
   &\quad + \sum_{\nu=1}^{\ell_1} \alpha_{{\nu}} \biggl[ \prod_{r\in [\ell_1]\setminus\{\nu\}} (x_{r}-\tilde{x}_j) \biggr] \biggl\{{ \left( \sum_{i=1}^{n+1}\left[\operatorname{adj}(D) \right]_{i, 1}x_{\nu}^{i-1}\right)}\epsilon^n\biggr\}\\
   &=: E+F,
   \end{align*}
   where by \eqref{eq:GammaParam} and \eqref{eq:gamma-beta-gamma-explicit}, 
   \begin{equation*}
    \begin{array}{lll}
    \Gamma_{\beta,\gamma}(x_{{\nu}})&=\sum_{(\omega_1,\dots,\omega_{n-\kappa})\in \mathcal{Q}^{n}_{n-\kappa}}\left\{\left(\prod_{s=1}^{n-\kk}\alpha_{\omega_s} \right)\left(\prod_{s=1}^{n-\kappa} (x_{\omega_s}-x_{{\nu}}) \right)\right. \vspace{0.1cm}\\
    &\hspace{5mm}\left.\times \left(\prod_{1\leq s<t\leq n-\kappa } (x_{\omega_t}-x_{\omega_s}) \right)s_{\underline{\lambda}}\left(x_{{\nu}},x_{\omega_1},\dots,x_{\omega_{n-\kappa}} \right) x_{{\nu}}^{\mathrm{a}}\psi_{\underline{\lambda},\omega_1,\dots,\omega_{n-\kappa}}\right\}.
    \end{array}
    \end{equation*}
    Here $\psi_{\underline{\lambda},\omega_1,\dots,\omega_{n-\kappa}}$ is given in \eqref{eq:psilambdDef} and $s_{\underline{\lambda}}\left(x_{\nu},x_{\omega_1},\dots,x_{\omega_{n-\kappa}} \right)$ is the Schur polynomial for the partition $\underline{\lambda}$ and in the variables $x_{\nu},x_{\omega_1},\dots,x_{\omega_{n-\kappa}}$ \cite[Chapter 3]{macdonald1998symmetric}. Note that to get a nonzero summand in $\Gamma_{\beta,\gamma}(x_{{\nu}})$, $\omega_1,\dots,\omega_{n-\kk}$ should \emph{all} be different from ${\nu}$.

    To take care of the term $F$, {\color{blue} we use \eqref{eq:HighPoweBd} to obtain }
    \begin{equation*}
    \frac{\epsilon^{n-1}}{\bar{\epsilon}^{n-1}\left|A\right|}\lessapprox \delta^{(n-1)(2\ell_*-1)-\sum_{s\in [\zeta]}\ell_s(\ell_s-1)} = \delta^{(n-1)(2\ell_*-1)-g_{\zeta,n}(\ell_1,\dots,\ell_{n})}
    \end{equation*}
    where we further employ the notations of Appendix \ref{prop:HighPowerBdsPf}. The proof therein shows that $n(2\ell_*-1)-g_{\zeta,n}(\ell_1,\dots,\ell_n)-2\ell_*+\ell_{\mu}\geq 0$ for any $1\leq \ell_{\mu}\leq n$. Setting $\ell_{\mu}=1$, we get $(n-1)(2\ell_*-1)-\sum_{s\in [\zeta]}\ell_s(\ell_s-1)\geq 0$. Therefore, it is clear that $\frac{\epsilon^{n-1}}{\bar{\epsilon}^{n-1}\left|A\right|}\lessapprox 1$, whence $\left|\frac{F}{A} \right|\lessapprox \epsilon$. {\color{blue} In particular, taking into account \eqref{eq:ABCEst}, this yields
    \begin{equation*}
    \begin{array}{lll}
    &\left|\frac{F}{A\cdot B \cdot C} \right|\lessapprox \frac{\epsilon}{\delta^{\ell_t -1}},
    \end{array}
    \end{equation*}
    which is the desired upper bound.}
Next, we consider $E$. To shorten the notation, we will write $\qs=(\omega_1,\dots,\omega_{n-\kappa})$ for a general element in $\mathcal{Q}^{n}_{n-\kappa}$. Rearranging the order of summation, we can write $E = \sum_{\kappa=1}^{n-1}\epsilon^{\kappa}E_k$, where
\begin{align}\label{eq:e-term}
\begin{split}
&E_{\kk}:= \sum_{\gamma \in \mathcal{Q}^{n+1}_{n+1-\kappa}}\sum_{\beta \in \mathcal{Q}^{n+1}_{n+1-\kappa}: 1\in \beta}\left[\operatorname{adj}_{n+1-\kappa}(D) \right]_{\gamma, \beta} \sum_{\qs \in \mathcal{Q}^{n}_{n-\kappa}} \tilde{E}_{\kappa}^{(\gamma,\beta)}(\qs)\\
&\tilde{E}_{\kappa}^{(\gamma,\beta)}(\qs):= \left(\prod_{s=1}^{n-\kk}\alpha_{\omega_s} \right) \left(\prod_{1\leq s<t\leq n-\kappa } (x_{\omega_t}-x_{\omega_s}) \right)\psi_{\underline{\lambda},\qs} \\
&\qquad \qquad \times  \sum_{\nu=1}^{\ell_1} \alpha_{\nu} s_{\underline{\lambda}}\left(x_{\nu},x_{\omega_1},\dots,x_{\omega_{n-\kappa}} \right) x_{\nu}^{\mathrm{a}} \left(\prod_{s=1}^{n-\kappa} (x_{\omega_s}-x_{\nu}) \right) \prod_{r\in [\ell_1]\setminus\{\nu \}} (x_{r}-\tilde{x}_j).
\end{split}
\end{align}

We denote, as in {\color{blue} the proof of} Lemma~\ref{lem:etaj}, $\qs_s=(1,2,\dots,s-1,s+1,\dots,n)\in\mathcal{Q}^n_{n-1}$ and $\mathfrak{X}=\{x_1,\dots,x_n\}$.

 Consider first the case $\kappa=1$. Given $(\omega_1,\dots,\omega_{n-1})\in \mathcal{Q}_{n-1}^n$, we consider the product $\prod_{s=1}^{n-1} (x_{\omega_s}-x_{{\nu}})$ and notice that it equals zero unless $\left\{{\nu} \right\}\bigcup \left\{\omega_j \right\}_{j=1}^{n-1}=[n]$. Thus, we obtain
\begin{equation*}
\begin{array}{lll}
&E_1 =  \sum_{\gamma \in \mathcal{Q}^{n+1}_{n}}\sum_{\beta \in \mathcal{Q}^{n+1}_{n}: 1\in \beta}\left[\operatorname{adj}_{n}(D) \right]_{\gamma, \beta}\sum_{\qs_{\nu},\nu \in [\ell_1]}\tilde{E}_{1}^{(\gamma,\beta)}(\qs_{\nu}),\\
&\tilde{E}_{1}^{(\gamma,\beta)}(\qs_{\nu}):= \left(\prod_{s\in [n]\setminus \left\{\nu \right\}}\alpha_{s} \right) \left(\prod_{1\leq s<t\leq n, s,t\neq \nu } (x_{t}-x_{s}) \right)\psi_{\underline{\lambda},\qs_{\nu}} \\
&\qquad \qquad \times  \sum_{\nu=1}^{\ell_1} \alpha_{\nu} s_{\underline{\lambda}}\left(\mathfrak{X} \right) x_{\nu}^{\mathrm{a}} \left(\prod_{s\in [n]\setminus\{\nu\}}(x_{s}-x_{\nu}) \right) \prod_{r\in [\ell_1]\setminus\{\nu \}} (x_{r}-\tilde{x}_j)
\end{array}
\end{equation*}
which {\color{blue} we \emph{treat as a multivariate polynomial} in what comes next, in order to infer its divisibility}. Using this presentation and the properties of $s_{\underline{\lambda}}$ and $\psi_{\underline{\lambda},\qs_{\nu}}$ {\color{blue} (recall that the former polynomial is symmetric, whereas the latter is an alternating polynomial, as was shown in Appendix \ref{lem:GammBetAlphExp})}, it 
can be readily verified that $E_1$ satisfies the following properties:
\begin{enumerate}
\item For any $s<t, \  s,t\notin\mathcal{C}_1$, if $x_s=x_t$ then $E_1=0$: { in this case we have $\prod_{1\leq s<t\leq n, s,t\neq \nu } (x_{t}-x_{s})\equiv 0$ in $\tilde{E}_{1}^{(\gamma,\beta)}(\qs_{\nu})$.}
\item $E_1$ is invariant with respect to any transposition of two nodes not in $\mathcal{C}_1$. { This follows from symmetry of the polynomials $s_{\underline{\lambda}}$ and $\psi_{\underline{\lambda},\qs_{\nu}}\cdot \prod_{1\leq s<t\leq n, s,t\neq \nu } (x_{t}-x_{s})$}.
\item For any $s<t, \ s,t\in\mathcal{C}_1$, if $x_s=x_t$ then $E_1=0$. This follows from $\prod_{s\in [n]\setminus{\nu}}(x_{s}-x_{\nu})\equiv 0$ in $\tilde{E}_{1}^{(\gamma,\beta)}(\qs_{\nu})$.
\item $E_1$ is invariant with respect to any transposition of the nodes of $\mathcal{C}_1$. {This again follows from symmetry of the polynomials in $\tilde{E}_{1}^{(\gamma,\beta)}(\qs_{\nu})$ and summation over $\qs_{\nu},\ \nu\in [\ell_1]$. }
\end{enumerate}
As a result, $E_1$ is divisible by    
    $$
    H:= \prod_{s<t; s,t\notin\mathcal{C}_1} (x_s-x_t)^2 \prod_{1\leq r < s \leq \ell_1}(x_{r}-x_{s})^2.
    $$
    This {\color{blue} divisibility} immediately implies
    $$
    |E_1| \lessapprox |H| \lessapprox \delta^{2\sum_{s\neq 1} \binom{\ell_s}{2}+2\binom{\ell_1}{2}}=\delta^{2\sum_{s\in[\zeta]}\binom{\ell_s}{2}} \implies \epsilon \left|\frac{E_1}{A} \right| \lessapprox \epsilon.
    $$
    {\color{blue} Hence, by again taking into account \eqref{eq:ABCEst}, we have
    \begin{equation*}
    \epsilon \left|\frac{E_1}{A\cdot B \cdot C}\right|\lessapprox \frac{\epsilon}{\delta^{\ell_t-1}},
    \end{equation*}
    which involves the desired upper bound.}

    We now proceed by induction on $\kappa$ to show that $|E_{\kk} \epsilon^\kk / A| \lessapprox \epsilon$ for $\kk=1,\dots,n-1$.  
    It can be readily verified that for each $\mathfrak{q}\in \mathcal{Q}^{n}_{n-\kappa}$, $\tilde{E}_{\kappa}^{(\gamma,\beta)}(\qs)$ in \eqref{eq:e-term} satisfies the symmetry properties (1)--(4) above with respect to the sets $\mathfrak{q}\setminus\mathcal{C}_1$ and $\mathfrak{q}\cap \mathcal{C}_1$. Consequently, for each $\gamma,\beta$, $\tilde{E}_{\kappa}^{(\gamma,\beta)}(\qs)$ is divisible by $H(\qs)$ where
    $$
    H(\mathfrak{q}):= \prod_{s<t, s,t\in \mathfrak{q}\setminus\mathcal{C}_1 } (x_s-x_t)^2  \prod_{r<s, r,s\in \mathfrak{q} \cap \mathcal{C}_1}(x_{r}-x_{s})^2
    $$
    does not depend on $\gamma,\beta$.

    Let $\mathfrak{q}_1\in \mathcal{Q}^n_{n-\kk}$, $\mathfrak{q}_2\in \mathcal{Q}^n_{n-\kk-1}$ such that $\mathfrak{q}_2 \subset \mathfrak{q}_1$ with the obvious meaning (i.e. they differ by a single index). Since $\epsilon \lessapprox \delta^{2\ell_*-1}$, we conclude, by an argument similar to the one employed in the proof of Lemma \ref{lem:PcalRecur}, that
    $$
    \epsilon |H(\mathfrak{q}_2)| \lessapprox |H(\mathfrak{q}_1)|.
    $$

    Applying the induction hypothesis, we have for each $\qs\in\mathcal{Q}^n_{n-\kk}$ and each $\gamma,\beta$ that
    $$
    \epsilon^{\kk-1} |H(\qs)| \lessapprox |H| \implies \epsilon^{\kk-1} \left| \tilde{E}_{\kappa}^{(\gamma,\beta)}(\qs) \right| \lessapprox |H| \implies \epsilon^{\kk-1} | E_{\kk} | \lessapprox |H|.
    $$
    We conclude that $|E_{\kk} \epsilon^\kk / A| \lessapprox \epsilon$ for $\kk=1,\dots,n-1$.
    This proves $\mathcal{V}_1 \lessapprox \delta^{1-\ell_t}\epsilon$, as required.
\end{proof}

\begin{proof}[Proof of Theorem~\ref{thm:coeffs-accuracy-Dima}]
Theorem \ref{Thm:ImprovedClust} provides the following, improved, estimate on $\Delta\av$, appearing in \eqref{eq:VDMsystem1}
\begin{align}
&\left|(\Delta\av)_j \right|\overset{}{\lessapprox} \frac{\epsilon}{\delta^{\ell_t-1}}+\frac{\epsilon}{\delta^{2\ell_t-1}}+\frac{\epsilon}{\delta^{\ell_t-1}}\lessapprox\frac{\epsilon}{\delta^{2\ell_t-1}},\quad j\in [n], \label{eq:alphaeClust1}
\end{align}
where the term $\frac{\epsilon}{\delta^{2\ell_*-2}}$ in \eqref{eq:alphaeClust} (which is obtained from nodes not in $\mathcal{C}_t$) is replaced by the improved estimate $\frac{\epsilon}{\delta^{\ell_t-1}}$, which follows from Theorem \ref{Thm:ImprovedClust}, {\color{blue} by applying it to each of the clusters different from $\mathcal{C}_t$}. This finishes the proof of Theorem \ref{thm:coeffs-accuracy-Dima}.
\end{proof}

\begin{remark}\label{Rem:SingClus11}
For an isolated node $x_j$, Theorem \ref{Thm:ImprovedClust} implies that the upper bound on $\left|(\Delta\av)_j \right|$ in Remark \ref{eq:SingleNodeAmp} is replaced by $\left|(\Delta\av)_j \right| \lessapprox \epsilon$.
\end{remark}

\section{Stability in finite precision computations}\label{sec:finite-precision}
\newcommand{\hank}{\ensuremath{\mathcal{H}_n}}

So far, the results presented in Theorem \ref{thm:node-accuracy-Dima} and Theorem \ref{thm:coeffs-accuracy-Dima} analyze the performance and output of Prony's method subject to computations in exact arithmetic. In particular, we show that 
the forward errors $\left\{|\tilde{x}_j-x_j|\right\}_{j=1}^n$ and $\left\{|\tilde{\alpha}_j-\alpha_j|\right\}_{j=1}^n$ are bounded by the condition number of the problem (see Theorem \ref{thm:minmax}) multiplied by the noise level. Therefore, Prony's method achieves the theoretically optimal recovery bound for both the nodes and the amplitudes.

In this section we analyze the performance of Prony's method in the regime of finite-precision computations. In order to carry out such an analysis, we recall the definition of a backward error induced by a particular numerical algorithm. We define the normwise backward error in the spirit of approximation theory:
\begin{defn}\label{def:GenBackErr}
Let $f:\mathbb{C}^m \to \mathbb{C}^n$ be a given function. Fix $x\in \mathbb{C}^m$ and let $y\in \mathbb{C}^n$ be an approximation of $f(x)$. The backwards error in $y$ is defined as 
\begin{align*}
    \eta(y) = \inf\left\{\nu \ : \ y=f\left(x+\Delta x\right), \left| \Delta x\right|\leq \nu \left| x\right|  \right\}.
\end{align*}
Namely, the (normwise) backward error is the smallest relative error between $x$ and a perturbed $x+\Delta x$, which is mapped to $y$ under $f$.
\end{defn}
In the context of computations, one can think of $y$ in Definition \ref{def:GenBackErr} as obtained from $f(x)$ by round-off, due to finite-precision computation. A well-known approach in numerical analysis is to bound the backward errors committed by a particular numerical algorithm, and conclude that the actual (forward) error is bounded by the multiple of the backward error and the condition number \cite{higham1996}. A prototypical bound of the forward error 
for Definition \ref{def:GenBackErr} is then of the form
\begin{equation*}
    \frac{\left|y-f(x) \right|}{\left| f(x)\right|}\leq \operatorname{cond}(f,x)\eta(y) + O(\eta(y))^2
\end{equation*}
where $\operatorname{cond}(f,x)$ is an appropriate condition number associated with the function $f$ and the point $x$. Unfortunately, deriving a-priori bounds on the backward error is essentially difficult, and therefore it is frequently of interest to compute a bound numerically, given an output of a particular algorithm.

To concretize the discussion above for the context of Prony's method and Algorithm \ref{alg:classical-prony}, we introduce the following definition. {In what follows, for a vector $\vec{y}=\operatorname{col}\{y_i\}_{i=0}^N$ and $0\leq j < k \leq N$ two indices, we denote $\vec{y}[j:k]=\operatorname{col}\{y_i\}_{i=j}^k$.}

\begin{defn}\label{def:backward-errors}
Let $\vec{q}^{\circ} = \operatorname{col}\left\{ q^{\circ}_j\right\}_{j=0}^{n-1}$, $\vec{x}^{\circ}=\operatorname{col}\{x^{\circ}_j\}_{j=1}^n$ and $\vec{\alpha}^{\circ}= \operatorname{col}\left\{\alpha^{\circ}_j \right\}_{j=1}^n$ denote the results of Steps 2, 3 and 4, respectively, in a (finite-precision) numerical implementation of Prony's method (\prettyref{alg:classical-prony}). The backward errors for each of the steps are defined as follows (all inequalities of the form $|\vec{a}|\leq c$ where $\vec{a}$ is a vector or a matrix are to be interpreted component-wise):
\begin{itemize}
\item {\underline{Step 2:} Let $\tilde{\vec{m}} = \operatorname{col}\left\{\tilde{m}_j \right\}_{j=0}^{2n-1}$ be the vector of perturbed moments. For a vector $\vec{v}\in\mathbb{C}^{2n-1}$, let $\hank(\vec{v})\in\mathbb{C}^{n\times n}$ denote the Hankel matrix constructed from $\vec{v}$. We define, for $\vec{\hat{m}}=\operatorname{col}\{\hat{m}_j\}_{j=0}^{2n-1}$,
\begin{align}\label{eq:backward-error-q}
\operatorname{berr}_1(\vec{q}^{\circ};\tilde{\vec{m}}) &= \inf \biggl\{\epsilon_1: \hank(\vec{\hat{m}}[0:2n-2]) \cdot \vec{q}^{\circ} = -\vec{\hat{m}}[n:2n-1],\quad |\hat{\vec{m}}-\tilde{\vec{m}}|\leq \epsilon_1 \biggr\}
\end{align}
to be the backward error corresponding to the Hankel system
$$
\hank(\vec{\tilde{m}}[0:2n-2]) \cdot \operatorname{col}\left\{\tilde{q}_j \right\}_{j=0}^{n-1} = -\vec{\tilde{m}}[n:2n-1].
$$ }
\item \underline{Step 3:} Denote $\hat{\vec{q}} = \operatorname{col}\left\{\hat{q}_j \right\}_{j=0}^{n-1}$ and let
\begin{align}\label{eq:eq:backward-error-x}
\operatorname{berr}_2(\vec{x}^{\circ};\vec{q}^{\circ}) &= \inf\biggl\{\epsilon_2: (x^{\circ}_j)^n+\sum_{i=0}^{n-1} \hat{q}_i (x^{\circ}_j)^i=0\quad \forall j=1,\dots,n,\quad |\vec{q}^{\circ}-\vec{\hat{q}}|\leq\epsilon_2\biggr\}
\end{align}
be the backward error corresponding to the numerically obtained roots of the Prony polynomial.
\item {\underline{Step 4:} Denote $\vec{m^*}\in\mathbb{C}^{n}$ and $\vec{x}^{*}=\operatorname{col}\left\{x_j^* \right\}_{j=1}^n$. Then define
\begin{align}\label{eq:eq:backward-error-alpha}
\operatorname{berr}_3(\vec{\alpha}^{\circ};\vec{\tilde{m}}[0:n-1],\vec{x}^{\circ}) &= \inf\biggl\{\epsilon_3: V(\vec{x}^*)\vec{\alpha}^{\circ}=\vec{m}^*,\quad |\vec{x}^*-\vec{x}^{\circ}| \leq \epsilon_3,\; \bigl|\vec{m}^*-\vec{\tilde{m}}[0:n-1]\bigr|\leq\epsilon_3 \biggr\}. 
\end{align}
to be the backward error corresponding to solving the Vandermonde system for the amplitudes.}
\end{itemize}
\end{defn}

Our goal is to estimate the total backward error of Algorithm \ref{alg:classical-prony}, by aggregating over the backward errors of the steps above. To estimate the backward errors in practice, there are several available results in the literature, e.g. \cite{bartels1992a,higham1992,stetter2004,sun1998}. For the Hankel linear system structured backward error (when the right-hand side depends on a subset of the same parameters as the Hankel matrix), we have the following:
{\begin{prop}\label{Prop:berr1}
    Let $\vec{r}=\vec{r}(\vec{\tilde{m}},\vec{q}^{\circ})=\vec{\tilde{m}}[n:2n-1]+\hank(\vec{\tilde{m}}[0:2n-2]) \cdot \vec{q}^{\circ}\in\mathbb{C}^n$ denote the actual residual. Then
    $$
    \operatorname{berr}_1(\vec{q}^{\circ};\vec{\tilde{m}}) \lessapprox \min \{ \|\vec{\delta}\|_{2}: \vec{\delta}\in\mathbb{C}^{2n}, C_n \vec{\delta} = \vec{r} \} \leq \|C_n^{\dagger}\|_2 \|\vec{r}\|_2,
    $$
    where
    \begin{equation}\label{eq:Cmat}
        C_n(\vec{q}^{\circ})=C_n = \begin{bmatrix}
            q^{\circ}_0 & q^{\circ}_1 & \dots & q^{\circ}_{n-1} & 1 & 0 & \dots & 0\\
            0 & q^{\circ}_0 & q^{\circ}_1 & \dots & q^{\circ}_{n-1} & 1 & 0 \dots &  0\\
            \ddots & \ddots & \ddots\\
            0  & \dots & 0 & q^{\circ}_0 & q^{\circ}_1 & \dots & q^{\circ}_{n-1} & 1
        \end{bmatrix}\in\mathbb{C}^{n\times 2n}.
    \end{equation}
\end{prop}
\begin{proof}
    Following \cite{higham1992}, denote $\vec{\delta}=\vec{\hat{m}}-\vec{\tilde{m}}$, then the constraint
    \begin{equation}\label{eq:hankel-constraint}
    \hank(\vec{\hat{m}}[0:2n-2]) \cdot \vec{q}^{\circ} = -\vec{\hat{m}}[n:2n-1]
    \end{equation}
     in \eqref{eq:backward-error-q} can be rewritten as $C_n\vec{\delta}=\vec{r}$. The exact value of the backward error is the solution to the underdetermined constrained minimization problem $\min  \{ \| \vec{\delta}\|_{\infty}: C_n\vec{\delta}=\vec{r} \}$, which can be, up to constants, bounded by the solution in the 2-norm.
\end{proof}
}

The backward error $\operatorname{berr}_2(\vec{x}^{\circ};\vec{q}^{\circ})$, given in \eqref{eq:eq:backward-error-x}, can be easily computed as follows: since the values of the roots in finite-precision, $\{x^{\circ}_j\}_{j=1}^n$, are known, the coefficients $\hat{\vec{q}} = \operatorname{col}\left\{\hat{q}_j \right\}_{j=0}^{n-1}$ of the corresponding polynomial can be computed explicitly, followed by direct evaluation of the norm $\left\|\vec{q}^{\circ}-\hat{\vec{q}} \right\|_{\infty}$.

As for the structured backward Vandermonde error, we have the following.

\begin{prop}
    $\operatorname{berr}_3(\vec{\alpha}^{\circ};{\vec{\tilde{m}}[0:n-1]},\vec{x}^{\circ})$ can be bounded using the following steps:
    \begin{enumerate}
        \item Compute the linearized backward error
        $$
        \operatorname{berr_3^{lin}}(\vec{\alpha}^{\circ};{\vec{\tilde{m}}[0:n-1]},\vec{x}^{\circ}) = \min \biggl\{\|(\vec{\delta_1},\vec{\delta_2})\|_{2}:  \underbrace{\bigl[ V'(\vec{x}^{\circ}) \operatorname{diag}(\vec{\alpha}^{\circ}) \quad -I_{n\times n}\bigr]}_{:=U}  (\vec{\delta_1},\vec{\delta_2})^T = \vec{r}\biggr\},
        $$
        where, as before, $\vec{r}={\vec{\tilde{m}}[0:n-1]}-V(\vec{x}^{\circ}) \vec{\alpha}^{\circ}$ and $V' = \big[{(x_j^{\circ})}^k \quad k{(x_j^{\circ})}^{k-1} \big]^{j=1,\dots,n}_{k=0,\dots,n-1}$ is the confluent Vandermonde matrix. The solution is given by $(\vec{\delta_1},\vec{\delta_2})^T = U^{\dagger}\vec{r}$.
        \item Substitute the obtained minimal perturbation in $\vec{x}$ and obtain an actual bound for the right-hand side perturbation:
        $$
        \vec{\delta_2}'=V(\vec{x}^{\circ}+\vec{\delta_1})\vec{\alpha}^{\circ}-{\vec{\tilde{m}}[0:n-1]}.
        $$
        \item Take the bound for the actual backward error to be the maximum between $\|\vec{\delta_1}\|,\;\|\vec{\delta_2}\|,\;\|\vec{\delta_2}'\|$.
    \end{enumerate}
\end{prop}
\begin{proof}
See p.26 in \cite{bartels1992a}. The $\infty$-norm estimates can again be bounded by the solution in the Euclidean norm.
\end{proof}

\begin{thm}\label{thm:finite-precision}
Suppose that the numerical algorithms in steps 2, 3 and 4 of \prettyref{alg:classical-prony} are \emph{backward stable}, i.e. the backward errors \eqref{eq:backward-error-q}, \eqref{eq:eq:backward-error-x} and \eqref{eq:eq:backward-error-alpha} are on the order of machine epsilon $\epsilon_M$, and also that $\epsilon_M \lessapprox \epsilon$. Further assume that for \eqref{eq:Cmat}, it holds that  $\|C_n^{\dagger}(\vec{q}^{\circ})\|_2 = O(1)$ (i.e. independent of the minimal separation $\delta$). Then the bounds of \prettyref{thm:node-accuracy-Dima} and \prettyref{thm:coeffs-accuracy-Dima} hold for $\{x^{\circ}_j\}_{j=1}^n$ and $\{\alpha^{\circ}_j\}_{j=1}^n$ in place of $\{\tilde{x}_j\}_{j=1}^n$ and $\{\tilde{\alpha}_j\}_{j=1}^n$.
\end{thm}
\begin{proof}
By backward stability of Step 2, $\vec{q}^{\circ}$ is the exact solution to the \emph{Hankel} system 
$$
{\hank(\vec{\hat{m}}[0:2n-2]) \cdot \vec{q}^{\circ} = -\vec{\hat{m}}[n:2n-1]}
$$
with moment vector {$\vec{\hat{m}}\in\mathbb{C}^{2n}$} such that 
$\|\vec{\hat{m}}-\vec{\tilde{m}}\|=O(\epsilon_M)$. By backward stability of Step 3 (root finding), $\vec{x}^{\circ}$ are the exact roots of the polynomial $\hat{q}(z)$ such that $\|\vec{q}^{\circ}-\vec{\hat{q}}\|=O(\epsilon_M)$. Using Proposition \ref{Prop:berr1} {and \eqref{eq:hankel-constraint}}, we have
\begin{align*}
\operatorname{berr}_1(\vec{q}^{\circ},\vec{\tilde{m}})&\lessapprox \|C_n(\vec{q}^{\circ})^{\dagger}\|_2 { \|\vec{\tilde{m}}[n:2n-1]+\hank(\vec{\tilde{m}}[0:2n-2]) \cdot \vec{q}^{\circ}\|_2}\\
&\leq \|C_n(\vec{q}^{\circ})^{\dagger}\|_2 {\| \left(\tilde{\vec{m}}-\hat{\vec{m}} \right)[n:2n-1]+\left[\hank(\vec{\tilde{m}}[0:2n-2]) -\hank(\vec{\hat{m}}[0:2n-2]) \right]\vec{q}^{\circ}\|_2}\\
& = O(\epsilon_M).   
\end{align*}
Since $\epsilon_M\lessapprox \epsilon$, we conclude that $\vec{x}^{\circ}$ are the exact roots of a polynomial $\bar{q}(\vec{m}^{\circ};z)$ {where $\vec{m^{\circ}}\in\mathbb{C}^{2n}$} with $\|\vec{m}^{\circ}-\vec{\tilde{m}}\|\leq \| \vec{m}^{\circ}-\vec{\hat{m}}\|+\|\vec{\hat{m}}-\vec{\tilde{m}}\|=O(\epsilon_M)$. The latter implies that the bounds for $|x^{\circ}_j-x_j|$ hold as specified in 
\prettyref{thm:node-accuracy-Dima}. By backward stability of Step 4, there exist $\vec{x}^*$ and $\vec{m}^*$ such that $\alpha^{\circ}$ is the exact solution of the Vandermonde system $V(\vec{x}^*)\vec{\alpha}^{\circ}=\vec{m}^*$ with $\vec{x}^*-\vec{x}^{\circ}=O(\epsilon_M)$ and $\|\vec{m}^*-{\vec{\tilde{m}}[0:n-1]}\|=O(\epsilon_M)$. Clearly $\vec{x}^*$ are the exact roots of a polynomial with coefficients $\vec{q}^*$ satisfying $\|\vec{q}^*-\hat{\vec{q}}\|=O(\epsilon_M)$. But then we also have $\|\vec{q}^*-\vec{q}^{\circ}\|=O(\epsilon_M)$ and by the same computation as above we conclude that $\vec{x}^*$ are the exact roots of a polynomial { $\bar{q}(\vec{\breve{m}};z)$ with $\|\vec{\breve{m}}-\vec{\tilde{m}}\|=O(\epsilon_M)$.} Applying \prettyref{thm:coeffs-accuracy-Dima} completes the proof.
\end{proof}

\section{Numerical results}\label{sec:numerics}

The numerical performance of both PM and DPM has already been investigated in \cite{decimatedProny}, suggesting their optimality in the corresponding regime (respectively, as either $\delta\ll {1\over \Omega}$ or $\srf \gg 1$ {where $\srf:={1\over{\Omega\delta}}$}). In particular, the results reported in \cite{decimatedProny} confirm the predictions of \prettyref{thm:node-accuracy-Dima} and \prettyref{thm:coeffs-accuracy-Dima}. Numerical results in the multi-cluster setting for additional SR algorithms such as ESPRIT, MUSIC and Matrix Pencil are available in, e.g., \cite{li2020a,li2021,batenkov2021b}. Here we complement the experiments in \cite{decimatedProny} by computing the backward errors in each step of PM (cf. \prettyref{thm:finite-precision}), implying that PM attains the optimal bounds in finite-precision arithmetic as well (\prettyref{fig:vanilla-nonproj-DP-backward}).

In all experiments, we consider a clustered configuration with $n=3$ nodes, where node $j=3$ is isolated, and construct the signal with random complex amplitudes (as in the model \eqref{eq:ft-of-measure}), while adding random bounded perturbations (noise) to the measurements. We measure the actual error amplification factors i.e., the  condition numbers of this problem) of the nodes and amplitudes (cf. \cite[Algorithm 3.3]{batenkov2021b}):
$$\mathcal{K}_{x,j}:=\epsilon^{-1}\Omega|x_j-\tilde{x}_j|, \quad  \mathcal{K}_{\alpha,j}:=\epsilon^{-1}|\alpha_j-\tilde{\alpha}_j|,$$
choosing $\epsilon,  \delta$ at random from a pre-defined range. Then, reconstruction of the signal is performed using one of three methods:  classical Prony Method, Decimated Prony Method and Matrix Pencil. For each method, we compare the scaling of the condition numbers in two scenarios: projecting the recovered nodes on the unit circle prior to recovering
the amplitudes versus non-projecting them. All computations were done in double-precision floating-point arithmetic.

As evident from \prettyref{fig:prony-initial-check} (middle pane), the correlations between the errors in the coefficients of the Prony polynomial $\bar{q}(z)$ are essential for obtaining the correct asymptotics for the errors in $\{x_j\}_{j=1}^n$, as done in the proof of \prettyref{thm:node-accuracy-Dima}. On the other hand, \prettyref{fig:prony-initial-check} (right pane) does not appear to suggest that the correlations between the errors in $\{x_j\}_{j=1}^n$ have any influence on the accuracy of recovering the $\alpha_j$'s. In hindsight, the reason is clear: the proof of the estimates \eqref{eq:alphaeClust} (or \eqref{eq:alphaeClustIsolated}) does not require any correlations between the different errors. In contrast, the improved analysis in \prettyref{sec:impBound} uses these correlations in an essential way via the expression $\sum_{\nu=1}^{\ell_1} \alpha_{{\nu}} \biggl[ \prod_{r\in [\ell_1]\setminus\{\nu\}} (x_{r}-\tilde{x}_j) \biggr] \bar{q}(x_{{\nu}})$, resulting in the improved bound \eqref{eq:alphaeClust1}. Thus, if we were to perturb the recovered nodes in random directions, Theorem~\ref{Thm:ImprovedClust} would no longer be valid, and the bound \eqref{eq:alphaeClust}  would be the ``next best thing''. It turns out that a simple {\bf projection of the complex nodes $\{\tilde{x}_j\}_{j=1}^n$ to the unit circle prior to recovering the amplitudes} $\alpha_j$ provides the required perturbation -- \prettyref{fig:discrepancy-DP} demonstrates the loss of accuracy of the non-cluster node $j=3$ when projecting all nodes. Here we also plot the normalized ``cluster discrepancy'' $\discr_1/\epsilon$ given by \eqref{eq:ContribCc} measuring the influcence between the different clusters, which should scale according to either \eqref{eq:TildVVBdClus} (second estimate) in the projected case, or according to Theorem~\ref{Thm:ImprovedClust} in the non-projected case. Note that the multi-cluster geometry is essential to observe such a behavior.

\begin{figure}
    \includegraphics[width=0.9\linewidth]{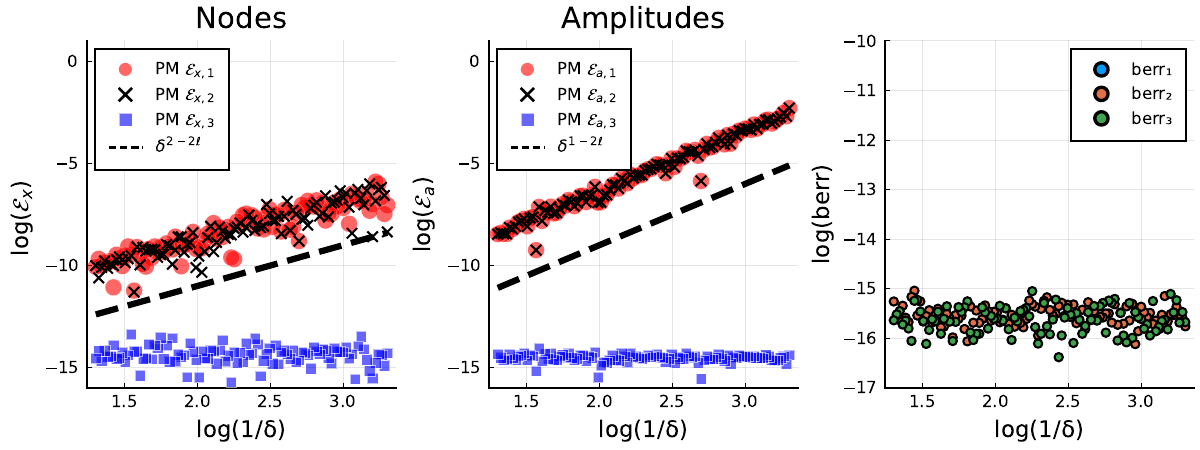}
    \caption{{\small Classical Prony method - asymptotic optimality. For cluster nodes $j=1,2$, the node errors $\mathcal{E}_{x,j}=|\tilde{x}_j-x_j|$ (left) scale is like $\delta^{2-2\ell}$, while the amplitude errors $\mathcal{E}_{a,j}=|\tilde{\alpha}_j-\alpha_j|$ (middle) scale like $\delta^{1-2\ell}$. For the non-cluster node $j=3$, both errors are bounded by a constant. Right: backward errors of each step, as specified in \prettyref{def:backward-errors}, are on the order of machine epsilon, implying numerical stability of PM according to \prettyref{thm:finite-precision}. Here $\epsilon=10^{-15}$ in all experiments.}}
    \label{fig:vanilla-nonproj-DP-backward}
\end{figure}

\begin{figure}
    \includegraphics[width=0.9\linewidth]{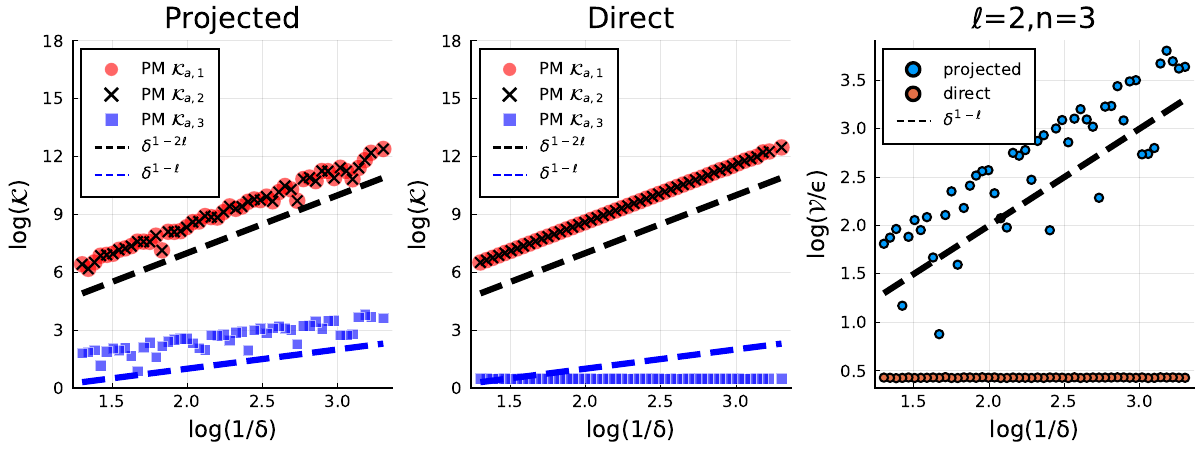}
    \caption{Prony's method: accuracy of amplitude recovery where the nodes are projected (left) or non-projected (center) prior to recovering the amplitudes. Right panel: the corresponding normalized amplitude discrepancy function $\discr_1/\epsilon$ (see text). The results are consistent with the estimates \eqref{eq:TildVVBdClus} (projected) and Theorem~\ref{Thm:ImprovedClust} (non-projected).}
    \label{fig:discrepancy-DP}
\end{figure}

Interestingly, the cancellation phenomenon just demonstrated appears in other methods for SR which are based on decoupling the recovery of $\{x_j\}_{j=1}^n$ from that of $\{\alpha_j\}_{j=1}^n$. Performing the same projection perturbation, the loss of accuracy can be seen for Decimated Prony method (DPM, \prettyref{fig:DPM-merged}) and also the Matrix Pencil (\prettyref{fig:MP-merged}). In all the above, if we project the nodes before computing the amplitudes, then the amplitudes accuracy will deteriorate according to the estimate \eqref{eq:alphaeClust}. {\bf Thus, non-projecting the nodes is crucial for maintaining the accuracy of non-cluster amplitudes.} While this phenomenon is perhaps to be expected for DPM which still relies on PM, Matrix Pencil entails an eigenvalue decomposition, and therefore it is a-priori, not obvious that cancellations should occur also there. We believe our insights can help towards a complete analysis of Matrix Pencil and related methods in the multi-cluster geometry.

\begin{figure}
    \includegraphics[width=0.98\linewidth]{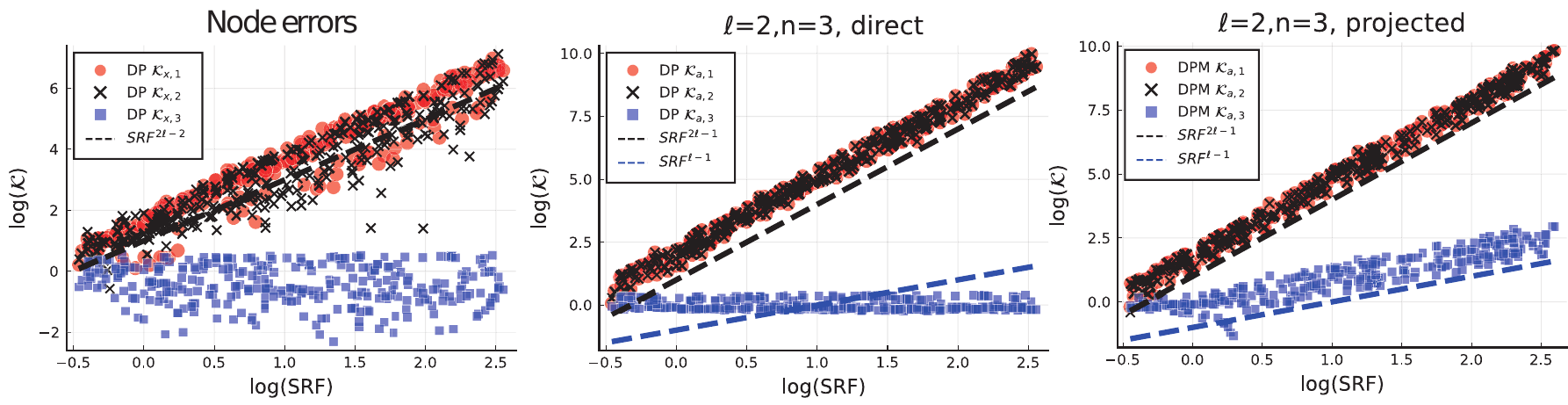}
    \caption{\small Decimated Prony Method - asymptotic optimality. For cluster nodes $j=1,2$, the node amplification factors $\mathcal{K}_x$ (left) scale like $\srf^{2\ell-2}$, and for the non-cluster node ($j=3$) it is bounded by a constant. The amplitude error amplification factors $\mathcal{K}_a$ for the non-cluster node with no projection (middle) are bounded by a constant while with projection (right) they scale like $\srf^{\ell-1}$. For the cluster nodes, both amplitude error amplification factors scale like $\srf^{2\ell-1}$.}
    \label{fig:DPM-merged}
\end{figure}

\begin{figure}
    \includegraphics[width=0.98\linewidth]{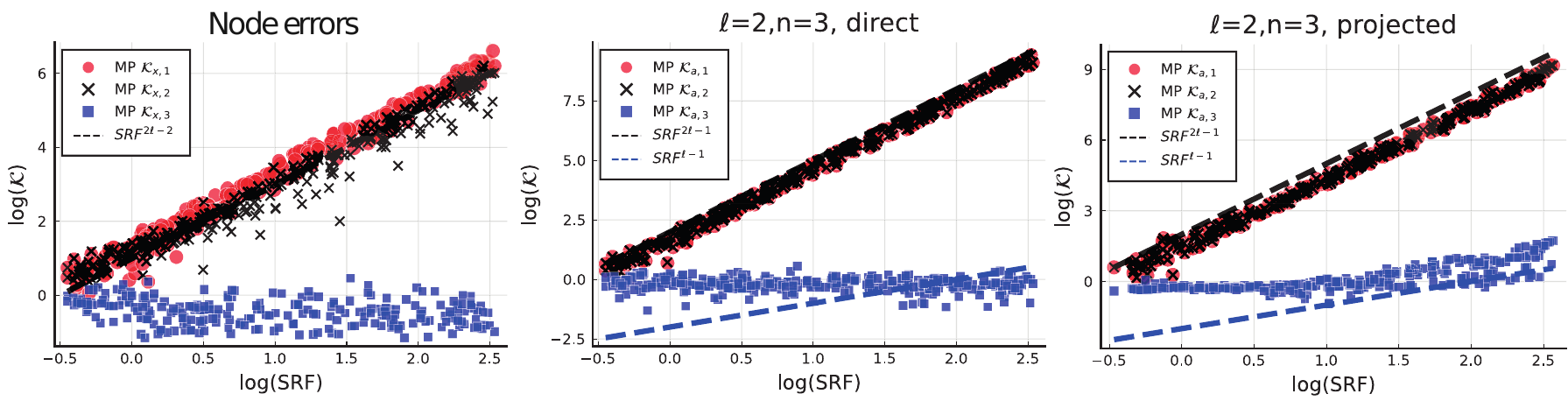}
    \caption{\small  Matrix Pencil's asymptotic optimality. For cluster nodes, the node amplification factors $\mathcal{K}_x$ (left) scale like $\srf^{2\ell-2}$, and for the non-cluster node (j=3) it is bounded by a constant. The amplitude error amplification factors $\mathcal{K}_a$ for the non-cluster node with no projection (middle) are bounded by a constant while with projection (right) they scale like $\srf^{\ell-1}$ for large enough $\srf$. For the cluster nodes, both amplitude error amplification factors scale like $\srf^{2\ell-1}$.}
    \label{fig:MP-merged}
\end{figure}

\newpage
\printbibliography

\appendix

\section{Proofs - Section \ref{Sec:Prelims}}

\subsection{Proof of Lemma \ref{Lem:HomogPronyPoly}}\label{Lem:HomogPronyPolyPf}

Expanding \eqref{eq:HomogPronyPol} via the first row we have 
\begin{align}
& \bar{p}(z) = (-1)^n\operatorname{det}\left(H_n \right)\left(z^n+(-1)z^{n-1}\operatorname{det}\left(H_n \right)^{-1} \operatorname{det}\begin{bmatrix}
m_0 & m_1 & \dots &m_{n-2}& m_n\\
\vdots & \vdots & \vdots & \vdots & \vdots\\
m_{n-1} & m_n & \dots & m_{2n-3} & m_{2n-1}
 \end{bmatrix}\right.\label{eq:pbar}\\
 &\hspace{12mm} \left.+\dots +(-1)^n \operatorname{det}\left(H_n \right)^{-1}\begin{bmatrix}
m_1 & m_2 & \dots &m_{n}\\
\vdots & \vdots & \vdots & \vdots\\
m_{n} & m_{n+1} & \dots & m_{2n-1}
 \end{bmatrix}\right)\nonumber.
\end{align}
Recall the system \eqref{eq:UnpertHanekl}. Applying Craemer's rule we get 
\begin{align}
& p_0 = (-1)^n\operatorname{det}\left(H_n \right)^{-1}\operatorname{det}\begin{bmatrix}
m_1 & m_1 & \dots &m_{n}\\
\vdots & \vdots & \vdots & \vdots\\
m_{n} & m_n & \dots & m_{2n-1}  \end{bmatrix},\nonumber \\
& \hspace{5cm} \vdots \label{eq:Craemer}\\
&p_{n-1} = (-1)\operatorname{det}\left(H_n \right)^{-1}\operatorname{det}\begin{bmatrix}
m_0 & m_1 & \dots &m_{n-2}& m_n\\
\vdots & \vdots & \vdots & \vdots & \vdots\\
m_{n-1} & m_n & \dots & m_{2n-3} & m_{2n-1}
 \end{bmatrix}.\nonumber
\end{align}
Substituting \eqref{eq:Craemer} into \eqref{eq:pbar} leads to \eqref{eq:HomogPronyPol1}. Identical computation for $\bar{q}(z)$ shows \eqref{eq:HomogPronyPol2}.

\subsection{Proof of Lemma \ref{lem:qbarAsymp}}\label{lem:qbarAsympPf}
{ Recall that the coefficient of $\epsilon^{\kappa}$ in \eqref{eq:CompundExpansion} is given by
\begin{equation*}
\theta_{n+1-\kappa}(z) = \operatorname{tr}\left( \operatorname{adj}_{n+1-\kappa}(D) C_{n+1-\kappa}(G(z))\right).
\end{equation*}
In the case $\kappa = 0$, we have $C_{n+1}(G(z)) = \operatorname{det}\left(G(z) \right)= \bar{p}(z)$, whereas $D\in \mathbb{R}^{(n+1)\times (n+1)}$ implies, by definition, that $\operatorname{adj}_{n+1}(D)=1$. Therefore, $\theta_{n+1}(z)=\bar{p}(z)$. In the case $\kappa=n+1$, we have, by definition, $\operatorname{adj}_{0}(D)=\operatorname{det}(D)$ and $C_0(G(z)) = 1$, whence $\theta_0(z) = \operatorname{det}(D)$, as claimed.

Consider now the case $1\leq \kappa \leq n$. By definition, we have 
\begin{equation*}
\begin{array}{lll}
&\operatorname{tr}\left( \operatorname{adj}_{n+1-\kappa}(D) C_{n+1-\kappa}(G(z))\right) \vspace{0.1cm}\\
&\hspace{10mm}= \sum_{\gamma \in \mathcal{Q}^{n+1}_{n+1-\kappa}}\sum_{\beta \in \mathcal{Q}^{n+1}_{n+1-\kappa}}\left[\operatorname{adj}_{n+1-\kappa}(D) \right]_{\gamma, \beta} \left[C_{n+1-\kappa}(G(z)) \right]_{\beta, \gamma}.
\end{array}
\end{equation*}
Fix any $\gamma,\beta \in \mathcal{Q}^{n+1}_{n+1-\kappa}$. Then,
\begin{equation*}
\left[\operatorname{adj}_{n+1-\kappa}(D) \right]_{\gamma, \beta} = (-1)^{\chi(\gamma,\beta)}\operatorname{det}\left(D\left[\beta^c|\gamma^c \right]\right).
\end{equation*}
where $\chi(\gamma,\beta) = \sum_{i\in \gamma}i+\sum_{j \in \beta}j$ and $D\left[\beta^c|\gamma^c \right]$ is the matrix obtained from $D$ by selecting the rows \emph{not} in $\beta$ and columns \emph{not} in $\gamma$. Recall the structure of $D$ in \eqref{PronPert1} and note that $1\notin \beta$ implies that $D\left[\beta^c|\gamma^c \right]$ has a first row of zeros, whence $\operatorname{det}\left(D\left[\beta^c|\gamma^c \right]\right)=0$. Therefore, in suffices to consider only $\beta \in \mathcal{Q}_{n+1-\kappa}$ such that $1\in \beta$. In the case $\kappa = n$, we thus obtain $\beta = (1)$ and $\gamma = (i), \ 1\leq i\leq n+1$, whence 
\begin{equation*}
\begin{array}{lll}
\theta_1(z) &= \sum_{i=1}^{n+1}\left[\operatorname{adj}_{1}(D) \right]_{i, 1} \left[C_{1}(G(z)) \right]_{1, i}= \sum_{i=1}^{n+1}\left[\operatorname{adj}(D) \right]_{i, 1}\left[G(z) \right]_{1, i}\\
& = \sum_{i=1}^{n+1}\left[\operatorname{adj}(D) \right]_{i, 1}z^{i-1}
\end{array}
\end{equation*}
as required. Finally, consider the case $1\leq \kappa \leq n-1$. Recall the notation \eqref{eq:GammaParam} corresponding to each $ \gamma = \left(i_1,\dots,i_{n+1-\kappa} \right)$ and $ \beta = \left(1,j_1,\dots,j_{n-\kappa} \right)$. By definition of the multiplicative compound {\color{blue} and the notations \eqref{eq:GammaParam}}
\begin{equation*}
\begin{array}{lll}      
&\left[C_{n+1-\kappa}(G(z)) \right]_{\beta, \gamma} = \det \begin{bmatrix}
[G(z)]_{1,i_1} & [G(z)]_{1,i_2} & \ldots & [G(z)]_{1,i_{n+1-\kappa}} \\
[G(z)]_{j_1,i_1} & [G(z)]_{j_1,i_2} & \ldots & [G(z)]_{j_1,i_{n+1-\kappa}} \\
\vdots\\
[G(z)]_{j_{n-\kappa},i_1} & [G(z)]_{j_{n-\kappa},i_2} & \ldots & [G(z)]_{j_{n-\kappa},i_{n+1-\kappa}}
\end{bmatrix} \vspace{0.1cm} \\
& = \det \begin{bmatrix}
z^{i_1-1} & z^{i_2-1} & \ldots & z^{i_{n+\kappa-1}-1} \\
m_{(j_1-2)+(i_1-1)} & m_{(j_1-2)+(i_2-1)} & \ldots & m_{(j_1-2)+(i_{n+1-\kappa}-1)} \\
\vdots\\
m_{(j_{n-\kappa}-2)+(i_1-1)} & m_{(j_{n-\kappa}-2)+(i_2-1)} & \ldots & m_{(j_{n-\kappa}-2)+(i_{n+1-\kappa}-1)}
\end{bmatrix}  \vspace{0.1cm} \\
& = \det \begin{bmatrix}
z^{\mathrm{a}} & z^{\mathrm{a}+\mathrm{k}_1} & \ldots & z^{\mathrm{a}+\mathrm{k}_{n-\kappa}} \\
m_{\mathrm{b}} & m_{\mathrm{b}+\mathrm{k}_1} & \ldots & m_{\mathrm{b}+\mathrm{k}_{n-\kappa}} \\
m_{\mathrm{b}+\mathrm{l}_1} & m_{\mathrm{b}+\mathrm{l}_1+\mathrm{k}_1} & \ldots & m_{\mathrm{b}+\mathrm{l}_1+\mathrm{k}_{n-\kappa}} \\
\vdots\\
m_{\mathrm{b}+\mathrm{l}_{n-\kappa-1}} & m_{\mathrm{b}+\mathrm{l}_{n-\kappa-1}+\mathrm{k}_{1}} & \ldots & m_{\mathrm{b}+\mathrm{l}_{n-\kappa-1}+\mathrm{k}_{n-\kappa}}
\end{bmatrix}=\Gamma_{\beta,\gamma}(z),
\end{array}
\end{equation*}
whence
\begin{equation*}
\theta_{n+1-\kappa}(z) = \sum_{\gamma \in \mathcal{Q}^{n+1}_{n+1-\kappa}}\sum_{\beta \in \mathcal{Q}^{n+1}_{n+1-\kappa}: 1\in \beta}\left[\operatorname{adj}_{n+1-\kappa}(D) \right]_{\gamma, \beta} \Gamma_{\beta,\gamma}(z)
\end{equation*}
as desired.}

\subsection{Proof of Theorem \ref{Lem:GammaEval}}\label{lem:GammBetAlphExp}
{ Recall that the algebraic moments are given by  
\begin{equation*}
m_{s} = \sum_{j=1}^n \alpha_jx_j^{s}.
\end{equation*}
Employing the representation of $\Gamma_{\beta,\gamma}$ in Lemma \ref{lem:qbarAsymp} we have 
\begin{equation*}
\begin{array}{lll}
&\Gamma_{\beta,\gamma}(z) =\operatorname{det} \begin{bmatrix}
z^{\mathrm{a}} & z^{\mathrm{a}+\mathrm{k}_1} & \ldots & z^{\mathrm{a}+\mathrm{k}_{n-\kappa}} \\
m_{\mathrm{b}} & m_{\mathrm{b}+\mathrm{k}_1} & \ldots & m_{\mathrm{b}+\mathrm{k}_{n-\kappa}} \\
m_{\mathrm{b}+\mathrm{l}_1} & m_{\mathrm{b}+\mathrm{l}_1+\mathrm{k}_1} & \ldots & m_{\mathrm{b}+\mathrm{l}_1+\mathrm{k}_{n-\kappa}} \\
\vdots\\
m_{\mathrm{b}+\mathrm{l}_{n-\kappa-1}} & m_{\mathrm{b}+\mathrm{l}_{n-\kappa-1}+\mathrm{k}_{1}} & \ldots & m_{\mathrm{b}+\mathrm{l}_{n-\kappa-1}+\mathrm{k}_{n-\kappa}}
\end{bmatrix}\vspace{0.1cm}\\
& = \operatorname{det}\begin{bmatrix}
z^{\mathrm{a}} & z^{\mathrm{a}+\mathrm{k}_1} & \ldots & z^{\mathrm{a}+\mathrm{k}_{n-\kappa}} \\
\sum_{\omega=1}^n \alpha_{\omega} x_{\omega}^{\mathrm{b}} & \sum_{\omega=1}^n \alpha_{\omega}x_{\omega}^{\mathrm{b}+\mathrm{k}_1} & \ldots & \sum_{\omega=1}^n \alpha_{\omega}x_{\omega}^{\mathrm{b}+\mathrm{k}_{n-\kappa}} \\
\sum_{\omega=1}^n \alpha_{\omega}x_{\omega}^{\mathrm{b}+\mathrm{l}_1} & \sum_{\omega=1}^n \alpha_{\omega}x_{\omega}^{\mathrm{b}+\mathrm{l}_1+\mathrm{k}_1} & \ldots & \sum_{\omega=1}^n \alpha_{\omega}x_{\omega}^{\mathrm{b}+\mathrm{l}_1+\mathrm{k}_{n-\kappa}} \\
\vdots\\
\sum_{\omega=1}^n \alpha_{\omega}x_{\omega}^{\mathrm{b}+\mathrm{l}_{n-\kappa-1}} & \sum_{\omega=1}^n \alpha_{\omega}x_{\omega}^{\mathrm{b}+\mathrm{l}_{n-\kappa-1}+\mathrm{k}_1}  & \ldots & \sum_{\omega=1}^n \alpha_{\omega}x_{\omega}^{\mathrm{b}+\mathrm{l}_{n-\kappa-1}+\mathrm{k}_{n-\kappa}}
\end{bmatrix}\vspace{0.1cm}\\
&=\sum_{\omega_1=1}^n\alpha_{\omega_1} \operatorname{det}\begin{bmatrix}
z^{\mathrm{a}} & z^{\mathrm{a}+\mathrm{k}_1} & \ldots & z^{\mathrm{a}+\mathrm{k}_{n-\kappa}} \\
x_{\omega_1}^{\mathrm{b}} & x_{\omega_1}^{\mathrm{b}+\mathrm{k}_1} & \ldots & x_{\omega_1}^{\mathrm{b}+\mathrm{k}_{n-\kappa}} \\
\sum_{\omega=1}^n \alpha_{\omega}x_{\omega}^{\mathrm{b}+\mathrm{l}_1} & \sum_{\omega=1}^n \alpha_{\omega}x_{\omega}^{\mathrm{b}+\mathrm{l}_1+\mathrm{k}_1} & \ldots & \sum_{\omega=1}^n \alpha_{\omega}x_{\omega}^{\mathrm{b}+\mathrm{l}_1+\mathrm{k}_{n-\kappa}} \\
\vdots\\
\sum_{\omega=1}^n \alpha_{\omega}x_{\omega}^{\mathrm{b}+\mathrm{l}_{n-\kappa-1}} & \sum_{\omega=1}^n \alpha_{\omega}x_{\omega}^{\mathrm{b}+\mathrm{l}_{n-\kappa-1}+\mathrm{k}_1}  & \ldots & \sum_{\omega=1}^n \alpha_{\omega}x_{\omega}^{\mathrm{b}+\mathrm{l}_{n-\kappa-1}+\mathrm{k}_{n-\kappa}}
\end{bmatrix}\vspace{0.1cm}\\
&=\dots =\sum_{\omega_1,\dots,\omega_{n-\kappa}=1}^n\left(\prod_{s=1}^{n-\kappa}\alpha_{\omega_s} \right)\operatorname{det}\begin{bmatrix}
z^{\mathrm{a}} & z^{\mathrm{a}+\mathrm{k}_1} & \ldots & z^{\mathrm{a}+\mathrm{k}_{n-\kappa}} \\
x_{\omega_1}^{\mathrm{b}} & x_{\omega_1}^{\mathrm{b}+\mathrm{k}_1} & \ldots & x_{\omega_1}^{\mathrm{b}+\mathrm{k}_{n-\kappa}} \\
x_{\omega_2}^{\mathrm{b}+\mathrm{l}_1} & x_{\omega_2}^{\mathrm{b}+\mathrm{l}_1+\mathrm{k}_1} & \ldots & x_{\omega_2}^{\mathrm{b}+\mathrm{l}_1+\mathrm{k}_{n-\kappa}} \\
\vdots\\
x_{\omega_{n-\kappa}}^{\mathrm{b}+\mathrm{l}_{n-\kappa-1}} & x_{\omega_{n-\kappa}}^{\mathrm{b}+\mathrm{l}_{n-\kappa-1}+\mathrm{k}_1}  & \ldots & x_{\omega_{n-\kappa}}^{\mathrm{b}+\mathrm{l}_{n-\kappa-1}+\mathrm{k}_{n-\kappa}}
\end{bmatrix}.
\end{array}
\end{equation*}
Consider a single summand in the latter sum. We have
\begin{equation*}
\begin{array}{lll}
&\left(\prod_{s=1}^{n-\kappa}\alpha_{\omega_s}\right) \operatorname{det}\begin{bmatrix}
z^{\mathrm{a}} & z^{\mathrm{a}+\mathrm{k}_1} & \ldots & z^{\mathrm{a}+\mathrm{k}_{n-\kappa}} \\
x_{\omega_1}^{\mathrm{b}} & x_{\omega_1}^{\mathrm{b}+\mathrm{k}_1} & \ldots & x_{\omega_1}^{\mathrm{b}+\mathrm{k}_{n-\kappa}} \\
x_{\omega_2}^{\mathrm{b}+\mathrm{l}_1} & x_{\omega_2}^{\mathrm{b}+\mathrm{l}_1+\mathrm{k}_1} & \ldots & x_{\omega_2}^{\mathrm{b}+\mathrm{l}_1+\mathrm{k}_{n-\kappa}} \\
\vdots\\
x_{\omega_{n-\kappa}}^{\mathrm{b}+\mathrm{l}_{n-\kappa-1}} & x_{\omega_{n-\kappa}}^{\mathrm{b}+\mathrm{l}_{n-\kappa-1}+\mathrm{k}_1}  & \ldots & x_{\omega_{n-\kappa}}^{\mathrm{b}+\mathrm{l}_{n-\kappa-1}+\mathrm{k}_{n-\kappa}}
\end{bmatrix}\vspace{0.1cm}\\
& = \left(\prod_{s=1}^{n-\kappa}\alpha_{\omega_s} \right) z^{\mathrm{a}}x_{\omega_1}^{\mathrm{b}}x_{\omega_2}^{\mathrm{b}+\mathrm{l}_1}\dots x_{\omega_{n-\kappa}}^{\mathrm{b}+\mathrm{l}_{n-\kappa-1}}\operatorname{det}\begin{bmatrix}
1 & z^{\mathrm{k}_1} & \ldots & z^{\mathrm{k}_{n-\kappa}} \\
1 & x_{\omega_1}^{\mathrm{k}_1} & \ldots & x_{\omega_1}^{\mathrm{k}_{n-\kappa}} \\
1 & x_{\omega_2}^{\mathrm{k}_1} & \ldots & x_{\omega_2}^{\mathrm{k}_{n-\kappa}} \\
\vdots\\
1 & x_{\omega_{n-\kappa}}^{\mathrm{k}_1}  & \ldots & x_{\omega_{n-\kappa}}^{\mathrm{k}_{n-\kappa}}
\end{bmatrix}.
\end{array}
\end{equation*}
Recall that \eqref{eq:GammaParam} implies that 
\begin{equation*}
1\leq \mathrm{k}_1<\dots<\mathrm{k}_{n-\kappa}.
\end{equation*}
Setting 
\begin{equation*}
\lambda_{1}=k_{n-\kappa}-n+\kappa, \dots, \lambda_{n-\kappa}=k_1-1, \lambda_{n-\kappa+1} = 0
\end{equation*}
we see that $0\leq \lambda_{n-\kappa+1}\leq \lambda_{n-\kappa}\leq \dots \leq \lambda_1$. Thus, considering $\underline{\lambda}:=(\lambda_1,\dots,\lambda_{n-\kappa+1})$ as an integer partition \cite[Chapter 1]{macdonald1998symmetric}, we have that the considered summand equals 
\begin{equation}\label{eq:Summand}
\begin{array}{lll}
&\left(\prod_{s=1}^{n-\kappa}\alpha_{\omega_s} \right)z^{\mathrm{a}}x_{\omega_1}^{\mathrm{b}}x_{\omega_2}^{\mathrm{b}+\mathrm{l}_1}\dots x_{\omega_{n-\kappa}}^{\mathrm{b}+\mathrm{l}_{n-\kappa-1}}\left(\prod_{s=1}^{n-\kappa} (x_{\omega_s}-z) \right)\\
&\hspace{15mm}\times \left(\prod_{1\leq s<t\leq n-\kappa } (x_{\omega_t}-x_{\omega_s}) \right)s_{\underline{\lambda}}\left(z,x_{\omega_1},\dots,x_{\omega_{n-\kappa}} \right)
\end{array}
\end{equation}
where $s_{\underline{\lambda}}\left(z,x_{\omega_1},\dots,x_{\omega_{n-\kappa}} \right)$ is the Schur polynomial for the partition $\underline{\lambda}$ and in the variables $z,x_{\omega_1},\dots,x_{\omega_{n-\kappa}}$ \cite[Chapter 3]{macdonald1998symmetric}. Note that the use of Schur polynomials allowed us to extract the Vandermonde determinant on $\left\{z,x_{\omega_1},\dots, x_{\omega_{n-\kappa}} \right\}$. To obtain an additional copy of the Vandermonde determinant, we sum over the terms \eqref{eq:Summand} to obtain the following representation of $\Gamma_{\beta,\gamma}$
\begin{equation*}
\begin{array}{lll}
\Gamma_{\beta,\gamma}(z)&=\sum_{\omega_1,\dots,\omega_{n-\kappa}=1}^n\left\{\left(\prod_{s=1}^{n-\kappa}\alpha_{\omega_s} \right)z^{\mathrm{a}}x_{\omega_1}^{\mathrm{b}}x_{\omega_2}^{\mathrm{b}+\mathrm{l}_1}\dots x_{\omega_{n-\kappa}}^{\mathrm{b}+\mathrm{l}_{n-\kappa-1}}\right.\vspace{0.1cm} \\
&\left.\times\left(\prod_{s=1}^{n-\kappa} (x_{\omega_s}-z) \right) \left(\prod_{1\leq s<t\leq n-\kappa } (x_{\omega_t}-x_{\omega_s}) \right)s_{\underline{\lambda}}\left(z,x_{\omega_1},\dots,x_{\omega_{n-\kappa}} \right)\right\}
\end{array}
\end{equation*}
Considering the latter representation we see that for selection of $\omega_1,\dots,\omega_{n-\kappa}\in [n]$ such that there exist some $1\leq s<t\leq n-\kappa$ with $\omega_s=\omega_t$ the corresponding summand is zero. Hence, taking the symmetry of Schur polynomials \cite[Chapter 3]{macdonald1998symmetric} into account, we obtain  
\begin{equation}\label{eq:gamma-beta-gamma-explicit}
\begin{array}{lll}
\Gamma_{\beta,\gamma}(z)&=\sum_{(\omega_1,\dots,\omega_{n-\kappa})\in \mathcal{Q}^{n}_{n-\kappa}}\left\{\left(\prod_{s=1}^{n-\kappa}\alpha_{\omega_s} \right)\left(\prod_{s=1}^{n-\kappa} (x_{\omega_s}-z) \right)\right. \vspace{0.1cm}\\
&\hspace{5mm}\left.\times \left(\prod_{1\leq s<t\leq n-\kappa } (x_{\omega_t}-x_{\omega_s}) \right)s_{\underline{\lambda}}\left(z,x_{\omega_1},\dots,x_{\omega_{n-\kappa}} \right) z^{\mathrm{a}}\psi_{\underline{\lambda},\omega_1,\dots,\omega_{n-\kappa}}\right\}
\end{array}
\end{equation}
{\color{blue} where
\begin{equation}\label{eq:psilambdDef}
\begin{array}{lll}
\psi_{\underline{\lambda},\omega_1,\dots,\omega_{n-\kappa}} &=  \sum_{\sigma \in S_{n-\kappa}}(-1)^{\operatorname{sgn}(\sigma)}x_{\omega_{\sigma(1)}}^{\mathrm{b}}x_{\omega_{\sigma(2)}}^{\mathrm{b}+\mathrm{l}_1}\dots x_{\omega_{\sigma(n-\kappa)}}^{\mathrm{b}+\mathrm{l}_{n-\kappa-1}}\vspace{0.1cm}\\
& = \operatorname{det} \begin{bmatrix}
x_{\omega_1}^{\mathrm{b}} & \ldots & x_{\omega_{n-\kappa}}^{\mathrm{b}} \\
x_{\omega_1}^{\mathrm{b}+\mathrm{l}_1} & \ldots & x_{\omega_{n-\kappa}}^{\mathrm{b}+\mathrm{l}_1} \\
\vdots\\
x_{\omega_{1}}^{\mathrm{b}+\mathrm{l}_{n-\kappa-1}}  & \ldots & x_{\omega_{n-\kappa}}^{\mathrm{b}+\mathrm{l}_{n-\kappa-1}}
\end{bmatrix} =\left(\prod_{j=1}^{n-\kappa}x_{\omega_j}^b \right) \operatorname{det}\begin{bmatrix}
1 & \ldots & 1 \\
x_{\omega_1}^{\mathrm{l}_1} & \ldots & x_{\omega_{n-\kappa}}^{\mathrm{l}_1} \\
\vdots\\
x_{\omega_{1}}^{\mathrm{l}_{n-\kappa-1}}  & \ldots & x_{\omega_{n-\kappa}}^{\mathrm{l}_{n-\kappa-1}}
\end{bmatrix}
\end{array}
\end{equation}}
and $S_{n-\kappa}$ denotes the symmetric group on $[n-\kappa]$. {\color{blue} In particular, if there exist $1\leq s<t\leq n-\kappa$ such that $x_{\omega_s}=x_{\omega_t}$, then $\psi_{\underline{\lambda},\omega_1,\dots,\omega_{n-\kappa}}=0$. Thus, $\psi_{\underline{\lambda},\omega_1,\dots,\omega_{n-\kappa}}$ is divisible by $\prod_{1\leq s<t\leq n-\kappa}(x_{\omega_t}-x_{\omega_s})$. In fact, introducing the partition
\begin{equation*}
    \mu_{n-\kappa} = 0,\ \mu_{n-\kappa -1 }=\mathrm{l}_1-1 ,\ \dots \ ,  \mu_1 = \mathrm{l}_{n-\kappa-1}- (n-\kappa-1)
\end{equation*}
we have that $\mu_1\geq \dots \geq \mu_{n-\kappa}$, and the integer partition $\underline{\mu}:=\left(\mu_1,\dots,\mu_{n-\kappa} \right)$ satisfies
\begin{equation*}
\begin{array}{lll}
& \operatorname{det}\begin{bmatrix}
1 & \ldots & 1 \\
x_{\omega_1}^{\mathrm{l}_1} & \ldots & x_{\omega_{n-\kappa}}^{\mathrm{l}_1} \\
\vdots\\
x_{\omega_{1}}^{\mathrm{l}_{n-\kappa-1}}  & \ldots & x_{\omega_{n-\kappa}}^{\mathrm{l}_{n-\kappa-1}}
\end{bmatrix} = s_{\underline{\mu}}(x_{\omega_1},\dots, x_{\omega_{n-\kappa}})\cdot \left(\prod_{1\leq s<t\leq n-\kappa }(x_{\omega_t}-x_{\omega_s})\right),
\end{array}
\end{equation*}}
whence we conclude that 
\begin{equation*}
\begin{array}{lll}
\Gamma_{\beta,\gamma}(z)&=\sum_{(\omega_1,\dots,\omega_{n-\kappa})\in \mathcal{Q}^n_{n-\kappa}}\left\{\left(\prod_{s=1}^{n-\kappa} (x_{\omega_s}-z) \right)\right. \vspace{0.1cm}\\
&\times \left.\left(\prod_{1\leq s<t\leq n-\kappa } (x_{\omega_t}-x_{\omega_s})^2 \right)\phi^{\beta,\gamma}_{(\omega_1,\dots,\omega_{n-\kappa})}(z)\right\}
\end{array}
\end{equation*}
{\color{blue} where 
\begin{equation}\label{eq:phi-explicit}
 \begin{array}{lll}
&\phi^{\beta,\gamma}_{(\omega_1,\dots,\omega_{n-\kappa})}(z) = \left(\prod_{s=1}^{n-\kappa}\alpha_{\omega_s} \right)s_{\underline{\lambda}}\left(z,x_{\omega_1},\dots,x_{\omega_{n-\kappa}} \right) s_{\underline{\mu}}(x_{\omega_1},\dots, x_{\omega_{n-\kappa}}) {z^{\textrm{a}} }\left(\prod_{j=1}^{n-\kappa}x_{\omega_j}^b \right)
 \end{array}
\end{equation}}
is a polynomial in $\left\{ z,w_{\omega_1},\dots,x_{\omega_{n-\kappa}}\right\}$.
}

\section{Proofs - Section~\ref{Sec:PfMainthm}}\label{thm:node-accuracy-DimaPf}

\subsection{Proof of Lemma \ref{lem:PcalRecur}}\label{lem:PcalRecurPf}
{ We have
\begin{equation*}
\begin{array}{lll}
\mathcal{P}_{(\omega_1,\dots,\omega_{n-\kappa})} = |x_{\omega_s}-z|\left( \prod_{1\leq t\leq n-\kappa, t\neq s}|x_{\omega_t}-x_{\omega_s}|^2\right)\mathcal{P}_{(\omega_1,\dots,\omega_{s-1},\omega_{s+1},\dots,\omega_{n-\kappa})}
\end{array}
\end{equation*}
We consider three cases. First, let $\omega_s = j_*$. Then, {\color{blue} recalling that $x_{j_*}$ has $j_* \in \mathcal{C}_{\mu}$ with $\ell_{\mu}$ nodes,}
\begin{equation*}
|x_{\omega_s}-z|\left( \prod_{1\leq t\leq n-\kappa, t\neq s}|x_{\omega_t}-x_{\omega_s}|^2\right) =  \rho_* \left( \prod_{1\leq t\leq n-\kappa, \omega_t\neq j_*}|x_{\omega_t}-x_{j_*}|^2\right)\gtrapprox \rho_* \delta^{2(\ell_{\mu}-1)}.  
\end{equation*}
{\color{blue} Indeed, when choosing elements $x_{\omega_t}$ so that the product is smallest, we can either choose $\omega_t\in \mathcal{C}_{\mu}\setminus \left\{j_* \right\}$, which contributes  $|x_{\omega_t}-x_{j_*}|^2 = \delta^2<1$ to the product, or we can choose $\omega_t\notin \mathcal{C}_{\mu}$, which contributes $|x_{\omega_t}-x_{j_*}|^2 \geq T^2 = O(1)$. Thus, the smallest product obtainable is achieved by choosing \emph{all} indices in $\mathcal{C}_{\mu}\setminus \left\{j_* \right\}$.
}

If $\omega_s\in \mathcal{C}_{\mu}\setminus \left\{j_* \right\}$ then {\color{blue}similar reasoning leads to}
\begin{equation*}
|x_{\omega_s}-z|\left( \prod_{1\leq t\leq n-\kappa, t\neq s}|x_{\omega_t}-x_{\omega_s}|^2\right) \gtrapprox ( \delta - \rho_*) \delta^{2(\ell_{\mu}-1)}. 
\end{equation*}
{\color{blue} Indeed, here $|x_{\omega_s}-z|\geq |x_{\omega_s}-x_{j_*}| - |z-x_{j_*}|\geq \delta- \rho_*$, whereas the second term follows by the same arguments as above.}

Finally, if $\omega_s\notin \mathcal{C}_{\mu}$ then 
\begin{equation*}
|x_{\omega_s}-z|\left( \prod_{1\leq t\leq n-\kappa, t\neq s}|x_{\omega_t}-x_{\omega_s}|^2\right) \gtrapprox \left(T- \rho_* \right) 
\delta^{2(\ell_{*}-1)}.  
\end{equation*}
{\color{blue} Indeed, note that in this case $|x_{\omega_s}-z|\geq |x_{\omega_s}-x_{j_*}| - |z-x_{j_*}|\geq T- \rho_*$, whereas for the product, the smallest result is obtained when $x_{\omega_s}$ belongs to the largest cluster, whereas $\left\{x_{\omega_1} ,\dots x_{\omega_{s-1}}, x_{\omega_{s+1}},\dots, x_{\omega_{n-\kappa}}\right\}$ are chosen first as all remaining elements of this cluster, and the remaining elements in different clusters (the latter contribute $O(1)$ to the product).} This shows \eqref{eq:PcalDivision}. 

This shows \eqref{eq:PcalDivision}. Setting $\epsilon = \bar{\epsilon} \delta^{2\ell_*-1}$ and $\rho_* = \bar{\rho}_*\delta^{2(\ell_*-\ell_{\mu})+1}$, where $\bar{\epsilon}$ and $\bar{\rho}_*$ are independent of $\delta$, {\color{blue} and taking into account $\delta <1$ and $\bar{\rho}_*<\frac{1}{3}\min(1,T)$, it follows that 
\begin{equation*}
\begin{array}{lll}
&\frac{\rho_*\delta^{2(\ell_{\mu}-1)}}{\epsilon} = \frac{\bar{\rho}_*\delta^{2\ell_*-1}}{\bar{\epsilon}\delta^{2\ell_*-1}}= \frac{\bar{\rho}_*}{\bar{\epsilon}},\vspace{0.1cm}\\
&\frac{(\delta - \rho_*) \delta^{2(\ell_{\mu}-1)}}{\epsilon}\geq \left(\frac{1}{\delta} \right)^{2(\ell_*-\ell_{\mu})}\frac{1 - \bar{\rho}_*\delta^{2(\ell_* -\ell_{\mu})}}{\bar{\epsilon} }\geq \left(\frac{1}{\delta} \right)^{2(\ell_*-\ell_{\mu})}\frac{1 - \bar{\rho}_* }{\bar{\epsilon} }\geq  \frac{ \bar{\rho}_*}{\bar{\epsilon}},\vspace{0.1cm} \\ &\frac{( T - \rho_*) \delta^{2(\ell_{*}-1)}}{\epsilon}\geq \frac{1}{\delta} \frac{ T - \bar{\rho}_*}{\bar{\epsilon} }\geq \frac{  \bar{\rho}_*}{\bar{\epsilon}}.
\end{array}
\end{equation*}
Hence, the rightmost
term} is the smallest of the lower bounds which leads to \eqref{eq:PcalDiv}.
}

\subsection{Proof of Proposition \ref{prop:HighPowerBds}}\label{prop:HighPowerBdsPf}
{ 
Recall that $\epsilon = \bar{\epsilon}\delta^{2\ell_*-1}$  and let $m\leq n$. Consider the following function and set of constraints 
\begin{equation*}
\begin{array}{lll}
&g_{m,n}(\varpi_1,\dots,\varpi_m) = \sum_{s=1}^{m}\varpi_s^2-n,\\ 
&(\varpi_1,\dots,\varpi_{m})^{\top}\in \mathfrak{A}_{m,n}:= \left\{ y\in \mathbb{R}^m \ | \ \sum_{s=1}^{m}y_s = n,\ y_s\geq 1 \ \forall s\in [m]
\right\}.
\end{array}
\end{equation*}
Since $\mathfrak{A}_{m,n}$ is compact ($m-1$ dimensional simplex in $\mathbb{R}^{m}$), $g_m$ achieves its extrema in $\mathfrak{A}_{m,n}$. We begin by finding extremum candidates in the interior of $\mathfrak{A}_{m,n}$. Employing Largange multipliers with the auxiliary parameter $\eta$, we obtain the following first order conditions
\begin{equation*}
    \sum_{s=1}^m\varpi_s=n,\quad 2\varpi_s=\eta ,\ s\in [m],
\end{equation*}
whence we get the candidate $L_{m,n}^{\top}=\left(\frac{n}{m},\dots, \frac{n}{m}\right)$. Next, consider the boundary $\partial \mathfrak{A}_{m,n}$. It can be easily seen that it is characterized by \emph{exactly} $k$ of the entries of $(\ell_1,\dots, \ell_m)$ being equal to $1$ for some $1\leq k \leq m$. By symmetry we can assume that $\ell_{m-k+1}=\dots = \ell_{m}=1$, whereas $\ell_1,\dots,\ell_{m-k}>1$. Moreover, in this case $g_{m,n}(\ell_1,\dots,\ell_m) = g_{m-k,n-k}(\ell_1,\dots,\ell_{m-k})$ with $(\ell_1,\dots,\ell_{m-k})^{\top}$ in the relative interior of $\mathfrak{A}_{m-k,n-k}$. Employing the same arguments as above we get the candidate $L_{m-k,n-k}^{\top}=\left(\frac{n-k}{m-k},\dots, \frac{n-k}{m-k},1,\dots,1\right), \ k<m$ and $L_{0,0}^{\top}=(1,\dots,1),\ k= m=n$. Substitution into the function gives 
\begin{equation}\label{eq:gminmax0}
\begin{array}{lll}
g_{m,n}\left(L_{m-k,n-k}^{\top}\right) = \begin{cases}
(n-k)\frac{n-m}{m-k},\quad 0\leq k\leq m-1,\\
0,\quad k= m=n
\end{cases} .
\end{array}
\end{equation}
A simple comparison shows that
\begin{equation}\label{eq:gminmax}
0=g_{m,n}\left(L_{0,0}^{\top}\right)<g_{m,n}\left(L_{m-1,n-1}^{\top}\right)<\dots<g_{m,n}\left(L_{1,n-m+1}^{\top}\right).
\end{equation}
I.e., taking only a single $\ell_i$ to be different from $1$ maximizes $g_{m,n}$. 

Now, consider 
\begin{equation}\label{eq:Highestterms}
\begin{array}{lll}
\frac{1}{\bar{\epsilon}^n}\frac{\epsilon^{n}}{\delta^{n(n-1)+2\ell_*-\ell_{\mu}-2\varrho} } &\overset{\eqref{eq:xiDef11}}{=} \delta^{n(2\ell_*-1)-2\ell_*+\ell_{\mu}-\sum_{s\in [\zeta]}\ell_s(\ell_s-1)} \\
&= \delta^{n(2\ell_*-1)-2\ell_*+\ell_{\mu}-g_{\zeta,n}(\ell_1,\dots,\ell_{\zeta})}.
\end{array}
\end{equation}
{\color{blue} We now use the previous maximization of $g_{\zeta,n}$ to show that the power of $\delta$ on the right-hand side is non-negative, whence the right-hand side is upper bounded by 1 (recall that $\delta<1$).} For the case of $\zeta=n$ (i.e., when $(\ell_1,\dots,\ell_n)=L_{0,0}^T$), we have that $\ell_*=1$ and \eqref{eq:gminmax} yields $n(2\ell_*-1)-2\ell_*+\ell_{\mu}-g_{\zeta,n}(\ell_1,\dots,\ell_{\zeta}) = n-1\geq 1$, whence the result follows. Assume that we have $\zeta<n$ and \emph{exactly} $k<\zeta$ singleton clusters. In this case $\ell_* = \lceil \frac{n-k}{\zeta-k} \rceil$. Writing $\frac{n-k}{\zeta-k}=m+r$
where $m\in \mathbb{N}$ and $0\leq r<1$, we have 
\begin{equation*}
\begin{array}{llll}
n(2\ell_*-1)-2\ell_*+\ell_{\mu}-g_{\zeta,n}(\ell_1,\dots,\ell_{\zeta}) &\overset{\eqref{eq:gminmax0},\eqref{eq:gminmax}}{\geq}n(2m+1)-2m-2+\ell_{\mu}-(n-k)(m+r)\\
& = (n+k-2)m+n(1-r)+kr+\ell_{\mu}-2.
\end{array}
\end{equation*}
Since $n\geq 2$ and $k\geq 0$, if  $\ell_{\mu}\geq 2$, the latter expression is clearly nonnegative. If $\ell_{\mu}=1$, then we must have either $n\geq 3$ or $k\geq 1$, whence the expression is also nonnegative (indeed, $k=0$ and $n=2$ imply a clustered configuration of 2 nodes with no singleton clusters, which contradicts $\ell_{\mu}=1$). The proof is concluded.
} 

\subsection{Proof of Lemma \ref{lem:PcalEst}}\label{lem:PcalEstPf}

Fixing $(\omega_1,\dots,\omega_{n-1})\in \mathcal{Q}_{n-1}^n$ we consider a single summand of the form $\epsilon\mathcal{P}(\omega_1,\dots,\omega_{n-1})$. We derive an upper bound on 
\begin{equation}\label{eq:PcalDef1}
\epsilon\mathcal{P}_{(\omega_1,\dots,\omega_{n-1})}=\epsilon\left(\prod_{s=1}^{n-1} |x_{\omega_s}-z| \right)\left(\prod_{1\leq s<t\leq n-1 } |x_{\omega_t}-x_{\omega_s}|^2 \right)  
\end{equation}
up to constants that are \emph{independent of} $\delta$. Recall that $|z-x_{j_*}| = \rho_*$, where {\color{blue} $j_*\in \mathcal{C}_{\mu}$} with $\operatorname{card}(\mathcal{C}_{\mu})=\ell_{\mu}$. To upper bound $\prod_{s=1}^{n-1} |x_{\omega_s}-z|$, note that exactly one of the following holds:
\begin{itemize}
\item \underline{Case 1}: $\omega_s \neq j_*$ for all $s\in [n-1]$. In this case
\begin{align}
\prod_{s=1}^{n-1} |x_{\omega_s}-z| &=\left(\prod_{\omega_s\notin \mathcal{C}_{\mu}}|x_{\omega_s}-z| \right) \left( \prod_{\omega_s\in \mathcal{C}_{\mu}\setminus \left\{ j_*\right\}}|x_{\omega_s}-z|\right)\nonumber\\
&\leq \left(\rho_*+\eta T \right)^{n-\ell_{\mu}} \left(\rho_*+\tau \delta \right)^{\ell_{\mu}-1}.\label{eq:Case1}
\end{align}
\item \underline{Case 2}: $\omega_s\neq t$ for some $t\in \mathcal{C}_{\mu}\setminus \left\{j_* \right\}$. In this case
\begin{align}
\prod_{s=1}^{n-1}|x_{\omega_s}-z| &=\left(\prod_{\omega_s\notin \mathcal{C}_n}|x_{\omega_s}-z| \right) \left( \prod_{\omega_s\in \mathcal{C}_{\mu}\setminus \left\{ r\right\}}|x_{\omega_s}-z|\right)\nonumber\\
&\leq \left(\rho_*+\eta T \right)^{n-\ell_{\mu}}\rho_* \left(\rho_*+\tau \delta \right)^{\ell_{\mu}-2}.\label{eq:Case2}
\end{align}
\item \underline{Case 3}: $\omega_s\neq t$ for some $t\notin \mathcal{C}_{\mu}$. In this case
\begin{align}
\prod_{s=1}^{n-1}|x_{\omega_s}-z| &=\left(\prod_{\omega_s\notin \mathcal{C}_{\mu}\cup \left\{ t\right\}}|x_{\omega_s}-z| \right)  \left( \prod_{\omega_s\in \mathcal{C}_{\mu}}|x_{\omega_s}-z|\right)\nonumber \\
&\leq \left(\rho_*+\eta T \right)^{n-\ell_{\mu}-1}\rho_* \left(\rho_*+\tau \delta \right)^{\ell_{\mu}-1}.\label{eq:Case3}
\end{align}
\end{itemize}
To upper bound $\prod_{1\leq s<t\leq n-1 } |x_{\omega_t}-x_{\omega_s}|^2$, consider an arbitrary
$\iota \in [\zeta]$ and $b\in \mathcal{C}_{\iota}$. Then, 
\begin{equation}\label{eq:varpiiota}
\varpi_{\iota} = \binom{n-1}{2}-\sum_{s\in [\zeta]\setminus \left\{\iota \right\}}\binom{\ell_{s}}{2}-\binom{\ell_{\iota}-1}{2}
\end{equation}
is the number of ways to choose two nodes which differ from $x_b$ and \emph{do not belong} to the same cluster. Note that the subscript $\iota$ is added to take into account that the node $x_b$ with $b\in \mathcal{C}_{\iota}$ is removed from the nodal set. Let $1\leq s<t\leq n-1$ and consider the nodes $x_{\omega_t}$ and $x_{\omega_s}$. Exactly one of the following options holds: (1) $\omega_t,\omega_{s}\in \mathcal{C}_{\iota'}$ for some $\iota'\neq \iota$, (2) $\omega_t,\omega_s\in \mathcal{C}_{\iota}\setminus \left\{b \right\}$ and (3) $\omega_s\in \mathcal{C}_{\iota'},\ \omega_t\in \mathcal{C}_{\iota''}$ where $\iota',\iota''\in[\zeta],\ \iota'\neq \iota''$. Hence, for $\iota \in [\zeta]$ and $b\in \mathcal{C}_{\iota}$, the following bound is obtained
\begin{equation}\label{eq:ProdBoundCluster11}
\prod_{1\leq s<t\leq n-1 } |x_{\omega_t}-x_{\omega_s}|^2 \leq \left(\eta T \right)^{2\varpi_{\iota}}  \left(\tau \delta \right)^{(n-1)(n-2)-2\varpi_{\iota}} .
\end{equation}
Using \eqref{eq:PcalDef1}-\eqref{eq:ProdBoundCluster11}, we have the following three cases:
\begin{itemize}
\item \underline{Case 1:}  $\omega_s \neq j_*$ for all $s\in [n-1]$. Here $x_{j_*}$  is removed from the nodal set, whence $\iota=\mu$ in \eqref{eq:varpiiota} and 
\begin{equation}\label{eq:PcalBound1}
\epsilon\mathcal{P}_{(\omega_1,\dots,\omega_{n-1})} \leq \epsilon\left(\eta T \right)^{2\varpi_{\mu}} \left(\tau \delta \right)^{(n-1)(n-2)-2\varpi_{\mu}}\left(\rho_*+\eta T \right)^{n-\ell_{\mu}} \left(\rho_*+\tau \delta \right)^{\ell_{\mu}-1}   
\end{equation}
\item \underline{Case 2}: $\omega_s\neq b$ for some $b\in \mathcal{C}_{\mu}\setminus \left\{j_* \right\}$, whence $\iota=\mu$ in \eqref{eq:varpiiota} and 
\begin{equation}\label{eq:PcalBound2}
\epsilon\mathcal{P}_{(\omega_1,\dots,\omega_{n-1})} \leq \epsilon\left(\eta T \right)^{2\varpi_{\mu}} \left(\tau \delta \right)^{(n-1)(n-2)-2\varpi_{\mu}}  \left(\rho_*+\eta T \right)^{n-\ell_{\mu}}\rho_* \left(\rho_*+\tau \delta \right)^{\ell_{\mu}-2}
\end{equation}
\item \underline{Case 3:} $\omega_s\neq b$ for some $b\in \mathcal{C}_{\iota},\ \iota\neq \mu$. In this case we have $\iota\neq \mu$ in \eqref{eq:varpiiota} and
\begin{equation}\label{eq:PcalBound3}
\epsilon\mathcal{P}_{(\omega_1,\dots,\omega_{n-1})} \leq \epsilon\left(\eta T \right)^{2\varpi_{\iota}} \left(\tau \delta \right)^{(n-1)(n-2)-2\varpi_{\iota}} \left(\rho_*+\eta T \right)^{n-\ell_{\mu}-1}\rho_* \left(\rho_*+\tau \delta \right)^{\ell_{\mu}-1}
\end{equation}
\end{itemize}

Now, recall \eqref{eq:xiDef11} and \eqref{eq:varpiiota} and note that the following relation holds
\begin{equation}\label{eq:PowerId}
(n-1)(n-2)-2\varpi_{\iota} =n(n-1) - 2\varrho - 2(\ell_{\iota}-1).
\end{equation}
Now, we consider the three cases \eqref{eq:PcalBound1}-\eqref{eq:PcalBound3} separately:
\begin{itemize}
\item \underline{Case 1:} 
\begin{equation*}
\begin{array}{lll}
\frac{\epsilon\mathcal{P}_{(\omega_1,\dots,\omega_{n-1})}}{\bar{\epsilon}\delta^{n(n-1)+2\ell_*-\ell_\mu-2\varrho}} &\leq \left(\eta T \right)^{2\varpi_{\mu}}\tau^{(n-1)(n-2)-2\varpi_{\mu}}\left(\rho_*+\eta T \right)^{n-\ell_{\mu}}\frac{\delta^{2\ell_*-1} \delta^{(n-1)(n-2)-2\varpi_{\mu}} \left(\rho_*+\tau \delta \right)^{\ell_{\mu}-1}}{\delta^{n(n-1)+2\ell_*-\ell_\mu-2\varrho}} \\
&\overset{\eqref{eq:PowerId}}{\lessapprox} \delta^{2\ell_{\mu}-2-2(\ell_{\mu}-1)}=1.
\end{array}
\end{equation*}
\item \underline{Case 2:}
\begin{equation*}
\begin{array}{lll}
\frac{\epsilon\mathcal{P}_{(\omega_1,\dots,\omega_{n-1})}}{\bar{\epsilon}\delta^{n(n-1)+2\ell_*-\ell_\mu-2\varrho}} &\leq \left(\eta T \right)^{2\varpi_{\mu}}\tau^{(n-1)(n-2)-2\varpi_{\mu}}\left(\rho_*+\eta T \right)^{n-\ell_{\mu}}\frac{\delta^{2\ell_*-1} \delta^{(n-1)(n-2)-2\varpi_{\mu}} \rho_*\left(\rho_*+\tau \delta \right)^{\ell_{\mu}-2}}{\delta^{n(n-1)+2\ell_*-\ell_\mu-2\varrho}} \\
&\overset{\eqref{eq:PowerId}}{\lessapprox}\bar{\rho}_*\delta^{2(\ell_*-\ell_{\mu})}.
\end{array}
\end{equation*}
\item \underline{Case 3:}
\begin{equation*}
\begin{array}{lll}
\frac{\epsilon\mathcal{P}_{(\omega_1,\dots,\omega_{n-1})}}{\bar{\epsilon}\delta^{n(n-1)+2\ell_*-\ell_\mu-2\varrho}} &\leq \left(\eta T \right)^{2\varpi_{\iota}}\tau^{(n-1)(n-2)-2\varpi_{\iota}}\left(\rho_*+\eta T \right)^{n-\ell_{\mu}-1}\frac{\delta^{2\ell_*-1} \delta^{(n-1)(n-2)-2\varpi_{\iota}} \rho_*\left(\rho_*+\tau \delta \right)^{\ell_{\mu}-1}}{\delta^{n(n-1)+2\ell_*-\ell_\mu-2\varrho}} \\
&\overset{\eqref{eq:PowerId}}{\lessapprox}\bar{\rho}_*\delta^{2(\ell_*-\ell_{\iota})+1}.
\end{array}
\end{equation*}
\end{itemize}
This concludes the proof.

\section{Proofs - Section~\ref{Sec:PfAmplitude}}

\subsection{Proof of Proposition \ref{eq:Clust1Sum}}\label{eq:Clust1SumPf}

Recall \eqref{eq:PertLagr} and \eqref{eq:VDMEval} and consider the product $\prod_{b\neq j}\frac{1}{|\tilde{x}_j-\tilde{x}_b|}$. Given $b\neq j, \ b\in \mathcal{C}_s$, we have
\begin{align*}
    |\tilde{x}_j-\tilde{x}_b| &\geq \left|x_j-x_b \right| - \rho_t-\rho_s\geq \begin{cases}
     T -\rho_t - \rho_s,\quad s\neq t\\
     \delta -2\rho_t,\quad s = t
    \end{cases}\geq\begin{cases}
     \left(1-\frac{2}{3\tau} \right)T, \quad s\neq t\\
     \frac{\delta}{3}, \quad s=t
    \end{cases}.
\end{align*}     
Therefore,
\begin{align*}
    \prod_{b\neq j}\frac{1}{\left|\tilde{x}_j - \tilde{x}_b \right|} \leq  \left(\frac{3}{\delta} \right)^{\ell_t-1}\left[ \left(1-\frac{2}{3\tau} \right)T\right]^{-n+\ell_t-1}\lessapprox \frac{1}{\delta^{\ell_t-1}}
\end{align*}
implying, together with \eqref{eq:PertLagr} and \eqref{eq:VDMEval}, that for any $a\in \left\{0,\dots,n-1 \right\}$
\begin{align*}
    |\tilde{\ell}_{j,a}|\lessapprox \prod_{b\neq j}\frac{1}{|\tilde{x}_j-\tilde{x}_b|}\lessapprox \frac{1}{\delta^{\ell_t-1}}.
\end{align*}
 Since $\left\|b_e \right\|_{\infty}<\epsilon$, the latter estimate gives the desired result.
\begin{remark}
Note that $\tilde{\ell}_{j,a}$ is of the form $\tilde{\ell}_{j,a} = \left(\prod_{b\neq j}\frac{1}{\tilde{x}_j-\tilde{x}_b}\right)\mathfrak{s}_{j,a}(\tilde{x}_1,\dots, \tilde{x}_{n})$, where $\mathfrak{s}_{j,a}(\tilde{x}_1,\dots, \tilde{x}_{n})$ is some symmetric polynomial in $\tilde{x}_1,\dots,\tilde{x}_n$. Since the perturbed nodes are bounded in a closed neighborhood of $\mathbb{S}^1$, {\color{blue} as follows from Theorem \ref{thm:node-accuracy-Dima}}, we have $\mathfrak{s}_{j,a}(\tilde{x}_1,\dots, \tilde{x}_{n}) = O(1)$ for all $j\in [n]$ and $a\in [n-1]$.
\end{remark}

\subsection{Proof of Proposition \ref{eq:Clust2Sum}}\label{eq:Clust2SumPf}
Consider first the case $s=j$. We have 
\begin{align}
\left[\tilde{V}^{-1}V\right]_{j,j}  &= \prod_{b\in [n]\setminus \left\{j\right\}}\left[\frac{x_j-\tilde{x}_b}{\tilde{x}_j-\tilde{x}_b} \right] = \prod_{b\in [n]\setminus \left\{j\right\}}\left[1+\frac{x_j-\tilde{x}_j}{\tilde{x}_j-\tilde{x}_b} \right]\nonumber \\
& =\left(\prod_{b\in \mathcal{C}_t\setminus \left\{j\right\}}\left[1+\frac{x_j-\tilde{x}_j}{\tilde{x}_j-\tilde{x}_b} \right]\right)\left(\prod_{b\notin \mathcal{C}_t}\left[1+\frac{x_j-\tilde{x}_j}{\tilde{x}_j-\tilde{x}_b} \right]\right)\nonumber \\
& = \left(1+ \sum_{\nu=1}^{\ell_t-1} \left[x_j-\tilde{x}_j\right]^{\nu}\lambda_{\nu}^{(1),j}\right)\left(1+ \sum_{\nu=1}^{n-\ell_t} \left[x_j-\tilde{x}_j\right]^{\nu}\lambda_{\nu}^{(2),j}\right)\label{eq:seqj}
\end{align} 
where
\begin{align*}
&\lambda_{\nu}^{(1),j} = \sum_{B\subseteq \mathcal{C}_t\setminus \left\{j \right\}, \ |B|=\nu}\prod_{b\in B}\frac{1}{|\tilde{x}_j-\tilde{x}_b|} \lessapprox \frac{1}{\delta^{\nu}},\quad  \lambda_{\nu}^{(2),j} = \sum_{B\subseteq [n]\setminus \mathcal{C}_t, \ |B|=\nu}\prod_{b\in B}\frac{1}{|\tilde{x}_j-\tilde{x}_b|} = O(1). 
\end{align*}
Note that {\color{blue} $\rho_j = \bar{\rho}_j\delta^{2(\ell_*-\ell_t)+1}$ implies $\frac{\rho_j}{\delta}\leq \bar{\rho}_j<\frac{1}{3}$, whence}
\begin{align*}
  &  \left|\sum_{\nu=1}^{\ell_t-1} \left[x_j-\tilde{x}_j\right]^{\nu}\lambda_{\nu}^{(1),j}\right| \lessapprox \sum_{\nu=1}^{\ell_t-1}\left(\frac{\rho_j}{\delta} \right)^{\nu} \lessapprox \frac{\rho_j}{\delta},\\
  &  {\color{blue} \left|\sum_{\nu=1}^{n-\ell_t} \left[x_j-\tilde{x}_j\right]^{\nu}\lambda_{\nu}^{(2),j}\right| \lessapprox \sum_{\nu=1}^{n-\ell_t}\rho_j^{\nu} \lessapprox \rho_j},
\end{align*}
which, together with \eqref{eq:seqj}, gives
\begin{align*}
\left|\left[\tilde{V}^{-1}V\right]_{j,j} - 1 \right|\lessapprox \rho_j+\frac{\rho_j}{\delta}\lessapprox \frac{\epsilon}{\delta^{2\ell_t-1}}.
\end{align*}
Now, let $s\neq j, \ s\in \mathcal{C}_t$. Then,
\begin{align*}
    &\left|\left[\tilde{V}^{-1}V\right]_{j,s} \right|  = \prod_{b\in [n]\setminus \left\{j\right\}}\left|\frac{x_s-\tilde{x}_b}{\tilde{x}_j-\tilde{x}_b} \right| \lessapprox \frac{1}{\delta^{\ell_t-1}} \left|x_s-\tilde{x}_s \right| \prod_{b\in \mathcal{C}_t\setminus \left\{s,j\right\}}\left|x_s-\tilde{x}_b \right|\\
    &\lessapprox \frac{1}{\delta^{\ell_t-1}} \rho_s \left(\prod_{b\in \mathcal{C}\setminus \left\{s,j \right\}}\left(\tau \delta +\rho_b\right) \right)\lessapprox \frac{\rho_s}{\delta}\lessapprox \frac{\epsilon}{\delta^{2\ell_t-1}} .
\end{align*} 
{\color{blue} In the last steps we used the fact that given any $b\in \mathcal{C}_t\setminus \left\{s,j \right\}$ we have that $\rho_b = \bar{\rho}_b\delta^{2(\ell_*-\ell_t)+1}$, which allows to extract $\delta$ from each of the terms in the product and that $s\in \mathcal{C}_t$, whence $\rho_s\lessapprox \frac{\epsilon}{\delta^{2\ell_t-2}}$.}

Finally, let $s\in \mathcal{C}_a, \ a\neq t$. By similar arguments, we obtain
\begin{align*}
&\left|\left[\tilde{V}^{-1}V\right]_{j,s} \right|  = \prod_{b\in [n]\setminus \left\{j\right\}}\left|\frac{x_s-\tilde{x}_b}{\tilde{x}_j-\tilde{x}_b} \right|  \lessapprox \frac{1}{\delta^{\ell_t-1}} \left|x_s-\tilde{x}_s \right| \prod_{b\in \mathcal{C}_a\setminus \left\{s\right\}}\left|x_s-\tilde{x}_b \right|\\
&\lessapprox \frac{1}{\delta^{\ell_t-1}} \rho_s \left(\prod_{b\in \mathcal{C}_a \setminus \left\{ s\right\}}  \left(\tau \delta +\rho_b \right) \right)\lessapprox \frac{1}{\delta^{\ell_t-\ell_a}}\rho_s\lessapprox \frac{\epsilon}{\delta^{\ell_a+\ell_t-2}}.
\end{align*}

\begin{remark}\label{Rem:SingClus}
For simplicity, the analysis above assumed $\ell_s\geq2$ for all $s\in [\zeta]$. For the case of an isolated node $\mathcal{C}_1 = \left\{ x_1\right\},\ \ell_1 = 1$ (without loss of generality), the analysis is modified slightly as follows: First, \eqref{eq:seqj} is written as 
\begin{align*}
\left[\tilde{V}^{-1}V\right]_{1,1}  &= \prod_{k\geq 2}\left[1+\frac{x_1-\tilde{x}_1}{\tilde{x}_1-\tilde{x}_b} \right] = 1+\sum_{\nu=1}^{n-1}\left[x_1-\tilde{x}_1 \right]^{\nu} \lambda_{\nu}^{(2),1}
\end{align*}
with $\lambda_{\nu}^{(2),1}$ given above. Hence,
\begin{align*}
\left|\left[\tilde{V}^{-1}V\right]_{1,1}-1 \right|\lessapprox \rho_1\lessapprox \epsilon.
\end{align*}
Second, for $s\neq 1, \ s\in \mathcal{C}_a$, we further have \begin{align*}
\left|\left[\tilde{V}^{-1}V\right]_{1,s} \right| &=  \prod_{k\geq 2}\left|\frac{x_s-\tilde{x}_k}{\tilde{x}_1-\tilde{x}_k} \right|\lessapprox \left|x_s-\tilde{x}_s \right| \prod_{k\in \mathcal{C}_a\setminus \left\{s\right\}}\left|x_s-\tilde{x}_k \right|\\
&\lessapprox \rho_s  \left(\rho_s+\tau \delta \right)^{\ell_a-1}\lessapprox\delta^{\ell_a-1}\rho_s\lessapprox \frac{\epsilon}{\delta^{\ell_a-1}}.
\end{align*}
\end{remark}


\end{document}